\documentclass[a4paper,12pt,oneside]{amsart}
\usepackage{charter}
\usepackage{mathrsfs}
\usepackage{appendix}
\usepackage{amssymb} 
\usepackage{fancyhdr}
\usepackage{typearea}
\usepackage[a4paper,width=16cm,top=2cm,bottom=2cm,footskip=1cm
]{geometry}

\usepackage{xcolor}
\usepackage
{hyperref}
\hypersetup{
    colorlinks = 
    true,
    linkcolor={magenta},
    citecolor={magenta},
     }

\makeatletter
      \def\@setcopyright{}
      \def\serieslogo@{}
      \makeatother

\newcommand{\Complex}{\mathbb C}
\newcommand{\C}{\mathbb C}
\newcommand{\Real}{\mathbb R}
\newcommand{\N}{\mathbb N}
\newcommand{\ddbar}{\overline\partial}
\newcommand{\pr}{\partial}
\newcommand{\ol}{\overline}
\newcommand{\Td}{\widetilde}
\newcommand{\norm}[1]{\left\Vert#1\right\Vert}
\newcommand{\abs}[1]{\left\vert#1\right\vert}

\newcommand{\set}[1]{\left\{#1\right\}}
\newcommand{\To}{\rightarrow}

 \def\cC{\mathscr{C}}

\theoremstyle{plain}
\newtheorem{thm}{Theorem}[section]
\newtheorem{cor}[thm]{Corollary}

\newtheorem{lem}[thm]{Lemma}

\newtheorem{ass}[thm]{Assumption}
\theoremstyle{definition}
\newtheorem{defn}[thm]{Definition}

\theoremstyle{remark}
\newtheorem{rem}[thm]{Remark}

\numberwithin{equation}{section}

\begin{document}
\title[]{On the singularities of the Bergman projections for 
lower energy forms on complex manifolds with boundary}
\author[]{Chin-Yu Hsiao}
\address{Institute of Mathematics, Academia Sinica, 6F, 
Astronomy-Mathematics Building, No.1, Sec.4, Roosevelt Road, Taipei 10617, Taiwan}
\email{chsiao@math.sinica.edu.tw or chinyu.hsiao@gmail.com} 
\author[]{George Marinescu}
\address{Universit{\"a}t zu K{\"o}ln,  Mathematisches Institut,
    Weyertal 86-90,   50931 K{\"o}ln, Germany\\
    \& Institute of Mathematics `Simion Stoilow', Romanian Academy,
Bucharest, Romania}
\email{gmarines@math.uni-koeln.de}

\begin{abstract}
Let $M$ be a complex manifold of dimension $n$ with smooth boundary $X$. 
Given $q\in\set{0,1,\ldots,n-1}$, let $\Box^{(q)}$ be the $\ddbar$-Neumann Laplacian 
for $(0,q)$ forms. We show that the spectral kernel of $\Box^{(q)}$ 
admits a full asymptotic expansion near the non-degenerate part of the boundary $X$ 
and the  Bergman projection admits an asymptotic expansion under some 
local closed range condition. As applications, we establish 
Bergman kernel asymptotic expansions for some domains with 
weakly pseudoconvex boundary and $S^1$-equivariant 
Bergman kernel asymptotic expansions and embedding theorems 
for domains with holomorphic $S^1$-action. 
\end{abstract}

\maketitle \tableofcontents

\section{Introduction and statement of the main results} \label{s-intro}

Let $M$ be a relatively compact open subset with smooth 
boundary $X$ of a complex manifold $M'$ of complex dimension $n$. 
The study of the $\ddbar$-Neumann Laplacian on $M$ is a classical 
subject in several complex variables. For $q\in\{0,1,\ldots,n-1\}$, 
let $\Box^{(q)}$ be the $\ddbar$-Neumann Laplacian for 
$(0,q)$-forms on $M$. The Levi form of $X$ 
is said to satisfy condition $Z(q)$ at $p\in X$ if it has at least 
$n-q$ positive or at least $q+1$
negative eigenvalues. When condition $Z(q)$ holds at each point 
of $X$, Kohn's $L^2$ estimates give the
hypoellipicity with loss of one derivative for the solutions
of $\Box^{(q)}u=f$, that is, $\ker\Box^{(q)}$ is a finite dimensional
subspace of $\Omega^{0,q}(\overline{M})$ and 
for each $(0,q)$-form $f$ orthogonal to $\ker\Box^{(q)}$
with derivatives of order $\leq s$ in $L^2$ 
the equation $\Box^{(q)}u=f$ has a solution $u$
with derivatives of order $\leq s+1$ in $L^2$
(see \cite{CS01,FK72,Hor65,Ko63-64,KN65}). 

The Bergman projection $B^{(q)}$ is the orthogonal projection onto the 
kernel of $\Box^{(q)}$ in the $L^2$ space. The Schwartz kernel 
$B^{(q)}(\cdot,\cdot)$ of $B^{(q)}$ is called the Bergman kernel.
If $Z(q)$ holds, the above results show that 
the Bergman projection $B^{(q)}$ is a smoothing operator on $\ol M$
and $B^{(q)}(\cdot,\cdot)$ is a smooth on $\ol{M}\times\ol{M}$. 
When $Z(q)$ fails at some point of $X$, the study of the 
boundary behavior of the Bergman kernel $B^{(q)}(\cdot,\cdot)$
is a very interesting problem.

The case when $q=0$ and the Levi form is positive definite on $X$ (so $Z(0)$ fails) 
is especially a classical subject with a rich history.
After the seminal paper of Bergman \cite{Berg32},
H\"ormander \cite[Theorem 3.5.1]{Hor65} (see also \cite{Die73})
determined the limit of $B^{(0)}(x,x)$ when 
$x$ approaches a strictly pseudoconvex point of the boundary
of a domain for which the maximal $\ddbar$ operator acting on functions
has a closed range. More precisely, if $\rho$ is a defining function of $M$,
then $(-\rho(x))^{n+1}B^{(0)}(x,x)\to c\det\mathcal{L}_{\rho}(x_0)$, as
$x\to x_0$, where $x_0\in X$ is a point where the Levi form $\mathcal{L}_{\rho}(x_0)$
is positive definite and $c>0$ is a universal constant.
There are many extensions and variations of H\"ormander's asymptotics for weakly
pseudoconvex or hyperconvex domains, see e.\,g.\ 
\cite{Ba92,BoSt95,Cat89,Chen17,NRSW89,Oh84} and references therein.

The existence of the complete asymptotic expansion $B^{(0)}(x,x)$ at
the boundary was obtained by 
Fefferman~\cite{Fer74} on the diagonal, namely, there are
functions $a,b\in\cC^\infty(\ol{M})$ such that
\begin{equation}\label{feff}
B^{(0)}(x,x)=a(x)(-\rho(x))^{-(n+1)}+b(x)\log(-\rho(x))
\end{equation} 
in $M$.
Subsequently, Boutet de Monvel-Sj\"{o}strand~\cite{BS76} proved the 
off-diagonal asymptotics of $B^{(0)}(x,y)$ in complete generality 
(cf.\ \eqref{e-gue190708ycd}, \eqref{bdms}).

If $q=n-1$ and the Levi form is negative definite (so $Z(n-1)$ fails), 
H\"{o}rmander~\cite[Theorem 4.6]{Hor04} obtained the corresponding 
asymptotics for the Bergman projection for $(0,n-1)$-forms 
in the distribution sense. For general $q>0$, the first author showed 
in~\cite{Hsiao08} that if $Z(q)$ fails, the Levi form is non-degenerate on $X$ 
and $\Box^{(q)}$ has $L^2$ closed range, the singularities of 
the Bergman projection for $(0,q)$-forms admits a full asymptotic expansion. 

In the developments about the Bergman projection mentioned above 
one assumes that the Levi form is non-degenerate on $X$. 
When the Levi form is degenerate on some part of $X$ there are fewer results. 
Fix a point $p\in X$. Suppose that $Z(q)$ fails at $p$ and the Levi form is 
non-degenerate near $p$ (the Levi form can be degenerate away $p$). 
In this work, we show that the spectral kernel of $\Box^{(q)}$ admits a 
full asymptotic expansion near $p$ and the Bergman projection for $(0,q)$-forms 
admits an asymptotic expansion near $p$ under certain closed range condition. 
Our results are natural generalizations of the asymptotics of the Bergman kernel
for strictly pseudoconvex domains by Fefferman \cite{Fer74}
and Boutet~de Monvel and Sj{\"o}strand \cite{BS76}
and they are conjectured by H\"ormander \cite[p.\ 1306]{Hor04}.


Another motivation to study the spectral kernel of $\Box^{(q)}$ 
comes from geometric quantization. 
An important question in the presence of a Lie group $G$ acting on $M'$
is ``quantization commutes with reduction" \cite{GS:82}, see \cite{Ma10} 
for a survey.
The study of $G$-invariant Bergman projection plays an important role 
in geometric quantization. 
If we consider a manifold with boundary as above, 
the $\ddbar$-Neumann Laplacian may not have 
$L^2$ closed range but the $G$-invariant $\ddbar$-Neumann Laplacian has $L^2$ closed range. 
In these cases, we can use the 
asymptotic expansion for the spectral kernel of $\Box^{(q)}$ to study 
$G$-invariant Bergman projection. Therefore, our results about spectral 
kernels for the $\ddbar$-Neumann Laplacian could have applications 
in geometric quantization on complex manifolds with boundary. 
In~\cite{HMM19}, we used the asymptotic expansions of the spectral kernels for the 
Kohn Laplacian to study the geometric quantization on CR manifolds. 

We now formulate the main results. 
We refer to Section~\ref{s:prelim} for some notations and terminology used here. 
Let $M$ be a relatively compact open subset with $\cC^\infty$ boundary $X$ of a
complex manifold $M'$ of dimension $n$. 
We fix a Hermitian metric $\langle\,\cdot\,|\,\cdot\,\rangle$ on 
$\Complex TM'$ so that $T^{1,0}M'\perp T^{0,1}M'$. 
The Hermitian metric $\langle\,\cdot\,|\,\cdot\,\rangle$ on $\C TM'$ 
induces by duality, a Hermitian metric $\langle\,\cdot\,|\,\cdot\,\rangle$ on 
$\oplus^{p,q=n}_{p,q=1}T^{*p,q}M'$. 
Let $\rho\in\cC^\infty(M',\Real)$ be a defining function of $X$,
that is, $\rho=0$ on $X$, $\rho<0$ on $M$ and $d\rho\neq0$ near $X$. 
From now on, we take a defining function $\rho$ so that 
$\norm{d\rho}=1$ on $X$. Let $dv_{M'}$ be the volume form on $M'$ 
induced by the Hermitian metric $\langle\,\cdot\,|\,\cdot\,\rangle$ on 
$\Complex TM'$ and let $(\,\cdot\,|\,\cdot\,)_M$  be the inner product on 
$\Omega^{0,q}(\ol M)$ induced by $\langle\,\cdot\,|\,\cdot\,\rangle$ 
(see \eqref{e-gue190312}). 
Let $L^2_{(0,q)}(M)$ be the completion of $\Omega^{0,q}(\ol M)$ 
with respect to $(\,\cdot\,|\,\cdot\,)_M$. We extend $(\,\cdot\,|\,\cdot\,)_M$ 
to $L^2_{(0,q)}(M)$ in the standard way. 

Given $q\in\{0,1,\ldots,n-1\}$, let 
\[\Box^{(q)}: {\rm Dom\,}\Box^{(q)}\subset L^2_{(0,q)}(M)\To L^2_{(0,q)}(M)\]
be the  $\ddbar$-Neumann Laplacian on $(0, q)$ forms 
(see \eqref{e-gue190312syd}). The operator $\Box^{(q)}$ is a non-negative self-adjoint operator. 
For a Borel set $B\subset\Real$ we denote by $E(B)$ the spectral projection of $\Box^{(q)}$ 
corresponding to the set $B$, where $E$ is the spectral measure of $\Box^{(q)}$.
For $\lambda\geq0$ we consider the spectral spaces of $\Box^{(q)}$,
\begin{equation} \label{e-gue190708yyd}
H^q_{\leq\lambda}(\ol M):={\rm Ran\,}E\bigr((-\infty,\lambda]\bigr)\subset L^2_{(0,q)}(M).
\end{equation}
For $\lambda=0$ we obtain the space of harmonic forms 
$H^q(\ol M):=H^q_{\leq0}(\ol M)={\rm Ker\,}\Box^{(q)}$. 
For $\lambda\geq0$, let
\begin{equation}\label{e-gue190708yydI}
B^{(q)}_{\leq\lambda}:L^2_{(0,q)}(M)\To H^q_{\leq\lambda}(\ol M)
\end{equation}
be the orthogonal projection with respect to the $L^2$ inner product $(\,\cdot\,|\,\cdot\,)_M$ 
and let
$$B^{(q)}_{\leq\lambda}(x,y)\in\mathscr D'(M\times M,T^{*0,q}M\boxtimes(T^{*0,q}M)^*)$$
denote the distribution kernel of $B^{(q)}_{\leq\lambda}$. For $\lambda=0$ we obtain the
\emph{Bergman projection} 
$B^{(q)}:=B^{(q)}_{\leq0}$ and the \emph{Bergman kernel} 
$B^{(q)}(x,y):=B^{(q)}_{\leq0}(x,y)$. 

The boundary $X$ is a CR manifold of real dimension $2n-1$
 with natural CR structure $T^{1,0}X:=T^{1,0}M'\cap\Complex TX$ 
 (see the discussion after \eqref{e-gue190312yyd}). 
 The Levi form of $X$ is given by \eqref{e-gue190312sds}. 
 Let $U$ be an open set of $M'$ with $U\cap X\neq\emptyset$. 
Let $A$ and $B$ be $\cC^\infty$ vector bundles over $M'$ and let 
$F_1, F_2: \cC^\infty_0(U\cap\ol M,A)\To\mathscr D'(U\cap\ol M,B)$ 
be continuous operators. Let 
$F_1(x,y), F_2(x,y)\in\mathscr D'((U\times U)\cap(\ol M\times\ol M), A\boxtimes B^*)$ 
be the distribution kernels of $F_1$ and $F_2$, respectively. We write 
$$F_1\equiv F_2\:\:\text{or} 
\:\: F_1(x,y)\equiv F_2(x,y)\mod\cC^\infty((U\times U)\cap(\ol M\times\ol M))$$ 
if $F_1(x,y)=F_2(x,y)+r(x,y)$, where 
$r(x,y)\in\cC^\infty((U\times U)\cap(\ol M\times\ol M),A\boxtimes B^*)$. 
Let $$S^{n}_{1, 0}((U\times U)\cap(\ol M\times\ol M)\times
]0, \infty[\,,T^{*0,q}M'\boxtimes(T^{*0,q}M')^*)$$ 
denote the H\"ormander symbol space on $(U\times U)\cap(\ol M\times\ol M)\times]0, \infty[$ 
(see the discussion after \eqref{e-gue190326scd} and \eqref{e-gue190529ycdI}).
Our first main result is the following. 

\begin{thm}\label{t-gue190708yyd}
Let $M$ be a relatively open subset with smooth boundary $X$ of a complex manifold $M'$ 
of complex dimension $n$. Let $U$ be an open set of $M'$ with $U\cap X\neq\emptyset$. 
Suppose that the Levi form is non-degenerate of constant signature $(n_-, n_+)$ on $U\cap X$, 
where $n_-$ denotes the number of negative eigenvalues of the Levi form on $U\cap X$. 
Fix $\lambda>0$. If $q\neq n_-$ then 
\begin{equation}\label{e-gue190810yyd}
B^{(q)}_{\leq\lambda}(x,y)\equiv0\mod\cC^\infty((U\times U)\cap(\ol M\times\ol M)).
\end{equation}
For $q=n_-$ we have 
\begin{equation}\label{e-gue190708ycd}
B^{(q)}_{\leq\lambda}(x,y)\equiv\int^\infty_0e^{i\phi(x, y)t}b(x, y, t)dt
\mod\cC^\infty((U\times U)\cap(\ol M\times\ol M)),
\end{equation}
where
\[
\begin{split}
&b(x, y, t)\in S^{n}_{1, 0}((U\times U)\cap(\ol M\times\ol M)\times
]0, \infty[,T^{*0,q}M'\boxtimes(T^{*0,q}M')^*),\\
b(x, y, t)\sim&\sum^\infty_{j=0}b_j(x, y)t^{n-j}\:\: 
\text{in 
$S^{n}_{1, 0}((U\times U)\cap(\ol M\times\ol M)\times
]0, \infty[,T^{*0,q}M'\boxtimes(T^{*0,q}M')^*)$},\\
\end{split}
\]
and $b_0(x,x)$ is given by \eqref{e-gue190531yyda} below. Moreover, 
\begin{equation}\label{e-gue190708ycdII}
\begin{split}
&\phi(x, y)\in C^\infty((U\times U)\cap(\ol M\times\ol M)),\ \ {\rm Im\,}\phi\geq0, \\
&\phi(x, x)=0,\ \ x\in U\cap X,\ \ \phi(x, y)\neq0\ \ \mbox{if}\ \ (x, y)\notin{\rm diag\,}((U\times U)\cap(X\times X)), \\
&{\rm Im\,}\phi(x, y)>0\ \ \mbox{if}\ \ (x, y)\notin(U\times U)\cap(X\times X), \\
&\phi(x, y)=-\ol{\phi(y, x)},\\
&\mbox{$d_x\phi(x,x)=-\omega_0(x)-id\rho(x)$, for every $x\in U\cap X$},
\end{split}\end{equation}
where 
$\omega_0(x)\in\cC^\infty(X, T^*X)$ is the global Reeb one form 
on $X$ given by \eqref{e-gue190312scdqII}. 
\end{thm}
We refer to Remark~\ref{r-gue190531yyd} 
for the precise meaning of the oscillatory integral 
in \eqref{e-gue190708ycd}
and to Theorem~\ref{t-gue190531yyda} 
for more properties for the phase $\phi$. 
The phase function $\varphi_-(x,y):=\phi(x,y)|_{X\times X}$
is the same as the phase function 
appearing in the description of the singularities of the Szeg\H{o} kernels 
for lower energy forms in~\cite[Theorems 3.3, 3.4]{HM17}. 

The complex Fourier integral operator 
$\int^\infty_0e^{i\phi(z, w)t}b(z, w, t)dt$ in \eqref{e-gue190708ycd} 
can be taken to be independent of $\lambda$. 
Hence for every $\lambda_1>\lambda>0$, 
$B^{(q)}_{\leq\lambda_1}(x,y)\equiv B^{(q)}_{\leq\lambda}(x,y)$
modulo $\cC^\infty((U\times U)\cap(\ol M\times\ol M)).$ 

By integrating over $t$ of the oscillatory integral 
$\int^\infty_0e^{i\phi(x, y)t}b(x, y, t)dt$ in \eqref{e-gue190708ycd}, 
we have the following corollary of Theorem~\ref{t-gue190708yyd}.

\begin{cor} \label{c-gue190709yyd}
Let $M$ be a relatively compact subset with smooth boundary $X$ 
of a complex manifold $M'$ of complex dimension $n$. 
Let $U$ be an open set of $M'$ with $U\cap X\neq\emptyset$. 
Suppose that the Levi form is non-degenerate of constant signature $(n_-, n_+)$ on $U\cap X$. 
Let $q=n_-$. 
Then there exist smooth functions
$F, G\in\cC^\infty((U\times U)\cap(\ol M\times\ol M),T^{*0,q}M'
\boxtimes(T^{*0,q}M')^*))$ such that for every $\lambda>0$, we have
modulo $\cC^\infty((U\times U)\cap(\ol M\times\ol M))$,
\begin{equation}\label{bdms}
B^{(q)}_{\leq\lambda}(x,y)\equiv 
F(-i(\phi(x, y)+i0))^{-n-1}+G\log(-i(\phi(x, y)+i0)).
\end{equation}
Moreover, we have
\begin{equation} \label{e-gue190709yyd} 
\begin{split}
F(x,y) &=\sum^{n}_{j=0}(n-j)!b_j(x, y)(-i\phi(x, y))^j+f_\lambda(z, w)(\phi(x, y))^{n+1},  \\
G(x,y) &\equiv\sum^\infty_{j=0}\frac{(-1)^{j+1}}{j!}b_{n+j+1}(x, y)(-i\phi(x, y))^j
\mod\cC^\infty((U\times U)\cap(\ol M\times\ol M))
\end{split}\end{equation}
where
$b_j(x, y)$, $j\in\N_0$, and $\phi(x,y)$ are as in Theorem~\ref{t-gue190708yyd} and
$f_\lambda(z, w)$  is a $\lambda$-dependent smooth function
in $\cC^\infty((U\times U)\cap(\ol M\times\ol M),T^{*0,q}M'
\boxtimes(T^{*0,q}M')^*)$.
\end{cor}
\begin{defn}\label{d-gue190609yyd}
Let $U$ be an open set in $M'$ with $U\cap X\neq\emptyset$. 
We say that $\Box^{(q)}$ has local closed range in $U$ if for every open set 
$W\subset U$ with $W\cap X\neq\emptyset$, $\ol W\Subset U$, 
there is a constant $C_W>0$ such that 
\[\norm{(I-B^{(q)})u}_M\leq C_W\norm{\Box^{(q)}u}_M,
\:\:\text{$u\in\Omega^{0,q}_0(W\cap\ol M)\cap{\rm Dom\,}\Box^{(q)}$.}\]
Note that if $\Box^{(q)}$ has closed range then
$\Box^{(q)}$ has local closed range in $U$ for any $U$.
\end{defn}
Our second main result is the following. 

\begin{thm}\label{t-gue190709yyd}
Let $M$ be an open relatively compact subset with smooth boundary $X$ of a complex 
manifold $M'$ of complex dimension $n$. Let $U$ be an open set of $M'$ with 
$U\cap X\neq\emptyset$. 
Suppose that the Levi form is non-degenerate of constant signature $(n_-, n_+)$
on $U\cap X$.
Let $q=n_-$. 
Suppose that $\Box^{(q)}$ has local closed range in $U$. Then 
\begin{equation}\label{e-gue190709ycdh}
B^{(q)}(x,y)\equiv\int^\infty_0e^{i\phi(x, y)t}b(x, y, t)dt
\mod\cC^\infty((U\times U)\cap(\ol M\times\ol M)),
\end{equation}
where $b(x, y, t)\in S^{n}_{1, 0}((U\times U)\cap(\ol M\times\ol M)\times]0, 
\infty[,T^{*0,q}M'\boxtimes(T^{*0,q}M')^*)$ and $\phi(x,y)
\in\cC^\infty((U\times U)\cap(\ol M\times\ol M))$ 
are as in Theorem~\ref{t-gue190708yyd}. 
In particular, $B^{(q)}(x,y)$ has asymptotics as in \eqref{bdms}.
\end{thm}
H\"ormander \cite[Theorem 4.6]{Hor04} determined the asymptotics
of $B^{(n-1)}(x,y)$ in the distributional sense near a boundary
point where the Levi form is negative definite
under the condition that $\Box^{(n-1)}$ has closed range.
Theorem \ref{t-gue190709yyd} thus generalizes this result
and gives the asymptotics in the $\cC^\infty$ sense.
\begin{rem}\label{r-gue190810yyd}
Let $(E,h^E)$ be a Hermitian holomorphic vector bundle over $M'$. 
As in \eqref{e-gue190312syd} below, we can consider the 
$\ddbar$-Neumann Laplacian on $(0, q)$-forms with values in $E$: 
\begin{equation} \label{e-gue1900810yydI}
\Box^{(q)}=\ddbar\,\ddbar^*+\ddbar^*\,\ddbar: 
{\rm Dom\,}\Box^{(q)}\subset L^2_{(0,q)}(M,E)\To L^2_{(0,q)}(M,E),
\end{equation}
where $L^2_{(0,q)}(M,E)$ denotes the $L^2$ space
of $(0,q)$-forms with values in $E$.
We can define 
$B^{(q)}_{\leq\lambda}(x,y)$ in the same way as above and by the same proofs, 
Theorem~\ref{t-gue190708yyd} and Theorem~\ref{t-gue190709yyd} 
hold also in the presence of a vector bundle $E$. 

In particular, we can consider the trivial line bundle $E=\C$ with 
the metric $h^E=e^{-\varphi}$, where $\varphi\in\cC^\infty(M')$
is a weight function. 
In this case the space $L^2_{(0,q)}(M,E)$ is the completion of 
$\Omega^{0,q}(\ol{M})$ with respect to the weighted $L^2$ inner product
$(u|v)_{\varphi}=\int_M\langle u|v\rangle e^{-\varphi} dv_{M'},$
and is denoted by $L^2_{(0,q)}(M,\varphi)$. 
The Bergman projection and kernel are denoted by $B^{(q)}_\varphi$ and
$B^{(q)}_\varphi(\cdot,\cdot)$, respectively.
So all the results above have versions 
for \emph{weighted Bergman kernels} $B^{(q)}_\varphi(\cdot,\cdot)$.
\end{rem} 
We now give some applications of the results above.
\begin{cor}
(i) Let $M$ be a bounded domain of holomorphy in $\C^n$ 
with smooth boundary and let $\varphi$ be any finction in $\cC^\infty(\ol{M})$. 
Let $U$ be an open set in $\C^n$ such that
$U\cap\partial M$ is 
strictly pseudoconvex.
Then the weighted Bergman kernel $B^{(0)}_\varphi(\cdot,\cdot)$
has the asymptotics \eqref{e-gue190709ycdh} and \eqref{bdms}
on $U$. In particular, Fefferman's asymptotics on the diagonal \eqref{feff}
hold for $B^{(0)}_\varphi(x,x)$ on $U$.

(ii) Let $M$ be an open relatively compact domain with smooth boundary
$X$ in a complex manifold $M'$ of dimension $n$. 
Assume that $M$ satisfies condition
$Z(1)$, i.\,e., the Levi form of $X$ has everywhere either $n-1$ positive 
or $2$ negative eigenvalues. 
Let $U$ be an open set in $M'$ such that
$U\cap X$ is strictly pseudoconvex. Then the Bergman kernel
$B^{(0)}(\cdot,\cdot)$
has the asymptotics \eqref{feff}, \eqref{bdms} and \eqref{e-gue190709ycdh}
on $U$. 
\end{cor}
Indeed, it follows from \cite[Theorem 2.2.1']{Hor65} in case (i) and 
\cite[Theorem 3.1.19]{FK72}, 
\cite[Theorem 3.4.1]{Hor65} in case (ii) that $\Box^{(0)}$ has closed range.
Note that these assertions are independent of the choice of the function 
$\varphi\in\cC^\infty(\ol{M})$, since changing $\varphi$ only means introducing 
equivalent norms in the Hilbert spaces concerned.
Obviously, the items (i) and (ii) hold also if we work with
Bergman kernels of holomorphic sections in a Hermitian
holomorphic vector bundle $(E,h^E)$ defined in a neighbourhood of $\ol{M}$
(cf.\ Remark \ref{r-gue190810yyd}).

We consider next Bergman kernels on corona domains.
\begin{cor}
Let $M\Subset\C^n$ be the corona domain $M=M_0\setminus\ol{M}_1$
between two pseudoconvex domains  $M_0$ and $M_1$ 
with smooth boundary and $M_1\Subset M_0$\,, $n\geq2$.
Let $U$ an open set such that $U\cap\partial M_1$ is strictly
pseudoconvex and $U\cap\partial M_0=\emptyset$.
Then the Bergman kernel $B^{(n-1)}(x,y)$ on $(0,n-1)$-forms
has the asymptotics \eqref{e-gue190709ycdh} and \eqref{bdms}.
\end{cor}
By \cite[Theorem 3.5]{Sh10}, the operator $\Box^{(n-1)}$ has closed range
in $L^2$ for a corona domain between two pseudoconvex domains
as above.
Moreover,  the Levi form of $\partial M$ is negative definite
on $U\cap\partial M$, so the Corollary follows from 
Theorem \ref{t-gue190709yyd}.

Let $M'$ be a complex manifold of complex dimension $n$. 
Let $M_0=\set{z\in M';\, \rho(z)<0}$ be a smooth pseudoconvex 
relatively compact open set, 
where $\rho\in\cC^\infty(M',\mathbb R)$
is a defining function with $d\rho\neq0$ on $X_0:=\partial M_0$
whose Levi form $\mathcal{L}$ is positive semidefinite on the holomorphic tangential
space to $X_0$.
Consider the corona domain
\begin{equation}\label{e-gue190905ycd}
M:=\set{z\in M';\, -\varepsilon<\rho(z)<0},
\end{equation}
where $\varepsilon>0$. We assume that 
$d\rho\neq0$ on $X_1:=\set{z\in M';\, \rho(z)=-\varepsilon}$.
Then $M$ is a relatively open subset with smooth boundary 
$X:=X_0\cup X_1$ of $M'$. If $\varepsilon>-\min_{\overline{M}_0}\varphi$
we have $X_1=\emptyset$ and we have $M=M_0$.

Suppose that there is a holomorphic line bundle $(L, h^L)$ 
over $M'$, where $h^L$ denotes 
a Hermitian metric of $L$ and $R^L$ is the curvature of $L$ induced by $h^L$. 
For every $k\in\mathbb N$, let $(L^k, h^{L^k})$ be the $k$-th power of 
$(L,h^L)$. Let $(\,\cdot\,|\,\cdot\,)_k$ be the $L^2$ inner product on 
$\Omega^{0,q}(M,L^k)$ induced by the given Hermitian metric 
$\langle\,\cdot\,|\,\cdot\,\rangle$ on $\Complex TM'$ and 
$h^L$ and let $L^2_{(0,q)}(M,L^k)$ be the completion of $\Omega^{0,q}(M,L^k)$. Let 
\[\Box^{(q)}_k: \ddbar\,\ddbar^*+\ddbar^*\,\ddbar: 
{\rm Dom\,}\Box^{(q)}\subset L^2_{(0,q)}(M,L^k)\To L^2_{(0,q)}(M,L^k)\]
be the $\ddbar$-Neumann operator on $M$  with values in $L^k$ and let 
\[B^{(q)}_k: L^2_{(0,q)}(M,L^k)\To{\rm Ker\,}\Box^{(q)}_k\]
be the orthogonal projection with respect to $(\,\cdot\,|\,\cdot\,)_k$ 
and let $B^{(q)}_k(\cdot,\cdot)\in\mathscr D'(M\times M, L^k\boxtimes(L^k)^*)$ 
be the distribution kernel of $B^{(q)}_k$.
\begin{thm}\label{t-gue190905ycd}
Assume that $(L,h^L)$ is positive in a neighborhood of $X_0$.
Let $U$ be an open set 
of $X_0$ in $M'$ with $U\cap X_1=\emptyset$. There exists $k_0\in\mathbb N$, 
such that for every $k\in\mathbb N$, $k\geq k_0$, $\Box^{(0)}_k$ has local closed range in $U$. 
\end{thm}
From 
Theorem~\ref{t-gue190709yyd}, Remark~\ref{r-gue190810yyd} 
and Theorem~\ref{t-gue190905ycd} we get: 
\begin{thm}\label{t-gue190909yyd}
Let $M_0=\set{z\in M';\, \rho(z)<0}$ be a smooth pseudoconvex domain 
in a complex manifold $M'$. Let $\varepsilon>0$ such that 
the corona domain
$M:=\set{z\in M';\, -\varepsilon<\rho(z)<0}$ is smooth and set $X_0=\{\rho=0\}$,
$X_1=\{\rho=-\varepsilon\}$.
Let $L$ be a holomorphic line bundle on $M'$
which is positive near $X_0$. 
Let $U$ be an open set of $M'$ such that $U\cap X_0$ is strictly
pseudoconvex and $U\cap X_1=\emptyset$. Let $k_0\in\mathbb N$ be as in 
Theorem~\ref{t-gue190905ycd}. 
Then for every $k\in\mathbb N$, $k\geq k_0$
the Bergman kernel of $M$ with values in $L^k$ satisfies
\[B^{(0)}_k(x,y)\equiv\int^{\infty}_0e^{i\phi(x,y)t}b(x,y,t)dt
\mod\cC^\infty((U\times U)\cap(\ol M\times\ol M),L^k\boxtimes(L^k)^*),\]
where $\phi(x,y)\in\cC^\infty((U\times U)\cap(\ol M\times\ol M))$ is as in 
Theorem~\ref{t-gue190708yyd}, 
\[
\begin{split}
&b(x,y,t)\in S^n_{1,0}(((U\times U)\cap(\ol M\times\ol M))\times
]0,+\infty[,L^k\boxtimes(L^k)^*),\\ 
b(x,y,t)\sim&\sum^{+\infty}_{j=0}b_j(x,y)t^{n-j}\:\: \text{in 
$S^n_{1,0}((U\times U)\cap(\ol M\times\ol M))\times]0,+\infty[,
L^k\boxtimes(L^k)^*)$},
\end{split}
\]
with
$b_0(x,x)=\pi^{-n}\abs{\det\,\mathcal{L}_x}|\ddbar\rho(x)|^2$
for every $x\in U\cap X$.
\end{thm}
The next applications concerns the asymptotics of the $S^1$-equivariant Bergman kernel
and embedding theorems.
We assume that $M'$ admits a holomorphic $S^{1}$-action $e^{i\theta}$, 
$\theta\in[0,2\pi[$, $e^{i\theta}: M'\to M'$, $x\in M'\To e^{i\theta}\circ x\in M'$. 
It means that the $S^{1}$-action preserves the complex structure $J$ of $M'$. 
Let $T_0\in\cC^\infty(M',TM')$ be the global real vector field on $M'$ induced 
by $e^{i\theta}$, that is 
$(T_0u)(x)=\frac{\pr}{\pr\theta}u(e^{i\theta}\circ x)\big|_{\theta=0}$ 
for every $u\in\cC^\infty(M')$.  We assume that

\begin{ass}\label{a-gue190812yyd}
$\Complex T_0(x)\oplus T^{1,0}_xX\oplus T^{0,1}_xX=\Complex T_xX$, 
for every $x\in X$, and the $S^{1}$-action preserves the boundary $X$, 
that is, we can find a defining function $\rho\in\cC^\infty(M',\Real)$ of $X$ 
such that $\rho(e^{i\theta}\circ x)=\rho(x)$, for every $x\in M'$ and every 
$\theta\in[0,2\pi]$.
\end{ass}

We take the Hermitian metric $\langle\,\cdot\,|\,\cdot\,\rangle$ on 
$\Complex TM'$ to be $S^1$-invariant and $\langle\,T_0\,|\,T_0\,\rangle=1$ on $X$. 
We take a $S^1$- invariant definite function $\rho$ so that $\norm{d\rho}=1$ on $X$. 
Fix an open connected component $X_0$ of $X$. From the fact that 
$\langle J(d\rho), T_0\rangle$ is always non-zero on $X_0$, we can check that 
\begin{equation}\label{e-gue190812yyd}
\langle\,J(d\rho)\,,\,T_0\,\rangle<0 ~\text{on}~X_0,
\end{equation}
where $J$ is the complex structure map on $T^*M'$. 
From \eqref{e-gue190812yyd} and noting that $\norm{T_0}=\norm{d\rho}=1$ 
on $X$, it is easy to see that 
\begin{equation}\label{e-gue190812yydI}
T_0=T\ \ \mbox{on $X_0$}, 
\end{equation}
where $T$ is given by \eqref{e-gue190312scdqI} below. 
For every $m\in\mathbb{Z}$, put
\begin{equation}\label{e-gue190812yydII}
\Omega^{0,q}_{m}(M')=\{u\in\Omega^{0,q}(M');\, \mathcal L_Tu=imu \}
\end{equation}
where $\mathcal L_Tu$ is the Lie derivative of $u$ along direction $T$. 
For convenience, we write $Tu:=\mathcal L_Tu$. Similarly, 
let $\Omega^{0,q}_m(\ol M)$ denote the space of restrictions to $M$ 
of elements in $\Omega^{0,q}_m(M')$. 
We write $\cC^\infty_m(\ol M):=\Omega^{0,0}_m(\ol M)$. 
Let $L^2_{(0,q),m}(M)$ be the completion of $\Omega^{0,q}_m(\ol M)$ 
with respect to $(\,\cdot\,|\,\cdot\,)_M$. For $q=0$, we write 
$L^2_m(M):=L^2_{(0,0),m}(M)$. 
Fix $\lambda\geq0$ and $m\in\mathbb Z$. Put 
\begin{equation}\label{e-gue190812yydIII}
H^q_{\leq\lambda,m}(\ol M):=
H^q_{\leq\lambda}(\ol M)\cap L^2_{(0,q),m}(M),
\end{equation}
where $H^q_{\leq\lambda}(\ol M)$ is given by \eqref{e-gue190708yyd}. Let 
\begin{equation}\label{e-gue190812ycd}
B^{(q)}_{\leq\lambda,m}: L^2_{(0,q)}(M)\To H^q_{\leq\lambda,m}(\ol M)
\end{equation}
be the orthogonal projection with respect to $(\,\cdot\,|\,\cdot\,)_M$ 
and let 
$$B^{(q)}_{\leq\lambda,m}(x,y)\in\mathscr D'(M\times M, T^{*0,q}M'
\boxtimes(T^{*0,q}M')^*),$$ 
be the distribution kernel of $B^{(q)}_{\leq\lambda,m}$. 
For $\lambda=0$, we write $H^q_{m}(\ol M):=H^q_{\leq0,m}(\ol M)$, 
$B^{(q)}_{m}:=B^{(q)}_{\leq0,m}$, 
$B^{(q)}_{m}(x,y):=B^{(q)}_{\leq0,m}(x,y)$. 
From~\cite[Theorem 3.3]{HHLS}, we see that 
$H^q_{\leq\lambda,m}(\ol M)$ is a finite dimensional subspace of 
$\Omega^{0,q}_m(\ol M)$ and hence 
$$B^{(q)}_{\leq\lambda,m}(x,y)\in\cC^\infty(\ol{M}\times\ol{M}, T^{*0,q}M'
\boxtimes(T^{*0,q}M')^*).$$ 
Moreover, it is straightforward to see that 
\begin{equation}\label{e-gue190812ycdI}
B^{(q)}_{\leq\lambda,m}(x,y)=
\frac{1}{2\pi}\int^{\pi}_{-\pi}B^{(q)}_{\leq\lambda}(x,e^{i\theta}y)
e^{im\theta}d\theta.
\end{equation}  
We have the following  
asymptotic expansion for the $S^1$-equivariant Bergman kernel
(see also Theorem~\ref{t-gue190817yydh}). 

\begin{thm}\label{t-gue190817yydhi}
With the notations and assumptions used above, fix $p\in X_0$ 
and let $U$ be an open set of $p$ in $M'$ with $U\cap X_0\neq\emptyset$. 
Suppose that the Levi form is 
positive $U\cap X_0$. Suppose that $Z(1)$ holds on $X$. 
Let $N_p:=\set{g\in S^1;\, g\circ p=p}=\set{g_0:=e, g_1,\ldots, g_r}$, 
where $e$ denotes the identify element in $S^1$ and $g_j\neq g_\ell$ if $j\neq \ell$, 
for every $j, \ell=0,1,\ldots,r$. We have 
\begin{equation}\label{e-gue190816yydli}
B^{(0)}_{m}(x,y)\equiv\sum^r_{\alpha=0}g^m_\alpha 
e^{im\phi(x,g_\alpha y)}b_{\alpha}(x,y,m)\mod O(m^{-\infty})\ \ 
\mbox{on $U\cap\ol M$}, 
\end{equation}
where for every $\alpha=0,1,\ldots,r$, 
\begin{equation}\label{e-gue190816ycdi}
\begin{split}
&b_\alpha(x, y, m)\in S^{n}_{{\rm loc\,}}((U\times U)\cap(\ol M\times\ol M)),\\
&\mbox{$b_\alpha(x, y, m)\sim\sum^\infty_{j=0}b_{\alpha,j}(x, y)m^{n-j}$ 
in $S^{n}_{{\rm loc\,}}((U\times U)\cap(\ol M\times\ol M))$},\\
&b_{\alpha,0}(x,x)=b_0(x,x),\ \ \mbox{$b_0(x,x)$ is given by \eqref{e-gue190531yyda}}, 
\end{split}
\end{equation}
and $\phi(x, y)\in\cC^\infty((U\times U)\cap(\ol M\times\ol M))$ is as \eqref{e-gue190708ycdII}. 
\end{thm} 

Actually, we have more general results than 
Theorem~\ref{t-gue190817yydhi}. In Theorem~\ref{t-gue190813yyd}, 
we get an asymptotic expansion for $B^{(q)}_{\leq\lambda,m}$ in 
$m$ for every $\lambda>0$, and in Theorem~\ref{t-gue190816yyydh}, 
we get an asymptotic expansion for $B^{(q)}_m$ in $m$ 
under the local closed range condition of $\Box^{(q)}$. 

For every $m\in\mathbb N$, let 
\begin{equation}\label{e-gue190817ycdi}
\Phi_m: \ol M\To\Complex^{d_m},\;\:
x\mapsto (f_1(x),\ldots,f_{d_m}(x)),
\end{equation}
where $\set{f_1(x), \ldots,f_{d_m}(x)}$ is an orthonormal basis for 
$H^0_m(\ol M)$ with respect to $(\,\cdot\,|\,\cdot\,)_M$ and 
$d_m={\rm dim\,}H^0_m(\ol M)$. 
We have the following $S^1$-equivaraint embedding theorem 
(see Theorem~\ref{t-gue190817yydI}).

\begin{thm}\label{t-gue190817yydIi}
With the notations and assumptions used above, 
assume that the Levi form is positive definite on $X_0$ and $Z(1)$ holds on $X$. 
For every $m_0\in\mathbb N$, there exist 
$m_1\in \mathbb N, \ldots, m_k\in\mathbb N$, with 
$m_j\geq m_0$, $j=1,\ldots,k$, and a $S^1$-invariant open set 
$V$ of $X_0$ such that the map 
\begin{equation}\label{e-gue190817ycdIi}
\begin{split}
\Phi_{m_1,\ldots,m_k}: V\cap\ol M&\To\Complex^{\hat d_m},\\
x&\mapsto (\Phi_{m_1}(x),\ldots,\Phi_{m_k}(x)),
\end{split}
\end{equation}
is a holomorphic embedding, where $\Phi_{m_j}$ is given by 
\eqref{e-gue190817ycdi} and $\hat d_m=d_{m_1}+\cdots+d_{m_k}$. 
\end{thm}

Without the $Z(1)$ condition, we can still formulate the following 
$S^1$-equivaraint embedding theorem (see the proof of Theorem~\ref{t-gue190817yydII})

\begin{thm}\label{t-gue190817yydIIi}
With the notations and assumptions used above, assume that the 
Levi form is positive definite on $X_0$.  For every $m_0\in\mathbb N$, 
there exist a $S^1$-invariant open set $V$ of $X_0$ and 
$f_j(x)\in\cC^\infty(V\cap\ol M)$, $j=1,\ldots,K$, with 
$\ddbar f_j=0$ on $V\cap\ol M$, $f_j(e^{i\theta}x)=e^{im_j\theta}f(x)$, 
$m_j\geq m_0$, $j=1,\ldots,K$, for every $e^{i\theta}\in S^1$ and 
every $x\in V$,  such that the map 
\begin{equation}\label{e-gue190817ycdIIi}
\begin{split}
\Phi: V\cap\ol M&\To\Complex^{K},\\
x&\mapsto (f_1(x),\ldots,f_K(x)),
\end{split}
\end{equation}
is a holomorphic embedding.
\end{thm}

The paper is organized as follows. In Section~\ref{s:prelim}, 
we collect some standard notations, terminology, definitions and statements 
we use throughout. To construct parametrices for $\Box^{(q)}$, 
we introduce in Section~\ref{s-gue190312scd} the operator $\Box^{(q)}_-$. 
In Section~\ref{s-gue190531yyd}, we construct parametrices for $\Box^{(q)}$ 
near a point $p\in X$ under the assumption that $Z(q)$ holds at $p$. 
Up to the authors' knowledge, the parametrices construction in 
Section~\ref{s-gue190531yyd} are also new results. 
In Section~\ref{s-gue190321myyd}, 
we obtain microlocal Hodge decomposition theorems for $\Box^{(q)}$ 
near a point $p\in X$ under the assumption that $Z(q)$ fails at $p$. 
By using the results in Section~\ref{s-gue190531yyd} and 
Section~\ref{s-gue190321myyd}, we prove Theorem~\ref{t-gue190708yyd} and 
Theorem~\ref{t-gue190709yyd} in Section~\ref{s-gue190531yyda}. 
In Section~\ref{s-gue190810yyd}, we prove the asymptotic expansions of the
$S^1$-equivariant Bergman kernel and embedding theorems for 
domains with holomorphic $S^1$-action.

\section{Preliminaries}\label{s:prelim}

\subsection{Standard notations} \label{s-ssna}
We shall use the following notations: $\mathbb N=\set{1,2,\ldots}$
is the set of natural numbers, $\mathbb N_0=\mathbb N\cup\set{0}$, $\Real$ 
is the set of real numbers, $\ol\Real_+:=\set{x\in\Real;\, x\geq0}$. 
For a multiindex $\alpha=(\alpha_1,\ldots,\alpha_n)\in\mathbb N_0^n$
we denote by $\abs{\alpha}=\alpha_1+\ldots+\alpha_n$ its norm and by $l(\alpha)=n$ its length.
For $m\in\mathbb N$, write $\alpha\in\set{1,\ldots,m}^n$ if $\alpha_j\in\set{1,\ldots,m}$, 
$j=1,\ldots,n$. $\alpha$ is strictly increasing if $\alpha_1<\alpha_2<\ldots<\alpha_n$. For $x=(x_1,\ldots,x_n)$ we write
\[
\begin{split}
&x^\alpha=x_1^{\alpha_1}\ldots x^{\alpha_n}_n,\\
& \pr_{x_j}=\frac{\pr}{\pr x_j}\,,\quad
\pr^\alpha_x=\pr^{\alpha_1}_{x_1}\ldots\pr^{\alpha_n}_{x_n}=\frac{\pr^{\abs{\alpha}}}{\pr x^\alpha}\,,\\
&D_{x_j}=\frac{1}{i}\pr_{x_j}\,,\quad D^\alpha_x=D^{\alpha_1}_{x_1}\ldots D^{\alpha_n}_{x_n}\,,
\quad D_x=\frac{1}{i}\pr_x\,.
\end{split}
\]
Let $z=(z_1,\ldots,z_n)$, $z_j=x_{2j-1}+ix_{2j}$, $j=1,\ldots,n$, be coordinates of $\Complex^n$.
We write
\[
\begin{split}
&z^\alpha=z_1^{\alpha_1}\ldots z^{\alpha_n}_n\,,\quad\ol z^\alpha=\ol z_1^{\alpha_1}\ldots\ol z^{\alpha_n}_n\,,\\
&\pr_{z_j}=\frac{\pr}{\pr z_j}=
\frac{1}{2}\Big(\frac{\pr}{\pr x_{2j-1}}-i\frac{\pr}{\pr x_{2j}}\Big)\,,\quad\pr_{\ol z_j}=
\frac{\pr}{\pr\ol z_j}=\frac{1}{2}\Big(\frac{\pr}{\pr x_{2j-1}}+i\frac{\pr}{\pr x_{2j}}\Big),\\
&\pr^\alpha_z=\pr^{\alpha_1}_{z_1}\ldots\pr^{\alpha_n}_{z_n}=\frac{\pr^{\abs{\alpha}}}{\pr z^\alpha}\,,\quad
\pr^\alpha_{\ol z}=\pr^{\alpha_1}_{\ol z_1}\ldots\pr^{\alpha_n}_{\ol z_n}=
\frac{\pr^{\abs{\alpha}}}{\pr\ol z^\alpha}\,.
\end{split}
\]
For $j, s\in\mathbb Z$, set $\delta_{j,s}=1$ if $j=s$, $\delta_{j,s}=0$ if $j\neq s$.

Let $W$ be a $\cC^\infty$ paracompact manifold.
We let $TW$ and $T^*W$ denote the tangent bundle of $W$
and the cotangent bundle of $W$ respectively.
The complexified tangent bundle of $W$ and the complexified cotangent bundle of $W$ are be denoted by $\Complex TW$
and $\Complex T^*W$, respectively. Write $\langle\,\cdot\,,\cdot\,\rangle$ to denote the pointwise
duality between $TW$ and $T^*W$.
We extend $\langle\,\cdot\,,\cdot\,\rangle$ bilinearly to $\Complex TW\times\Complex T^*W$.
Let $G$ be a $\cC^\infty$ vector bundle over $W$. The fiber of $G$ at $x\in W$ will be denoted by $G_x$.
Let $E$ be another vector bundle over $W$. We write
$E\boxtimes G^*$ to denote the vector bundle over $W\times W$ with fiber over $(x, y)\in W\times W$
consisting of the linear maps from $G_y$ to $E_x$.  Let $Y\subset W$ be an open set. 
From now on, the spaces of distribution sections of $G$ over $Y$ and
smooth sections of $G$ over $Y$ will be denoted by $\mathscr D'(Y, G)$ and $\cC^\infty(Y, G)$ respectively.
Let $\mathscr E'(Y, G)$ be the subspace of $\mathscr D'(Y, G)$ whose elements have compact support in $Y$ and 
let $\cC^\infty_0(Y, G)$ be the subspace of $\cC^\infty(Y, G)$ whose elements have compact support in $Y$.
For $m\in\Real$, let $H^m(Y, G)$ denote the Sobolev space
of order $m$ of sections of $G$ over $Y$. Put
\begin{gather*}
H^m_{\rm loc\,}(Y, G)=\big\{u\in\mathscr D'(Y, G);\, \varphi u\in H^m(Y, G),
      \, \mbox{for every $\varphi\in \cC^\infty_0(Y)$}\big\}\,,\\
      H^m_{\rm comp\,}(Y, G)=H^m_{\rm loc}(Y, G)\cap\mathscr E'(Y, G)\,.
\end{gather*}
We recall the Schwartz kernel theorem \cite[Theorems\,5.2.1, 5.2.6]{Hor03}, \cite[Thorem\,B.2.7]{MM07}.
Let $G$ and $E$ be $\cC^\infty$ vector
bundles over a paracompact orientable $\cC^\infty$ manifold $W$ equipped with a smooth density of integration. If
$A: \cC^\infty_0(W,G)\To \mathscr D'(W,E)$
is continuous, we write $A(x, y)$ to denote the distribution kernel of $A$.
The following two statements are equivalent
\begin{enumerate}
\item $A$ is continuous: $\mathscr E'(W,G)\To \cC^\infty(W,E)$,
\item $A(x,y)\in \cC^\infty(W\times W,E\boxtimes G^*)$.
\end{enumerate}
If $A$ satisfies (a) or (b), we say that $A$ is smoothing on $W$. Let
$A,B: \cC^\infty_0(W,G)\to \mathscr D'(W,E)$ be continuous operators.
We write 
\begin{equation} \label{e-gue160507f}
\mbox{$A\equiv B$ (on $W$)} 
\end{equation}
if $A-B$ is a smoothing operator. 

We say that $A$ is properly supported if the restrictions of the two projections 
$(x,y)\mapsto x$, $(x,y)\mapsto y$ to ${\rm Supp\,}A(x,y)$
are proper.

Let $H(x,y)\in\mathscr D'(W\times W,E\boxtimes G^*)$. We write $H$ to denote the unique 
continuous operator $\cC^\infty_0(W,G)\To\mathscr D'(W,E)$ with distribution kernel $H(x,y)$. 
In this work, we identify $H$ with $H(x,y)$. 

\subsection{Some standard notations in semi-classical analysis}\label{s-gue190813}

Let $W_1$ be an open set in $\Real^{N_1}$ and let $W_2$ be an open set in $\Real^{N_2}$. Let $E$ and $F$ be vector bundles over $W_1$ and $W_2$, respectively. 
An $m$-dependent continuous operator
$A_m:\cC^\infty_0(W_2,F)\To\mathscr D'(W_1,E)$ is called $m$-negligible on $W_1\times W_2$
if, for $m$ large enough, $A_m$ is smoothing and, for any $K\Subset W_1\times W_2$, any
multi-indices $\alpha$, $\beta$ and any $N\in\mathbb N$, there exists $C_{K,\alpha,\beta,N}>0$
such that
\begin{equation}\label{e-gue13628III}
\abs{\pr^\alpha_x\pr^\beta_yA_m(x, y)}\leq C_{K,\alpha,\beta,N}m^{-N}\:\: \text{on $K$},\ \ \forall m\gg1.
\end{equation}
In that case we write
\[A_m(x,y)=O(m^{-\infty})\:\:\text{on $W_1\times W_2$,}\]
or
\[A_m=O(m^{-\infty})\:\:\text{on $W_1\times W_2$.}\]
If $A_m, B_m:\cC^\infty_0(W_2, F)\To\mathscr D'(W_1, E)$ are $m$-dependent continuous operators,
we write $A_m= B_m+O(m^{-\infty})$ on $W_1\times W_2$ or $A_m(x,y)=B_m(x,y)+O(m^{-\infty})$ on $W_1\times W_2$ if $A_m-B_m=O(m^{-\infty})$ on $W_1\times W_2$. 
When $W=W_1=W_2$, we sometime write "on $W$". 

Let $\Omega_1$ and $\Omega_2$ be smooth manifolds and let $E$ and $F$ be vector bundles over $\Omega_1$ and $\Omega_2$, respectively. Let $A_m, B_m: \cC^\infty(\Omega_2,F)\To\cC^\infty(\Omega_1,E)$ be $m$-dependent smoothing operators. We write $A_m=B_m+O(m^{-\infty})$ on $\Omega_1\times\Omega_2$ if on every local coordinate patch $D$ of $\Omega_1$ and local coordinate patch $D_1$ of $\Omega_2$, $A_m=B_m+O(m^{-\infty})$ on $D\times D_1$.
When $\Omega_1=\Omega_2$, we sometime write on $\Omega_1$.

We recall the definition of the semi-classical symbol spaces

\begin{defn} \label{d-gue140826}
Let $W$ be an open set in $\Real^N$. Let
\begin{equation*}
\begin{split}
&S(1;W):=\Big\{a\in\cC^\infty(W)\,|\, \mbox{for every $\alpha\in\mathbb N^N_0$}:
\sup_{x\in W}\abs{\pr^\alpha a(x)}<\infty\Big\},\\
S^0_{{\rm loc\,}}(1;W):=&
\Big\{(a(\cdot,m))_{m\in\Real}\,|\, \mbox{for all $\alpha\in\mathbb N^N_0$,
$\chi\in\cC^\infty_0(W)$}\,:\:\sup_{m\geq1}\sup_{x\in W}
\abs{\pr^\alpha(\chi a(x,m))}<\infty\Big\}\,.
\end{split}
\end{equation*}
For $k\in\Real$, let
\[
S^k_{{\rm loc}}(1):=S^k_{{\rm loc}}(1;W)=\Big\{(a(\cdot,m))_{m\in\Real}\,|\,(m^{-k}a(\cdot,m))\in S^0_{{\rm loc\,}}(1;W)\Big\}\,.
\]
Hence $a(\cdot,m)\in S^k_{{\rm loc}}(1;W)$ if for every $\alpha\in\mathbb N^N_0$ and $\chi\in\cC^\infty_0(W)$, there
exists $C_\alpha>0$ independent of $m$, such that $\abs{\pr^\alpha (\chi a(\cdot,m))}\leq C_\alpha m^{k}$ holds on $W$.

Consider a sequence $a_j\in S^{k_j}_{{\rm loc\,}}(1)$, $j\in\N_0$, where $k_j\searrow-\infty$,
and let $a\in S^{k_0}_{{\rm loc\,}}(1)$. We say
\[
a(\cdot,m)\sim
\sum\limits^\infty_{j=0}a_j(\cdot,m)\:\:\text{in $S^{k_0}_{{\rm loc\,}}(1)$},
\]
if, for every
$\ell\in\N_0$, we have $a-\sum^{\ell}_{j=0}a_j\in S^{k_{\ell+1}}_{{\rm loc\,}}(1)$ .
For a given sequence $a_j$ as above, we can always find such an asymptotic sum
$a$, which is unique up to an element in
$S^{-\infty}_{{\rm loc\,}}(1)=S^{-\infty}_{{\rm loc\,}}(1;W):=\cap _kS^k_{{\rm loc\,}}(1)$.

Similarly, we can define $S^k_{{\rm loc\,}}(1;Y,E)$ in the standard way, where $Y$ is a smooth manifold and $E$ is a vector bundle over $Y$. 
\end{defn}

\subsection{Set up and Terminology} \label{s-su}

Let $M$ be a relatively compact open subset with $\cC^\infty$ boundary $X$ of a
complex manifold $M'$ of dimension $n$. Let $T^{1,0}M'$ and $T^{0,1}M'$ be the holomorphic tangent bundle of $M'$ and the anti-holomorphic tangent bundle of $M'$. 
We fix a Hermitian metric $\langle\,\cdot\,|\,\cdot\,\rangle$ on $\Complex TM'$ so that $T^{1,0}M'\perp T^{0,1}M'$. For $p, q\in\mathbb N$, let $T^{*p,q}M'$ be the vector bundle of $(p,q)$ forms on $M'$. The Hermitian metric $\langle\,\cdot\,|\,\cdot\,\rangle$ on $\Complex TM'$ induces by duality, a Hermitian metric $\langle\,\cdot\,|\,\cdot\,\rangle$ on 
$\oplus^{p,q=n}_{p,q=1}T^{*p,q}M'$. Let $\norm{\cdot}$ be the corresponding pointwise norm with respect to $\langle\,\cdot\,|\,\cdot\,\rangle$. 
Let $\rho\in\cC^\infty(M',\Real)$ be a defining function of $X$. That is, $\rho=0$ on $X$, $\rho<0$ on $M$ and $d\rho\neq0$ near $X$. From now on, we take a defining function $\rho$ so that $\norm{d\rho}=1$ on $X$. 

Let $A$ be a $\cC^\infty$ vector bundle over $M'$. 
Let $U$ be an open set in $M'$. Let 
\[
\begin{split}
&\cC^\infty(U\cap \ol M,A),\ \ \mathscr D'(U\cap \ol M,A),\ \ 
\mathscr E'(U\cap \ol M,A),\\ 
H^s&(U\cap \ol M,A),\ \ H^s_{{\rm comp\,}}(U\cap \ol M,A),\ \ 
H^s_{{\rm loc\,}}(U\cap \ol M,A),
\end{split}
\]
(where $\ s\in\mathbb R$)
denote the spaces of restrictions to $U\cap\ol M$ of elements in 
\[
\begin{split}
\cC^\infty&(U\cap M',A),\ \ \mathscr D'(U\cap M',A),\ \ 
\mathscr E'(U\cap M',A),\\  
&H^s(M',A),\ \  H^s_{{\rm comp\,}}(M',A),\ \  
H^s_{{\rm loc\,}}(M',A),
\end{split}
\] 
respectively. Write 
\[
\begin{split}
L^2(U\cap\ol M,A):=&H^0(U\cap \ol M,A),\ \ 
L^2_{{\rm comp\,}}(U\cap\ol M,A):=H^0_{{\rm comp\,}}(U\cap \ol M,A),\\ 
&L^2_{{\rm loc\,}}(U\cap\ol M,A):=H^0_{{\rm loc\,}}(U\cap \ol M,A).
\end{split}
\]
For every $p, q=1,\ldots,n$, we denote 
\[
\begin{split}
\Omega^{p,q}&(U\cap\ol M):=\cC^\infty(U\cap\ol M,T^{*p,q}M'),\ \ 
\Omega^{p,q}(M'):=\cC^\infty(M',T^{*p,q}M'),\\ 
&\Omega^{p,q}_0(M'):=\cC^\infty_0(M',T^{*p,q}M'),\ \ 
\Omega^{p,q}_0(M):=\cC^\infty_0(M,T^{*p,q}M').
\end{split}
\]
Let $A$ and $B$ be $\cC^\infty$ vector bundles over $M'$. 
Let $U$ be an open set in $M'$. Let 
$$F_1, F_2: \cC^\infty_0(U\cap\ol M,A)\To\mathscr D'(U\cap\ol M,B)$$ 
be continuous operators. Let 
$F_1(x,y), F_2(x,y)\in\mathscr D'((U\times U)\cap(\ol M\times\ol M), A\boxtimes B^*)$ 
be the distribution kernels of $F_1$ and $F_2$ respectively. 
We write 
$$F_1\equiv F_2\mod\cC^\infty((U\times U)\cap(\ol M\times\ol M))$$ 
or $F_1(x,y)\equiv F_2(x,y)\mod\cC^\infty((U\times U)\cap(\ol M\times\ol M))$ 
if $F_1(x,y)=F_2(x,y)+r(x,y)$, where 
$r(x,y)|_{(U\times U)\cap(\ol M\times\ol M)}\in\cC^\infty((U\times U)\cap(\ol M\times\ol M),A\boxtimes B^*)$. Similarly, 
let $\hat F_1, \hat F_2: \cC^\infty_0(U\cap\ol M,A)\To\mathscr D'(U\cap X,B)$ be continuous operators. Let 
$\hat F_1(x,y), \hat F_2(x,y)\in\mathscr D'((U\times U)\cap(X\times\ol M), A\boxtimes B^*)$ be the distribution kernels of $\hat F_1$ and $\hat F_2$ respectively. We write 
$\hat F_1\equiv\hat F_2\mod\cC^\infty((U\times U)\cap(X\times\ol M))$ or $\hat F_1(x,y)\equiv\hat F_2(x,y)\mod\cC^\infty((U\times U)\cap(X\times\ol M))$ if $\hat F_1(x,y)=\hat F_2(x,y)+\hat r(x,y)$, where 
$\hat r(x,y)\in\cC^\infty((U\times U)\cap(X\times\ol M),A\boxtimes B^*)$. Similarly, let $\tilde F_1, \tilde F_2: \cC^\infty_0(U\cap X,A)\To\mathscr D'(U\cap\ol M,B)$ be continuous operators. Let 
\[\tilde F_1(x,y), \tilde F_2(x,y)\in\mathscr D'((U\times U)\cap(\ol M\times X), A\boxtimes B^*)\] 
be the distribution kernels of $\tilde F_1$ and $\tilde F_2$ respectively. We write 
$\tilde F_1\equiv\tilde F_2\mod\cC^\infty((U\times U)\cap(\ol M\times X))$ or $\tilde F_1(x,y)\equiv\tilde F_2(x,y)\mod\cC^\infty((U\times U)\cap(\ol M\times X))$ if $\tilde F_1(x,y)=\tilde F_2(x,y)+\tilde r(x,y)$, where 
$\tilde r(x,y)\in\cC^\infty((U\times U)\cap(\ol M\times X),A\boxtimes B^*)$.

Let $F_m, G_m: \cC^\infty_0(U\cap\ol M,A)\To\mathscr D'(U\cap\ol M,B)$ be $m$-dependent continuous operators. Let $F_m(x,y), G_m(x,y)\in\mathscr D'((U\times U)\cap(\ol M\times\ol M), A\boxtimes B^*)$ be the distribution kernels of $F_m$ and $G_m$ respectively. We write 
\begin{equation}\label{e-gue190813yyd}
F_m\equiv G_m\mod O(m^{-\infty})\ \ \mbox{ on $U\cap\ol M$} 
\end{equation}
if there is a $r_m(x,y)\in \cC^\infty(U\times U, A\boxtimes B^*)$ with $r_m(x,y)=O(m^{-\infty})$ on $U\times U$ such that  $r_m(x,y)|_{(U\times U)\cap(\ol M\times\ol M)}=F_m(x,y)-G_m(x,y)$, for $m\gg1$. 

Let $k\in\Real$. Let $U$ be an open set in $M'$ and let $E$ be a vector bundle over $M'\times M'$. Let 
\begin{equation}\label{e-gue190813yydI}
S^k_{{\rm loc}}(1,(U\times U)\cap(\ol M\times\ol M),E)
\end{equation}
denote the space of restrictions to $U\cap\ol M$ of elements in 
$S^k_{{\rm loc}}(1,U\times U,E)$.
Let
\[a_j\in S^{k_j}_{{\rm loc}}(1,(U\times U)\cap(\ol M\times\ol M),E),\ \ j=0,1,2,\dots,\] 
with $k_j\searrow -\infty$, $j\To \infty$.
Then there exists
\[a\in S^{k_0}_{{\rm loc}}(1,(U\times U)\cap(\ol M\times\ol M),E)\]
such that
\[a-\sum^{\ell-1}_{j=0}a_j\in S^{k_\ell}_{{\rm loc}}(1,(U\times U)\cap(\ol M\times\ol M),E),\]
for every $\ell=1,2,\ldots$. If $a$ and $a_j$ have the properties above, we write
\[a\sim\sum^\infty_{j=0}a_j \text{ in }
S^{k_0}_{{\rm loc}}(1,(U\times U)\cap(\ol M\times\ol M),E).\]

Let $dv_{M'}$ be the volume form on $M'$ induced by the Hermitian metric $\langle\,\cdot\,|\,\cdot\,\rangle$ on $\Complex TM'$ and 
and let $(\,\cdot\,|\,\cdot\,)_M$ and $(\,\cdot\,|\,\cdot\,)_{M'}$ be the inner products on $\Omega^{0,q}(\ol M)$ and $\Omega^{0,q}_0(M')$
defined by
\begin{equation} \label{e-gue190312}
\begin{split}
&(\,f\,|\,h\,)_M=\int_M\langle\,f\,|\,h\,\rangle dv_{M'},\ \ f, h\in\Omega^{0,q}(\ol M),\\
&(\,f\,|\,h\,)_{M'}=\int_{M'}\langle\,f\,|\,h\,\rangle dv_{M'},\ \ f, h\in\Omega^{0,q}_0(M').
\end{split}
\end{equation}
Let $\norm{\cdot}_M$ and $\norm{\cdot}_{M'}$ be the corresponding norms with respect to $(\,\cdot\,|\,\cdot\,)_M$ and $(\,\cdot\,|\,\cdot\,)_{M'}$ respectively. 
Let $L^2_{(0,q)}(M)$ be the completion of $\Omega^{0,q}(\ol M)$ with respect to $(\,\cdot\,|\,\cdot\,)_M$. We extend $(\,\cdot\,|\,\cdot\,)_M$ to $L^2_{(0,q)}(M)$
in the standard way. 
Let $\ddbar: \Omega^{0,q}(M')\To\Omega^{0,q+1}(M')$
be the part of the exterior differential operator which maps forms of type $(0,q)$ to forms of
type $(0,q+1)$ and we denote by
$\ol{\pr}^*_f: \Omega^{0,q+1}(M')\To\Omega^{0,q}(M')$
the formal adjoint of $\ddbar$. That is
\[(\,\ddbar f\,|\,h\,)_{M'}=(f\,|\,\ol{\pr}^*_fh\,)_{M'},\]
$f\in\Omega^{0,q}_0(M')$, $h\in\Omega^{0,q+1}(M')$.
We shall also use the notation $\ddbar$ for the closure in $L^2$ of the $\ddbar$ operator, initially defined on
$\Omega^{0,q}(\ol M)$ and $\ddbar^*$  for the Hilbert space adjoint 
of $\ddbar$. Recall that for $u\in L^2_{(0,q}(M)$, we say that 
$u\in{\rm Dom\,}\ddbar$ if we can find a sequence 
$u_j\in\Omega^{0,q}(\ol M)$, $j=1,2,\ldots$, 
with $\lim_{j\To\infty}\norm{u_j-u}_M=0$ such that 
$\lim_{j\To\infty}\norm{\ddbar u_j-v}_M=0$, for some 
$v\in L^2_{(0,q+1)}(M)$. We set $\ddbar u=v$. 
The $\ddbar$-Neumann Laplacian on $(0, q)$-forms is then 
the non-negative self-adjoint operator in the space 
$L^2_{(0,q)}(M)$ (see~\cite[Chapter 1]{FK72}):
\begin{equation} \label{e-gue190312syd}
\Box^{(q)}=\ddbar\,\ddbar^*+\ddbar^*\,\ddbar: 
{\rm Dom\,}\Box^{(q)}\subset L^2_{(0,q)}(M)\To L^2_{(0,q)}(M),
\end{equation}
where
\begin{equation} \label{e-gue190312sydq}
\begin{split}
{\rm Dom\,}\Box^{(q)}=\Big\{
& u\in L^2_{(0,q)}(M), u\in{\rm Dom\,}\ddbar^*\cap{\rm Dom\,}\ddbar, \\
&\ddbar^*u\in{\rm Dom\,}\ddbar,\ \ \ddbar u\in{\rm Dom\,}\ddbar^*\Big\}
\end{split}\end{equation}
and $\Omega^{0,q}(\ol M)\cap{\rm Dom\,}\Box^{(q)}$ is dense in 
${\rm Dom\,}\Box^{(q)}$ for the norm
\[u\in{\rm Dom\,}\Box^{(q)}\To \big\|u\big\|_M+\big\|\ddbar u\big\|_M+
\big\|\ddbar^*u\big\|_M\]
(see~\cite[p.14]{FK72}). We shall write ${\rm Spec\,}\Box^{(q)}$ 
to denote the spectrum of $\Box^{(q)}$. 
For a Borel set $B\subset\Real$ we denote by $E(B)$ 
the spectral projection of $\Box^{(q)}$ 
corresponding to the set $B$, where $E$ is the spectral measure 
of $\Box^{(q)}$. For $\lambda\geq0$, we set
\begin{equation} \label{e-gue190312sydI}
H^q_{\leq\lambda}(\ol M):={\rm Ran\,}E\bigr((-\infty,\lambda]\bigr)\subset L^2_{(0,q)}(M).
\end{equation}
For $\lambda=0$, we denote
\begin{equation} \label{e-gue190312sydII}
H^q(\ol M):=H^q_{\leq0}(\ol M)={\rm Ker\,}\Box^{(q)}.
\end{equation}
For $\lambda\geq0$, let
\begin{equation}\label{e-gue190312sydIII}
B^{(q)}_{\leq\lambda}:L^2_{(0,q)}(M)\To H^q_{\leq\lambda}(\ol M)
\end{equation}
be the orthogonal projection with respect to the product $(\,\cdot\,|\,\cdot\,)_M$ and let
\begin{equation}\label{e-gue190312yyd}
B^{(q)}_{\leq\lambda}(x,y)
\in\mathscr D'(M\times M,T^{*0,q}M\boxtimes(T^{*0,q}M)^*),
\end{equation}
denote the distribution kernel of $B^{(q)}_{\leq\lambda}$. For $\lambda=0$, 
we denote $B^{(q)}:=B^{(q)}_{\leq0}$, $B^{(q)}(x,y):=B^{(q)}_{\leq0}(x,y)$.

Now, we consider the boundary $X$ of $M$. The boundary $X$ is a compact CR manifold 
of dimension $2n-1$ with natural CR structure 
$T^{1,0}X:=T^{1,0}M'\cap\Complex TX$. Let $T^{0,1}X:=\ol{T^{1,0}X}$. 
The Hermitian metric on $\Complex TM'$ induces Hermitian metrics 
$\langle\,\cdot\,|\,\cdot\,\rangle$ on $\Complex TX$ and 
also on the bundle $\oplus^{2n-1}_{j=1}\Lambda^j(\Complex T^*X)$. 
Let $dv_X$ be the volume form on $X$ induced by the Hermitian metric 
$\langle\,\cdot\,|\,\cdot\,\rangle$ on $\Complex TX$ and let $(\,\cdot\,|\,\cdot\,)_X$ 
be the $L^2$ inner product on 
$\cC^\infty(X, \oplus^{2n-1}_{j=1}\Lambda^j(\Complex T^*X))$ 
induced by $dv_X$ and the Hermitian metric 
$\left\langle\,\cdot\,|\,\cdot\,\right\rangle$ on 
$\oplus^{2n-1}_{j=1}\Lambda^j(\Complex T^*X)$. 

Let $\frac{\pr}{\pr\rho}\in\cC^\infty(X,TM')$ be the global real vector field on $X$ given by 
\begin{equation}\label{e-gue190312scdq}
\begin{split}
&\Big\langle\,\frac{\pr}{\pr\rho}\,,\,d\rho\,\Big\rangle=1\ \ \text{on $X$},\\
&\Big\langle\,\frac{\pr}{\pr\rho}(p)\,|\,v\,\Big\rangle=0\ \ 
\text{at every $p\in X$, for every $v\in T_pX$.}
\end{split}
\end{equation}
Let $J: TM'\To TM'$ be  the complex structure map and put
\begin{equation}\label{e-gue190312scdqI}
T=J\Big(\frac{\pr}{\pr\rho}\Big)\in\cC^\infty(M',TM').
\end{equation}
It is easy to see that $T$ is a global non-vanishing vector 
field on $X$, $T$ is orthogonal to $T^{1,0}X\oplus T^{0,1}X$ 
and $\norm{T}=1$ on $X$. 
Put 
$$T^{*1,0}X:=(T^{0,1}X\oplus\Complex T)^\perp\subset\Complex T^*X\,,\:\: 
T^{*0,1}X:=(T^{1,0}X\oplus\Complex T)^\perp\subset\Complex T^*X.$$
Let $\omega_0\in\cC^\infty(X,T^*X)$ be the global one form on $X$ given by 
\begin{equation}\label{e-gue190312scdqII}
\begin{split}
&\langle\,\omega_0(p)\,,\,u\,\rangle=0,\ \ 
\text{for every $p\in X$ and every $u\in T^{1,0}_pX\oplus T^{0,1}_pX$},\\
&\langle\,\omega_0\,,\,T\,\rangle=-1\ \ \mbox{on $X$}.
\end{split}
\end{equation}
We have the pointwise orthogonal decompositions:
\begin{equation} \label{e-gue190312scdqIII}
\begin{split}
&\Complex T^*X=T^{*1,0}X\oplus T^{*0,1}X\oplus\set{\lambda\omega_0;\,
\lambda\in\Complex},\\
&\Complex TX=T^{1,0}X\oplus T^{0,1}X\oplus\set{\lambda T;\,\lambda\in\Complex}.
\end{split}
\end{equation}
For $p\in X$, the Levi form $\mathcal{L}_p$ is the Hermitian quadratic form on $T^{1,0}_pX$ given by
\begin{equation} \label{e-gue190312sds}
\mathcal{L}_p(Z,\ol W)=-\frac{1}{2i}\langle\,d\omega_0(p)\,,Z\wedge\ol W\,\rangle,\ \ Z, W\in T^{1,0}_pX. 
\end{equation}

Define the vector bundle of $(0, q)$ forms by $T^{*0,q}X:=\Lambda^{q}T^{*0,1}X$.
Let $D\subset X$ be an open set. Let $\Omega^{0,q}(D)$ denote the space of smooth sections 
of $T^{*0,q}X$ over $D$ and let $\Omega^{0,q}_0(D)$ be the subspace of
$\Omega^{0,q}(D)$ whose elements have compact support in $D$.

\section{The operator $\Box^{(q)}_-$}\label{s-gue190312scd}

In this section, we fix $q\in\set{0,1,\ldots,n-1}$. Let
\[\Box^{(q)}_f=\ddbar\,\ddbar^*_f+\ddbar^*_f\,\ddbar: \Omega^{0,q}(M')
\To\Omega^{0,q}(M')\]
denote the complex Laplace-Beltrami operator on $(0, q)$ forms. 
Let $\gamma$ denote the operator of restriction to the boundary $X$. Let us consider the map
\begin{equation}\label{e-gue190312scd}
\begin{split}
F^{(q)}:H^{2}(\ol M,T^{*0,q}M')&\rightarrow L^{2}_{(0,q)}(M)\oplus
H^{\frac{3}{2}}(X,T^{*0,q}M')\\
u&\mapsto (\Box_f^{(q)}u, \gamma u).
\end{split}
\end{equation}
By \cite{B71} we know that
$\dim{\rm Ker\,}F^{(q)}<\infty$ and ${\rm Ker\,}F^{(q)}\subset \Omega^{0,q}(\ol M)$. Let
\begin{equation}\label{e-gue190312scdI}
K^{(q)}: H^2(\overline M, T^{*0,q}M')\rightarrow{\rm Ker\,}F^{(q)}
\end{equation}
be the orthogonal projection with respect to $(\,\cdot\,|\,\cdot\,)_M$.  
Put $\tilde \Box_f^{(q)}=\Box^{(q)}_f+K^{(q)}$ and consider the map
\begin{equation}\label{e-gue190312scdII}
\begin{split}
\tilde F^{(q)}: H^2(\overline M, T^{*0,q}M')&\rightarrow 
L^{2}_{(0,q)}(M)\oplus H^{\frac{3}{2}}(X,T^{*0,q}M'),\\
u&\mapsto (\tilde\Box_f^{(q)}u, \gamma u).
\end{split}
\end{equation}
It is easy to see that $\tilde F^{(q)}$ is injective. Let
\begin{equation}\label{e-gue190312scdIII}
\tilde P: \cC^\infty(X, T^{*0,q}M')\rightarrow\Omega^{0,q}(\overline M)
\end{equation}
be the Poisson operator for $\tilde \Box^{(q)}_f$ which is well-defined 
since \eqref{e-gue190312scdII} is injective. The Poisson operator $\tilde P$ satisfies
\begin{equation}\label{e-gue190312mscd}
\begin{split}
&\tilde\Box^{(q)}_f\tilde Pu=0,\ \ \gamma\tilde Pu=u,\ \ 
\text{ for every} u\in\cC^\infty(X, T^{*0,q}M').\\
\end{split}
\end{equation}
By Boutet de Monvel~\cite[p.\,29]{B71} the operator $\tilde P$ 
extends continuously
\begin{equation}\label{e-gue190313ad}
\tilde P: H^s(X, T^{*0,q}M')\rightarrow H^{s+\frac{1}{2}}(\overline M, T^{*0,q}M'),\ \ \forall s\in\mathbb R,
\end{equation}
and there is a continuous operator 
\begin{equation}\label{e-gue190429yyd}
D^{(q)}: H^s(\ol M, T^{*0,q}M')\rightarrow H^{s+2}(\overline M, T^{*0,q}M'),\ \ 
\forall s\in\mathbb R,
\end{equation}
such that 
\begin{equation}\label{e-gue190418yyda}
D^{(q)}\tilde\Box^{(q)}_f+\tilde P\gamma=I\ \ \mbox{on $\Omega^{0,q}(\ol M)$}.
\end{equation}
Let $\hat{\mathscr E}'(\overline M, T^{*0,q}M')$ denote the space 
of continuous linear map from $\Omega^{0,q}(\ol M)$ to $\Complex$ 
with respect to $(\,\cdot\,|\,\cdot\,)_M$. 
Let
\begin{equation}\label{e-gue190515yyda}
\tilde P^*: \hat{\mathscr E}'(\overline M, T^{*0,q}M')\rightarrow\mathscr D'(X, T^{*0,q}M')
\end{equation} 
be the operator defined by
\[(\,\tilde P^* u\,|\,v\,)_X=(\,u\,|\,\tilde Pv\,)_M,\ \ 
u\in \hat{\mathscr E}'(\overline M, T^{*0,q}M'),\ \  
v\in\cC^\infty(X, T^{*0,q}M').\]
By \cite[p.\,30]{B71} the operator
\begin{equation}\label{e-gue190515yyd}
\tilde P^*: H^s(\ol M, T^{*0,q}M')\rightarrow H^{s+\frac{1}{2}}(X, T^{*0,q}M'),\ \,
\end{equation} 
is continuous for every $s\in\mathbb R$ and
\[\tilde P^*: \Omega^{0,q}(\overline M)\rightarrow\cC^\infty(X, T^{*0,q}M').\]

Let $L\in T^{*0,1}M'$ and let $L^{\wedge}: T^{*0,q}M'\To T^{*0,q+1}M'$ 
be the operator with wedge multiplication by $L$ and
let $L^{\wedge,*}:T^{*0,q+1}M'\To T^{*0,q}M'$ be its adjoint with respect 
to $\langle\,\cdot\,|\,\cdot\,\rangle$, that is,
\begin{equation}\label{e-gue190312mscdI}
\langle\,L\wedge u\,|\,v\,\rangle=\langle\,u\,|\,L^{\wedge,*} v\,\rangle,\ \ 
u\in T^{*0,q}M',\ \ v\in T^{*0,q+1}M'.
\end{equation}
Let 
\begin{equation}\label{e-gue190312mscdII}
\Box^{(q)}_-:=(\ddbar\rho)^{\wedge,*}\gamma\ddbar
\tilde P: \Omega^{0,q}(X)\To\Omega^{0,q}(X). 
\end{equation}
In this section, we will construct parametrix for $\Box^{(q)}_-$ 
under certain Levi curvature assumptions. Let 
\[\triangle_X:=dd^*+d^*d:\cC^\infty(X,\Lambda^q(\Complex T^*X))\To
\cC^\infty(X,\Lambda^q(\Complex T^*X))\]
be the De-Rham Laplacian, where 
\[d^*: \cC^\infty(X,\Lambda^{q+1}(\Complex T^*X))\To 
\cC^\infty(X,\Lambda^q(\Complex T^*X))\] 
is the formal adjoint of the exterior derivative $d$ with respect to 
$(\,\cdot\,|\,\cdot\,)_X$. Let $\sqrt{-\triangle_X}$ be the square 
root of $-\triangle_X$. Put 
\begin{equation}\label{e-gue190313sydI}
\begin{split}
&\Sigma^-=\set{(x,\lambda\omega_0(x))\in T^*X;\, \lambda<0},\\
&\Sigma^+=\set{(x,\lambda\omega_0(x))\in T^*X;\, \lambda>0}.
\end{split}
\end{equation}
\begin{thm}[{\cite[Proposition 4.1]{Hsiao08}}] \label{t-gue180313syd}
The operator 
\[\Box^{(q)}_-:=(\ddbar\rho)^{\wedge,*}\gamma\ddbar\tilde P: 
\Omega^{0,q}(X)\To\Omega^{0,q}(X)\]
is a classical
pseudodifferential operator of order one and we have
\begin{equation} \label{e-gue190313sydII}
\Box^{(q)}_-=\frac{1}{2}(iT+\sqrt{-\triangle_X})+\text{lower order terms}.
\end{equation}
In particular, $\Box^{(q)}_-$ is elliptic outside $\Sigma^-$.
\end{thm}
Let $\ddbar_b:\Omega^{0,q}(X)\To\Omega^{0,q+1}(X)$ 
be the tangential Cauchy-Riemann operator. It is not difficult to see that 
\begin{equation}\label{e-gue190313sydIII}
\ddbar_b=2(\ddbar\rho)^{\wedge,*}(\ddbar\rho)^\wedge\gamma\ddbar\tilde P
: \Omega^{0,q}(X)\To\Omega^{0,q+1}(X).
\end{equation}
We notice that for $u\in\cC^\infty(X,\Lambda^q(\Complex T^*X))$, 
\begin{equation}\label{e-gue190313scs}
\mbox{$u\in\Omega^{0,q}(X)$ if and only if $u=
2(\ddbar\rho)^{\wedge,*}(\ddbar\rho)^\wedge u$ on $X$}
\end{equation}
and 
\begin{equation}\label{e-gue190313scsI}
\mbox{$2(\ddbar\rho)^{\wedge,*}
(\ddbar\rho)^\wedge+2(\ddbar\rho)^\wedge(\ddbar\rho)^{\wedge,*}=I$ 
on $\cC^\infty(X,\Lambda^q(\Complex T^*X))$}.
\end{equation}
Consider 
\[\gamma\ol{\pr}^*_f\tilde Pu: \cC^\infty(X,  \Lambda^{q+1}
(\Complex T^*X))\To\cC^\infty(X, \Lambda^q(\Complex T^*X)).\]
 It is not difficult to check that (see~\cite[Lemma 2.2]{Hsiao08})
\begin{equation}\label{e-gue190313scd}
\gamma\ol{\pr}^*_f\tilde P: \Omega^{0,q+1}(X)\To\Omega^{0,q}(X). 
\end{equation}
Put 
\begin{equation}\label{e-gue190313scdI}
\Td\Box^{(q)}_b:=\gamma\ol{\pr}^*_f\tilde P\ddbar_b+\ddbar_b\gamma\ol{\pr}^*_f\tilde P:\Omega^{0,q}(X)\To\Omega^{0,q}(X). 
\end{equation}
\begin{lem}\label{l-gue190312mscd}
We have 
\[\Td\Box^{(q)}_b=-4(\ddbar\rho)^{\wedge,*}(\ddbar\rho)^\wedge\gamma\ol{\pr}^*_f\tilde P(\ddbar\rho)^\wedge\Box^{(q)}_-+R^{(q)}\ \ \mbox{on $\Omega^{0,q}(X)$},\]
where $R^{(q)}: \Omega^{0,q}(X)\To\Omega^{0,q}(X)$ is a smoothing operator. 
\end{lem}

\begin{proof}
From \eqref{e-gue190313sydIII}, \eqref{e-gue190313scs}, \eqref{e-gue190313scsI},  \eqref{e-gue190313scd} and \eqref{e-gue190312mscd}, we have 
\begin{equation}\label{e-gue190313scsII}
\begin{split}
\Td\Box^{(q)}_b&=2(\ddbar\rho)^{\wedge,*}(\ddbar\rho)^\wedge\Td\Box^{(q)}_b\\
&=2(\ddbar\rho)^{\wedge,*}(\ddbar\rho)^\wedge\Bigr(\gamma\ol{\pr}^*_f\tilde P\ddbar_b+\ddbar_b\gamma\ol{\pr}^*_f\tilde P)\\
&=2(\ddbar\rho)^{\wedge,*}(\ddbar\rho)^\wedge\gamma\ol{\pr}^*_f\tilde P\ddbar_b+2(\ddbar\rho)^{\wedge,*}(\ddbar\rho)^\wedge\ddbar_b\gamma\ol{\pr}^*_f\tilde P\\
&=2(\ddbar\rho)^{\wedge,*}(\ddbar\rho)^\wedge\gamma\ol{\pr}^*_f\tilde P\ddbar_b+2(\ddbar\rho)^{\wedge,*}(\ddbar\rho)^\wedge\gamma\ddbar\tilde P\gamma\ol{\pr}^*_f\tilde P\\
&=2(\ddbar\rho)^{\wedge,*}(\ddbar\rho)^\wedge\gamma\ol{\pr}^*_f\tilde P\Bigr(\gamma\ddbar\tilde P-2(\ddbar\rho)^\wedge(\ddbar\rho)^{\wedge,*}\gamma\ddbar\tilde P\Bigr)\\
&\quad+2(\ddbar\rho)^{\wedge,*}(\ddbar\rho)^\wedge\gamma\ddbar\tilde P\gamma\ol{\pr}^*_f\tilde P\\
&=2(\ddbar\rho)^{\wedge,*}(\ddbar\rho)^\wedge\gamma\Box^{(q)}_f\tilde P-4(\ddbar\rho)^{\wedge,*}(\ddbar\rho)^\wedge\gamma\ol{\pr}^*_f\tilde P(\ddbar\rho)^\wedge\Box^{(q)}_-\\
&=-2(\ddbar\rho)^{\wedge,*}(\ddbar\rho)^\wedge\gamma K^{(q)}\tilde P-4(\ddbar\rho)^{\wedge,*}(\ddbar\rho)^\wedge\gamma\ol{\pr}^*_f\tilde P(\ddbar\rho)^\wedge\Box^{(q)}_-,
\end{split}
\end{equation}
where $K^{(q)}$ is as in \eqref{e-gue190312scdI}. Note that $K^{(q)}\equiv0\mod\cC^\infty(\ol M\times\ol M)$. From this observation and \eqref{e-gue190313ad}, we deduce that 
\[-2(\ddbar\rho)^{\wedge,*}(\ddbar\rho)^\wedge\gamma K^{(q)}\tilde P: H^s(X, T^{*0,q}M')\rightarrow H^{s+N}(X, T^{*0,q}M'),\]
for every $s\in\mathbb R$ and every $N\in\mathbb N$. Hence, $-2(\ddbar\rho)^{\wedge,*}(\ddbar\rho)^\wedge\gamma K^{(q)}\tilde P$ is smoothing. From this observation and \eqref{e-gue190313scsII}, the lemma follows. 
\end{proof}

Lemma~\ref{l-gue190312mscd} gives a relation between $\Td\Box^{(q)}_b$ and $\Box^{(q)}_-$. Put 
\begin{equation}\label{e-gue190313adI}
A^{(q)}:=-4(\ddbar\rho)^{\wedge,*}(\ddbar\rho)^\wedge\gamma\ol{\pr}^*_f\tilde P(\ddbar\rho)^\wedge: \Omega^{0,q}(X)\To\Omega^{0,q}(X). 
\end{equation}
Then, $\Td\Box^{(q)}_b\equiv A^{(q)}\Box^{(q)}_-$. The operator $A^{(q)}$ is a classical pseudodifferential operaor of order one. We are going to show that 
$A^{(q)}$ is elliptic near $\Sigma^-$. 
We pause and introduce some notations. Near $X$, put
\begin{equation} \label{e-gue190313acd}
\Td T^{*0,1}_zM'=\set{u\in T^{*0,1}_zM';\, \langle\,u\,|\,\ddbar\rho(z)\,\rangle=0}
\end{equation}
and
\begin{equation} \label{e-gue190313acdI}
\Td T^{0,1}_zM'=\set{u\in T^{0,1}_zM';\, \langle\,u\,|\,(iT+\frac{\pr}{\pr\rho})(z)\,\rangle=0}.
\end{equation}
We have the orthogonal decompositions with respect to $\langle\,\cdot\,|\,\cdot\,\rangle$ for every $z\in M'$, $z$ is near $X$: 
\begin{equation} \label{e-gue190313acdII}
\begin{split}
&T^{*0,1}_zM'=\Td T^{*,0,1}_zM'\oplus\set{\lambda(\ddbar\rho)(z);\, \lambda\in\Complex},\\
&T^{0,1}_zM'=\Td T^{0,1}_zM'\oplus\set{\lambda(iT+\frac{\pr}{\pr\rho})(z);\, \lambda\in\Complex}.
\end{split}
\end{equation}
Note that $\Td T^{*,0,1}_zM'=T^{*0,1}_zX$, $\Td T^{0,1}_zM'=T^{0,1}_zX$,  for every $z\in X$. Fix $z_0\in X$. We can choose an orthonormal frame
$t_1(z),\ldots,t_{n-1}(z)$
for $\Td T^{*,0,1}_zM'$ varying smoothly with $z$ in a neighborhood $U$ of $z_0$ in $M'$. Then 
\[t_1(z),\ldots,t_{n-1}(z),t_n(z):=\frac{\ddbar\rho(z)}{\norm{\ddbar\rho(z)}}\]
is an orthonormal frame for $T^{*0,1}_zM'$. Let
\[T_1(z),\ldots,T_{n-1}(z),T_n(z)\]
denote the basis of $T^{0,1}_zM'$ which is dual to
$t_1(z),\ldots,t_n(z)$. We have $T_j(z)\in\Td T^{0,1}_zM'$, $j=1,\ldots,n-1$, 
and $T_n=\frac{iT+\frac{\pr}{\pr\rho}}{\norm{iT+\frac{\pr}{\pr\rho}}}$.
By \cite[(4.11]{Hsiao08})
\begin{equation} \label{e-gue190313hcd}
\gamma\ol{\pr_f}^*\tilde P=\sum^{n-1}_{j=1}t^{\wedge, *}_j\circ T^*_j
+(\ddbar\rho)^{\wedge, *}\circ(iT-\sqrt{-\triangle_X})+
\text{lower order terms},
\end{equation}
where $T^*_j$ is the adjoint of $T_j$, $j=1,\ldots,n-1$, that is, 
$(\,T_jf\,|\,g\,)_X=(\,f\,|\,T^*_jg\,)_X$, for every 
$f, g\in\cC^\infty_0(U\cap X)$, $j=1,\ldots,n-1$. 

\begin{thm}\label{t-gue190313ad}
We have 
\begin{equation}\label{e-gue190314cydII}
A^{(q)}=-(iT-\sqrt{-\triangle})+\mbox{lower order terms}\ \ \mbox{on $\Omega^{0,q}(X)$}.
\end{equation}
Hence, the  operator $A^{(q)}$ is elliptic near $\Sigma^-$. 
\end{thm}

\begin{proof}
From \eqref{e-gue190313hcd} and \eqref{e-gue190313adI}, we have 
\begin{equation}\label{e-gue190314cyd}
A^{(q)}=-4(\ddbar\rho)^{\wedge,*}(\ddbar\rho)^\wedge\Bigr(\sum^{n-1}_{j=1}t^{\wedge,*}_j(\ddbar\rho)^\wedge T^*_j+(\ddbar\rho)^{\wedge,*}(\ddbar\rho)^\wedge(iT-\sqrt{-\triangle})+\mbox{lower order terms}\Bigr).
\end{equation}
We notice that 
\begin{equation}\label{e-gue190314cydI}
\begin{split}
&(\ddbar\rho)^{\wedge,*}(\ddbar\rho)^\wedge(\ddbar\rho)^\wedge T^*_j=0\ \ \mbox{on $\Omega^{0,q}(X)$, for every $j=1,\ldots,n-1$},\\
&4(\ddbar\rho)^{\wedge,*}(\ddbar\rho)^\wedge(\ddbar\rho)^{\wedge,*}(\ddbar\rho)^\wedge=I\ \ \mbox{on $\Omega^{0,q}(X)$}. 
\end{split}
\end{equation}
From \eqref{e-gue190314cyd} and \eqref{e-gue190314cydI}, we get \eqref{e-gue190314cydII}. 
\end{proof}

We pause and introduce some notations. Let $D$ be an open set of $X$. Let
$$L^m_{\frac{1}{2},\frac{1}{2}}(D,T^{*0,q}X\boxtimes(T^{*0,q}X)^*)\,,\:\:
\:\:L^m_{{\rm cl\,}}(D,T^{*0,q}D\boxtimes(T^{*0,q}D)^*),$$
denote the space of pseudodifferential operators on $D$ of order $m$ type $(\frac{1}{2},\frac{1}{2})$
from sections of $T^{*0,q}X$ to sections of $T^{*0,q}X$ and the space of classical
pseudodifferential operators on $D$ of order $m$ from sections of
$T^{*0,q}X$ to sections of $T^{*0,q}X$ respectively. The classical result of
Calderon and Vaillancourt tells us that any 
$A\in L^m_{\frac12,\frac12}(D,T^{*0,q}X\boxtimes(T^{*0,q}X)^*)$ 
induces for any $s\in\Real$ a continuous operator 
\begin{equation}\label{e-gue190322yyd}
A:H^s_{\rm comp}(D,T^{*0,q}X)\To H^{s-m}_{\rm loc}(D,T^{*0,q}X).
\end{equation}
We refer to H\"{o}rmander~\cite[Chapter 18]{Hor85} for a proof. 
Let $A\in L^m_{\frac{1}{2},\frac{1}{2}}(D,T^{*0,q}X\boxtimes(T^{*0,q}X)^*)$, 
$B\in L^{m_1}_{\frac{1}{2},\frac{1}{2}}(D,T^{*0,q}X\boxtimes(T^{*0,q}X)^*)$, 
where $m, m_1\in\mathbb R$. If $A$ or $B$ is properly supported, then the composition of $A$
and $B$ is well-defined and
\begin{equation}\label{e-gue190326scd}
AB\in L^{m+m_1}_{\frac{1}{2},\frac{1}{2}}(D,T^{*0,q}X\boxtimes(T^{*0,q}X)^*).
\end{equation}
For $m\in\mathbb R$, $\rho, \delta\in\mathbb R$, $0\leq\rho, \delta\leq1$, let 
$$S^{m}_{\rho,\delta}(T^*D,T^{*0,q}X\boxtimes(T^{*0,q}X)^*)$$ be the 
H\"ormander symbol space on $T^*D$ 
with values in $T^{*0,q}X\boxtimes(T^{*0,q}X)^*$ of order $m$ and 
type $(\rho,\delta)$. 
Let 
\[
S^{-\infty}_{\rho,\delta}(T^*D,T^{*0,q}X\boxtimes(T^{*0,q}X)^*):=
\cap_{m\in\mathbb{R}} S^m_{\rho,\delta}(T^*D,T^{*0,q}X
\boxtimes(T^{*0,q}X)^*).
\]
Let $a_j\in S^{m_j}_{\rho,\delta}(T^*D,T^{*0,q}X\boxtimes(T^{*0,q}X)^*)$, 
$j=0,1,2,\ldots\,$, with $m_j\To-\infty$, $j\To\infty$. 
Then there exists $a\in S^{m_0}_{\rho,\delta}(T^*D,T^{*0,q}X\boxtimes(T^{*0,q}X)^*)$ 
 such that 
$$a-\sum^{k-1}_{j=0}a_j\in S^{m_k}_{1,0}(T^*D,T^{*0,q}X\boxtimes(T^{*0,q}X)^*) 
\:\:\text{for $k=1,2,\ldots$\,.}$$ 
In this case we write $$a\sim\sum^{+\infty}_{j=0}a_j\:\:\text{in 
$S^{m_0}_{\rho,\delta}(T^*D,T^{*0,q}X\boxtimes(T^{*0,q}_xX)^*)$}.$$
The symbol $a$ is
unique modulo $S^{-\infty}_{\rho,\delta}(T^*D,T^{*0,q}X\boxtimes(T^{*0,q}X)^*)$.

\begin{defn}\label{d-gue190322yyd}
Let $A\in L^m_{\frac{1}{2},\frac{1}{2}}(D,T^{*0,q}X\boxtimes(T^{*0,q}X)^*)$, 
where $m\in\Real$. We write 
\[\mbox{$A\equiv 0$ near $\Sigma^-\cap T^*D$}\]
if there exists $A'\in L^m_{\frac{1}{2},\frac{1}{2}}(D,T^{*0,q}X\boxtimes(T^{*0,q}X)^*)$ 
with full symbol 
$a(x,\eta)\in S^m_{\frac{1}{2},\frac{1}{2}}(T^*D,T^{*0,q}X\boxtimes(T^{*0,q}X)^*)$ 
such that 
\[\mbox{$A\equiv A'$ on $D$}\]
and $a(x,\eta)$ vanishes in an open neighborhood of $\Sigma^-\cap T^*D$. 
\end{defn}

We now come back to our situation. Let $D\subset X$ be an open coordinate patch with local coordinates $x=(x_1,\ldots,x_{2n-1})$. Assume that the Levi form is non-degenerate of constant signature $(n_-,n_+)$ on $D$. It is clear that 
\[\gamma\ol{\pr}^*_f\tilde P=\ddbar^*_b+\mbox{lower order terms}.\]
From this observation, we can repeat the proof of~\cite[Proposition 6.3]{Hsiao08} with minor change and deduce that 

\begin{thm}\label{t-gue190314mscd}
With the notations and assumptions above, let $q\neq n_-$. We can find properly supported operator $E^{(q)}\in L^{-1}_{\frac{1}{2},\frac{1}{2}}(D,T^{*0,q}X\boxtimes(T^{*0,q}X)^*)$ such that 
\begin{equation}\label{e-gue190314yd}
\Td\Box^{(q)}_bE^{(q)}\equiv I+R\ \ \mbox{on $D$},
\end{equation}
where $R\in L^1_{\frac{1}{2},\frac{1}{2}}(D,T^{*0,q}X\boxtimes(T^{*0,q}X)^*)$ with $R\equiv0$ near $\Sigma^-\cap T^*D$. 
\end{thm}

We can now prove the main result of this section. 

\begin{thm}\label{t-gue190314mscdI}
Let $D\subset X$ be an open coordinate patch with local coordinates $x=(x_1,\ldots,x_{2n-1})$. Assume that the Levi form is non-degenerateof constant signature $(n_-, n_+)$ on $D$. Let $q\neq n_-$. Then, we can find a properly supported operator $G^{(q)}\in L^0_{\frac{1}{2},\frac{1}{2}}(D,T^{*0,q}X\boxtimes(T^{*0,q}X)^*)$ such that 
\begin{equation}\label{e-gue190314ydI}
\mbox{$\Box^{(q)}_-G^{(q)}\equiv I$ on $D$}. 
\end{equation}
\end{thm}

\begin{proof}
Let $A^{(q)}\in L^1_{{\rm cl\,}}(D,T^{*0,q}X\boxtimes(T^{*0,q}X)^*)$ be as in \eqref{e-gue190313adI}. Since $A^{(q)}$ is elliptic near $\Sigma^-$ (see Theorem~\ref{t-gue190313ad}), there is a properly supported 
elliptic pseudodifferential operators $H^{(q)}, H^{(q)}_1\in  L^{-1}_{{\rm cl\,}}(D,T^{*0,q}X\boxtimes(T^{*0,q}X)^*)$ such that 
\begin{equation}\label{e-gue190314d}
\begin{split}
A^{(q)}H^{(q)}-I\equiv0\ \ \mbox{near $\Sigma^-\cap T^*D$},\\
H^{(q)}_1A^{(q)}-I\equiv0\ \ \mbox{near $\Sigma^-\cap T^*D$}.
\end{split}
\end{equation}
From Lemma~\ref{l-gue190312mscd}, \eqref{e-gue190313adI} and \eqref{e-gue190314d}, we have 
$\Td\Box^{(q)}_b\equiv A^{(q)}\Box^{(q)}_-$, $H^{(q)}_-\Td\Box^{(q)}_b\equiv H^{(q)}_1A^{(q)}\Box^{(q)}_-$ and hence 
\begin{equation}\label{e-gue190314dI}
\Box^{(q)}_-\equiv H^{(q)}_1\Td\Box^{(q)}_b\ \ \mbox{near $\Sigma^-\cap T^*D$}.
\end{equation}
Let $E^{(q)}\in L^{-1}_{\frac{1}{2},\frac{1}{2}}(D,T^{*0,q}X\boxtimes(T^{*0,q}X)^*)$ be as in Theorem~\ref{t-gue190314mscd}. 
From \eqref{e-gue190314dI}, we have 
\begin{equation}\label{e-gue190314dII}
\mbox{$\Box^{(q)}_-E^{(q)}A^{(q)}-I\equiv H^{(q)}_1\Td\Box^{(q)}_bE^{(q)}A^{(q)}-I$ near $\Sigma^-\cap T^*D$}.
\end{equation}
From \eqref{e-gue190314yd}, we have $H^{(q)}_1\Td\Box^{(q)}_bE^{(q)}A^{(q)}-I\equiv H^{(q)}_1(I+R)A^{(q)}-I$ and hence 
\begin{equation}\label{e-gue190314dd}
\mbox{$H^{(q)}_1\Td\Box^{(q)}_bE^{(q)}A^{(q)}-I\equiv H^{(q)}_1A^{(q)}-I$ near $\Sigma^-\cap T^*D$}. 
\end{equation}
From \eqref{e-gue190314dd} and \eqref{e-gue190314d}, we get 
\begin{equation}\label{e-gue190314ddd}
\mbox{$H^{(q)}_1\Td\Box^{(q)}_bE^{(q)}A^{(q)}-I\equiv0$ near $\Sigma^-\cap T^*D$}. 
\end{equation}
From \eqref{e-gue190314dII} and \eqref{e-gue190314dd} and \eqref{e-gue190314ddd}, we conclude that 
\[\Box^{(q)}_-E^{(q)}A^{(q)}=I+r,\]
where $r\in L^{1}_{\frac{1}{2},\frac{1}{2}}(D,T^{*0,q}X\boxtimes(T^{*0,q}X)^*)$ with $r\equiv0$ near $\Sigma^-\cap T^*D$. Since $\Box^{(q)}_-$ is elliptic outside $\Sigma^-$, we can 
find a properly supported operator $r_1\in L^{1}_{\frac{1}{2},\frac{1}{2}}(D,T^{*0,q}X\boxtimes(T^{*0,q}X)^*)$ such that 
$\Box^{(q)}_-r_1\equiv-r$ on $D$.
Let $G^{(q)}\in  L^{0}_{\frac{1}{2},\frac{1}{2}}(D,T^{*0,q}X\boxtimes(T^{*0,q}X)^*)$ be a properly supported operator so that 
$G^{(q)}\equiv E^{(q)}A^{(q)}+r_1$ on $D$. Then, $\Box^{(q)}_-G^{(q)}\equiv I$ on $D$. The theorem follows. 
\end{proof}

\section{Parametics for the $\ddbar$-Neumann Laplacian}\label{s-gue190531yyd}

Let $D$ be a local coordinate patch of $X$ with local coordinates $x=(x_1,\ldots,x_{2n-1})$. Then, $\hat x:=(x_1,\ldots,x_{2n-1},\rho)$ are local coordinates of $M'$ defined in 
an open set $U$ of $M'$ with $U\cap X=D$. Until further notice, we work on $U$. 

We introduce some notations. Let $F: \Omega^{0,q}_0(U\cap\ol M)\To\mathscr D'(U\cap\ol M, T^{*0,q}M')$ be a continuous operator. We say that $F$ is properly supported on $U\cap\ol M$ if for every $\chi\in\cC^\infty_0(U\cap\ol M)$, there are $\chi_1\in\cC^\infty_0(U\cap\ol M)$, $\chi_2\in\cC^\infty_0(U\cap\ol M)$, such that $F\chi u=\chi_2Fu$, $\chi Fu=F\chi_1u$, for every $u\in\Omega^{0,q}_0(U\cap\ol M)$. When $F$ is properly supported on $U\cap\ol M$, $F$ can be extended 
continuously to $F: \Omega^{0,q}(U\cap\ol M)\To\mathscr E'(U\cap\ol M, T^{*0,q}M')$. 
We say that $F$ is smoothing away the diagonal on $U\cap\ol M$ if for every $\chi, \chi_1\in\cC^\infty_0(U\cap\ol M)$ with ${\rm Supp\,}\chi\cap{\rm Supp\,}\chi_1=\emptyset$, we have 
\[\chi F\chi_1\equiv0\mod\cC^\infty((U\times U)\cap(\ol M\times\ol M)).\]
We need 

\begin{lem}\label{l-gue190429yyd}
Let $\tau_1\in\cC^\infty(X)$, $\tau\in\cC^\infty(\ol M)$ with ${\rm Supp\,}\tau\cap{\rm Supp\,}\tau_1=\emptyset$. Then, 
\[\tau\tilde P\tau_1\equiv0\mod\cC^\infty(\ol M\times X).\]
\end{lem}

\begin{proof}
From \eqref{e-gue190418yyda}, we have 
\begin{equation}\label{e-gue190429yydI}
\begin{split}
\tau\tilde P\tau_1=(D^{(q)}\Box^{(q)}_f+\tilde P\gamma)\tau\tilde P\tau_1
=D^{(q)}\Box^{(q)}_f\tau\tilde P\tau_1=D^{(q)}[\tau, \Box^{(q)}_f]\tilde P\tau_1.
\end{split}
\end{equation}
By \eqref{e-gue190429yyd} the operator 
\[D^{(q)}[\tau, \Box^{(q)}_f]: H^s(\ol M, T^{*0,q}M')\To H^{s+1}(\ol M, T^{*0,q}M')\]
is continuous, for every $s\in\mathbb Z$. Using this observation, 
\eqref{e-gue190313ad} and \eqref{e-gue190429yydI}, we have 
\[\tau\tilde P\tau_1: H^s(X, T^{*0,q}M')\To H^{s+\frac{3}{2}}(\ol M, T^{*0,q}M')\]
is continuous, for every $s\in\mathbb Z$. We have proved that for any 
$\Td\tau\in\cC^\infty(\ol M)$ with ${\rm Supp\,}\Td\tau\cap{\rm Supp\,}\tau_1=\emptyset$, 
then
\begin{equation}\label{e-gue190429yydII}
\Td\tau\tilde P\tau_1: H^s(X, T^{*0,q}M')\To H^{s+\frac{3}{2}}(\ol M, T^{*0,q}M')
\end{equation}
is continuous, for every $s\in\mathbb Z$.

Let $\Td\tau\in\cC^\infty(\ol M)$ with $\Td\tau=1$ near ${\rm Supp\,}\tau$ and ${\rm Supp\,}\Td\tau\cap{\rm Supp\,}\tau_1=\emptyset$. From \eqref{e-gue190429yydI}, we have 
\begin{equation}\label{e-gue190429yydIII}
\tau\tilde P\tau_1=D^{(q)}[\tau, \Box^{(q)}_f]\Td\tau\tilde P\tau_1.
\end{equation}
From \eqref{e-gue190429yydIII}, \eqref{e-gue190429yydII} and \eqref{e-gue190418yyda}, we have
\[\tau\tilde P\tau_1: H^s(X, T^{*0,q}M')\To H^{s+\frac{5}{2}}(\ol M, T^{*0,q}M')\]
is continuous, for every $s\in\mathbb Z$. Continuing in this way, we conclude that 
\[\tau\tilde P\tau_1: H^s(X, T^{*0,q}M')\To H^{s+\frac{2N+1}{2}}(\ol M, T^{*0,q}M')\]
is continuous, for every $s\in\mathbb Z$ and $N>0$. The lemma follows. 
\end{proof}

From Lemma~\ref{l-gue190429yyd}, we get 

\begin{lem}\label{l-gue190429yydI}
Let $\tau_1\in\cC^\infty(X)$, $\tau\in\cC^\infty(\ol M)$ with ${\rm Supp\,}\tau\cap{\rm Supp\,}\tau_1=\emptyset$. Then, 
\[\tau_1\tilde P^*\tau\equiv0\mod\cC^\infty(X\times\ol M).\]
Recall that $\tilde P^*$ is given by \eqref{e-gue190515yyda}. 
\end{lem}

We come back to our situation. Until further notice, we assume that the Levi form is non-degenerate of constant signature $(n_-,n_+)$ on $D$. 

\begin{thm}\label{t-gue190321myyd}
With the assumptions and notations above, let $q\neq n_-$. We can find properly supported continuous operator on $U\cap\ol M$: 
\[N^{(q)}: H^s_{{\rm loc\,}}(U\cap\ol M, T^{*0,q}M')\To H^{s+1}_{{\rm loc\,}}(U\cap\ol M, T^{*0,q}M'),\ \ \mbox{for every $s\in\mathbb Z$},\]
such that  
\begin{equation}\label{e-gue190326yyd}
\mbox{$(\ddbar\rho)^{\wedge,*}\gamma N^{(q)}u|_D=0$, for every $u\in\Omega^{0,q}(U\cap\ol M)$}, 
\end{equation}
\begin{equation}\label{e-gue190326yydI}
\mbox{$(\ddbar\rho)^{\wedge,*}\gamma\ddbar N^{(q)}u|_D=0$, for every $u\in\Omega^{0,q}(U\cap\ol M)$}
\end{equation}
and 
\begin{equation}\label{e-gue190326yydII}
\mbox{$\Box^{(q)}_fN^{(q)}=I+F^{(q)}$ on $\Omega^{0,q}_0(U\cap M)$},
\end{equation}
where $F^{(q)}: \mathscr D'(U\cap M)\To\Omega^{0,q}(U\cap M)$ is a properly supported continuous operator on $U\cap\ol M$ with $F^{(q)}\equiv0\mod\cC^\infty((U\times U)\cap(\ol M\times\ol M))$.
\end{thm}

\begin{proof}
Since $\Box^{(q)}_f$ is an elliptic operator on $M'$, we can find a properly supported continuous operator
\[N^{(q)}_1: H^s_{{\rm loc\,}}(U\cap\ol M, T^{*0,q}M')\To H^{s+2}_{{\rm loc\,}}(U\cap\ol M, T^{*0,q}M'),\ \ \mbox{for every $s\in\mathbb Z$},\]
such that $N^{(q)}_1$ is smoothing away the diagonal on $U\cap\ol M$ and 
\begin{equation}\label{e-gue190326yydIII}
\mbox{$\Box^{(q)}_fN^{(q)}_1=I+F_1$ on $\Omega^{0,q}_0(U\cap M')$},
\end{equation}
where $F_1\equiv0\mod\cC^\infty((U\times U)\cap(\ol M\times\ol M))$.
Consider for every $s\in\mathbb Z$,
\[N^{(q)}_2:=N^{(q)}_1-\tilde P\gamma N^{(q)}_1
:H^s_{{\rm comp\,}}(U\cap\ol M, T^{*0,q}M')
\To H^{s+2}_{{\rm loc\,}}(U\cap\ol M, T^{*0,q}M').\]
From \eqref{e-gue190312mscd} and \eqref{e-gue190326yydIII}, we see that 
\begin{equation}\label{e-gue190326ycd}
\mbox{$\gamma N^{(q)}_2u|_D=0$, for every $u\in\Omega^{0,q}_0(U\cap\ol M)$}, 
\end{equation}
and 
\begin{equation}\label{e-gue190326yscdI}
\mbox{$\Box^{(q)}_fN^{(q)}_2=I+F_2$ on $\Omega^{0,q}_0(U\cap M')$}, 
\end{equation}
where $F_2\equiv0 \mod\cC^\infty((U\times U)\cap(\ol M\times\ol M))$. 
From Lemma~\ref{l-gue190429yyd}, it is not difficult to check that $N^{(q)}_2$ is smoothing away the diagonal on $U\cap\ol M$. Hence, we can find a 
properly supported continuous operator
\[N^{(q)}_3: H^s_{{\rm loc\,}}(U\cap\ol M, T^{*0,q}M')\To H^{s+2}_{{\rm loc\,}}(U\cap\ol M, T^{*0,q}M'),\ \ \mbox{for every $s\in\mathbb Z$},\]
such that 
\begin{equation}\label{e-gue190326yscdII}
N^{(q)}_3\equiv N^{(q)}_2\mod\cC^\infty((U\times U)\cap(\ol M\times\ol M)).
\end{equation}
From \eqref{e-gue190326ycd} and \eqref{e-gue190326yscdII}, we conclude that 
\begin{equation}\label{e-gue190326yscdIII}
\gamma N^{(q)}_3\equiv 0\mod\cC^\infty((U\times U)\cap(X\times\ol M)).
\end{equation}
Let $E^{(q)}\equiv0\mod\cC^\infty((U\times U)\cap(\ol M\times\ol M))$ be any smoothing properly supported extension of $\gamma N^{(q)}_3$. That is, 
$\gamma E^{(q)}u|D=\gamma N^{(q)}_3u|_D$, for every $u\in\Omega^{0,q}(U\cap\ol M)$ and $E^{(q)}$ is properly supported on $U\cap\ol M$. Let 
\begin{equation}\label{e-gue190326syd}
\begin{split}
N^{(q)}_4&:=N^{(q)}_3-E^{(q)}\\
&: H^s_{{\rm loc\,}}(U\cap\ol M, T^{*0,q}M')\To H^{s+2}_{{\rm loc\,}}(U\cap\ol M, T^{*0,q}M'),\ \ \mbox{for every $s\in\mathbb Z$}.
\end{split}
\end{equation}
Then $N^{(q)}_4$ is properly supported on $U\cap\ol M$ and 
\begin{equation}\label{e-gue190326sydI}
\begin{split}
&\mbox{$\gamma N^{(q)}_4u|_D=0$, for every $u\in\Omega^{0,q}(U\cap\ol M)$}, \\
&\mbox{$\Box^{(q)}_fN^{(q)}_4=I+F_3$ on $\Omega^{0,q}_0(U\cap M)$}, 
\end{split}
\end{equation}
where $F_3\equiv0 \mod\cC^\infty((U\times U)\cap(\ol M\times\ol M))$. 
Let $G^{(q)}\in L^0_{\frac{1}{2},\frac{1}{2}}(D,T^{*0,q}X\boxtimes(T^{*0,q}X)^*)$ be as in Theorem~\ref{t-gue190314mscdI}. Put 
\begin{equation}\label{e-gue190326sydII}
\begin{split}
N^{(q)}_5&:=N^{(q)}_4-\tilde PG^{(q)}(\ddbar\rho)^{\wedge,*}\gamma\ddbar N^{(q)}_4\\
&: H^s_{{\rm comp\,}}(U\cap\ol M, T^{*0,q}M')\To H^{s+1}_{{\rm loc\,}}(U\cap\ol M, T^{*0,q}M'),\ \ \mbox{for every $s\in\mathbb Z$}.
\end{split}
\end{equation}
From Theorem~\ref{t-gue190314mscdI}, \eqref{e-gue190314ydI} and \eqref{e-gue190326sydI}, we can check that 
\begin{equation}\label{e-gue190326sydIz}
\begin{split}
&\mbox{$(\ddbar\rho)^{\wedge,*}\gamma N^{(q)}_5u|_D=0$, for every $u\in\Omega^{0,q}(U\cap\ol M)$}, \\
&\mbox{$(\ddbar\rho)^{\wedge,*}\gamma\ddbar N^{(q)}_5\equiv0\mod\cC^\infty((U\times U)\cap(X\times\ol M))$},\\
&\mbox{$\Box^{(q)}_fN^{(q)}_5=I+F_4$ on $\Omega^{0,q}_0(U\cap M)$}, 
\end{split}
\end{equation}
where $F_4\equiv0 \mod\cC^\infty((U\times U)\cap(\ol M\times\ol M))$.
It is not difficult to check that $N^{(q)}_5$ is smoothing away the diagonal on $U\cap\ol M$. Hence, we can find a 
properly supported continuous operator
\[N^{(q)}_6: H^s_{{\rm loc\,}}(U\cap\ol M, T^{*0,q}M')\To H^{s+1}_{{\rm loc\,}}(U\cap\ol M, T^{*0,q}M'),\ \ \mbox{for every $s\in\mathbb Z$},\]
such that 
\begin{equation}\label{e-gue190326sydIIz}
N^{(q)}_5\equiv N^{(q)}_6\mod\cC^\infty((U\times U)\cap(\ol M\times\ol M)).
\end{equation}
Let $R^{(q)}\equiv0\mod\cC^\infty((U\times U)\cap(\ol M\times\ol M))$ 
be any smoothing properly supported extension of 
$2(\ddbar\rho)^\wedge(\ddbar\rho)^{\wedge,*}\gamma N^{(q)}_6$.
For every $s\in\mathbb Z$ put 
\begin{equation}\label{e-gue190327syd}
N^{(q)}_7:=N^{(q)}_6-R^{(q)}: 
H^s_{{\rm loc\,}}(U\cap\ol M, T^{*0,q}M')\To 
H^{s+1}_{{\rm loc\,}}(U\cap\ol M, T^{*0,q}M')\,.
\end{equation}
From \eqref{e-gue190313scsI}, we have 
\begin{equation}\label{e-gue190327sydI}
\begin{split}
(\ddbar\rho)^{\wedge,*}\gamma N^{(q)}_7&=(\ddbar\rho)^{\wedge,*}\gamma N^{(q)}_6-(\ddbar\rho)^{\wedge,*}\gamma R^{(q)}\\
&=(\ddbar\rho)^{\wedge,*}\gamma N^{(q)}_6-2(\ddbar\rho)^{\wedge,*}(\ddbar\rho)^\wedge(\ddbar\rho)^{\wedge,*}\gamma N^{(q)}_6\\
&=(\ddbar\rho)^{\wedge,*}\gamma N^{(q)}_6-(\ddbar\rho)^{\wedge,*}\gamma N^{(q)}_6=0.
\end{split}
\end{equation}
From \eqref{e-gue190327sydI} and \eqref{e-gue190326sydIz}, we have 
\begin{equation}\label{e-gue190327sydII}
\begin{split}
&\mbox{$(\ddbar\rho)^{\wedge,*}\gamma N^{(q)}_7u|_D=0$, for every $u\in\Omega^{0,q}(U\cap\ol M)$}, \\
&\mbox{$(\ddbar\rho)^{\wedge,*}\gamma\ddbar N^{(q)}_7\equiv0\mod\cC^\infty((U\times U)\cap(X\times\ol M))$},\\
&\mbox{$\Box^{(q)}_fN^{(q)}_7=I+F_5$ on $\Omega^{0,q}_0(U\cap M)$},
\end{split}
\end{equation}
where  $F_5\equiv0 \mod\cC^\infty((U\times U)\cap(\ol M\times\ol M))$.
Let $J^{(q)}\equiv0\mod\cC^\infty((U\times U)\cap(\ol M\times\ol M))$ be any smoothing properly supported extension of $(\ddbar\rho)^{\wedge,*}\gamma\ddbar N^{(q)}_7$. 
Let $\chi\in\cC^\infty(]-\varepsilon,\varepsilon[)$ with $\chi\equiv1$ near $0$, where $\varepsilon>0$ is a small constant. For for every $s\in\mathbb Z$ put 
\begin{equation}\label{e-gue190327sydIII}
\begin{split}
N^{(q)}:=N^{(q)}_7-2\chi(\rho)\rho J^{(q)}: 
H^s_{{\rm loc\,}}(U\cap\ol M, T^{*0,q}M')
\To H^{s+1}_{{\rm loc\,}}(U\cap\ol M, T^{*0,q}M').
\end{split}
\end{equation}
It is not difficult to see that $N^{(q)}$ is properly supported on $U\cap\ol M$, 
\[N^{(q)}\equiv N^{(q)}_7\mod\cC^\infty((U\times U)\cap(\ol M\times\ol M))\]
and 
\[\mbox{$(\ddbar\rho)^{\wedge,*}\gamma N^{(q)}u|_D=(\ddbar\rho)^{\wedge,*}\gamma N^{(q)}_7u|_D=0$, for every $u\in\Omega^{0,q}(U\cap\ol M)$}.\]
From \eqref{e-gue190313scsI}, we have for every $u\in\Omega^{0,q}(U\cap M)$, 
\begin{equation}\label{e-gue190327scd}
\begin{split}
(\ddbar\rho)^{\wedge,*}\gamma\ddbar N^{(q)}u|_D&=(\ddbar\rho)^{\wedge,*}\gamma\ddbar N^{(q)}_7u|_D-2(\ddbar\rho)^{\wedge,*}(\ddbar\rho)^\wedge\gamma J^{(q)}u|_D\\
&=(\ddbar\rho)^{\wedge,*}\gamma\ddbar N^{(q)}_7u|_D-2(\ddbar\rho)^{\wedge,*}(\ddbar\rho)^\wedge(\ddbar\rho)^{\wedge,*}\gamma\ddbar N^{(q)}_7u|_D\\
&=(\ddbar\rho)^{\wedge,*}\gamma\ddbar N^{(q)}_7u|_D-(\ddbar\rho)^{\wedge,*}\gamma\ddbar N^{(q)}_7u|_D=0.
\end{split}
\end{equation}
We have proved that $N^{(q)}$ satisfies \eqref{e-gue190326yyd}, \eqref{e-gue190326yydI} and \eqref{e-gue190326yydII}. 
The theorem follows. 
\end{proof}

Let $N^{(q)}$ be as in Theorem~\ref{t-gue190321myyd} and let $(N^{(q)})^*: \Omega^{0,q}_0(U\cap M)\To\mathscr D'(U\cap M,T^{*0,q}M')$ be the formal adjoint of $N^{(q)}$ given by 
\[\mbox{$(\,(N^{(q)})^*u\,|\,v\,)_M=(\,u\,|\,N^{(q)}v\,)_M$, for every $u, v\in\Omega^{0,q}_0(U\cap M)$}.\]

\begin{lem}\label{l-gue190403yyd}
With the assumptions and notations used above, we have 
\begin{equation}\label{e-gue190403yyd}
\mbox{$(N^{(q)})^*u=N^{(q)}u+H^{(q)}u$, for every $u\in\Omega^{0,q}_0(U\cap M)$}, 
\end{equation}
where  $H^{(q)}: \mathscr D'(U\cap M,T^{*0,q}M')\To\Omega^{0,q}(U\cap M)$ is a properly supported continuous operator on $U\cap\ol M$ with $H^{(q)}\equiv0\mod\cC^\infty((U\times U)\cap(\ol M\times\ol M))$.
\end{lem}

\begin{proof}
Let $u, v\in\Omega^{0,q}_0(U\cap M)$. From \eqref{e-gue190326yydII}, we have 
\begin{equation}\label{e-gue190403yydI}
\begin{split}
&(\,(N^{(q)})^*u\,|\,v\,)_M=(\,(N^{(q)})^*(\Box^{(q)}_fN^{(q)}-F^{(q)})u\,|\,v\,)_M\\
&=(\,\Box^{(q)}_fN^{(q)}u\,|\,N^{(q)}v\,)_M-(\,F^{(q)}u\,|\,N^{(q)}v\,)_M.
\end{split}
\end{equation}
From \eqref{e-gue190326yyd} and \eqref{e-gue190326yydI}, we can integrate by parts and get 
\begin{equation}\label{e-gue190403yydII}
\begin{split}
&(\,\Box^{(q)}_fN^{(q)}u\,|\,N^{(q)}v\,)_M=(\,N^{(q)}u\,|\,\Box^{(q)}_fN^{(q)}v\,)_M\\
&=(\,N^{(q)}u\,|\,(I+F^{(q)})v\,)_M,\ \ \mbox{here we used \eqref{e-gue190326yydII}}. 
\end{split}
\end{equation}
From \eqref{e-gue190403yydI} and \eqref{e-gue190403yydII}, we deduce that 
\begin{equation}\label{e-gue190403yydIII}
(\,(N^{(q)})^*u\,|\,v\,)_M=(\,(N^{(q)}+(F^{(q)})^*)u\,|\,v\,)_M-(\,u\,|\,(F^{(q)})^*N^{(q)}v\,)_M,
\end{equation}
where $(F^{(q)})^* : \Omega^{0,q}_0(U\cap M)\To\mathscr D'(U\cap M,T^{*0,q}M')$ is the formal adjoint of $F^{(q)}$ with respect to $(\,\cdot\,|\,\cdot\,)_M$.  
It is clear that $(F^{(q)})^*$ is a properly supported continuous operator on $U\cap\ol M$ with $(F^{(q)})^*\equiv0\mod\cC^\infty((U\times U)\cap(\ol M\times\ol M))$.

It is not difficult to check that $(F^{(q)})^*N^{(q)}$ is a properly supported continuous operator on $U\cap\ol M$ with $(F^{(q)})^*N^{(q)}\equiv0\mod\cC^\infty((U\times U)\cap(\ol M\times\ol M))$. Let $((F^{(q)})^*N^{(q)})^*: \Omega^{0,q}_0(U\cap M)\To\mathscr D'(U\cap M,T^{*0,q}M')$ is the formal adjoint of $(F^{(q)})^*N^{(q)}$ with respect to $(\,\cdot\,|\,\cdot\,)_M$. 
Then $((F^{(q)})^*N^{(q)})^*$ is a properly supported continuous operator on $U\cap\ol M$ with $((F^{(q)})^*N^{(q)})^*\equiv0\mod\cC^\infty((U\times U)\cap(\ol M\times\ol M))$. From this observation and \eqref{e-gue190403yydIII}, we have
\[(\,(N^{(q)})^*u\,|\,v\,)_M=(\,\bigr(N^{(q)}+(F^{(q)})^*-((F^{(q)})^*N^{(q)})^*\bigr)u\,|\,v\,)_M.\]
We get \eqref{e-gue190403yyd}. 
\end{proof}

From \eqref{e-gue190403yyd}, we can extend $(N^{(q)})^*$ to
\[(N^{(q)})^*: L^2_{{\rm loc\,}}(U\cap\ol M, T^{*0,q}M')\To L^2_{{\rm loc\,}}(U\cap\ol M, T^{*0,q}M'),\ \ \mbox{for every $s\in\mathbb Z$},\]
as a properly supported continuous operator on $U\cap\ol M$ and we have 
\begin{equation}\label{e-gue190404yyd}
\mbox{$(N^{(q)})^*u=N^{(q)}u+H^{(q)}u$, for every $u\in L^2_{{\rm loc\,}}(U\cap M,T^{*0,q}M')$}, 
\end{equation}
where $H^{(q)}$ is as in \eqref{e-gue190403yyd}. Moreover, for every $g\in L^2_{{\rm comp\,}}(U\cap M, T^{*0,q}M')$ and every $u\in L^2_{{\rm loc\,}}(U\cap M,T^{*0,q}M')$, we have 
\begin{equation}\label{e-gue190404yyda}
(\,(N^{(q)})^*u\,|\,g)_M=(\,u\,|\,N^{(q)}g\,)_M,\ \ (\,(N^{(q)})^*g\,|\,u)_M=(\,g\,|\,N^{(q)}u\,)_M.
\end{equation}

We can now improve Theorem~\ref{t-gue190321myyd}.

\begin{thm}\label{t-gue190403yydI}
With the assumptions and notations used above, let $q\neq n_-$. We have 
\begin{equation}\label{e-gue190403yydIz}
\mbox{$N^{(q)}\Box^{(q)}u=u+F^{(q)}_1u$ on $U\cap M$, for every $u\in{\rm Dom\,}\Box^{(q)}$},
\end{equation} 
\begin{equation}\label{e-gue190403yydq}
\mbox{$\Box^{(q)}_fN^{(q)}u=u+F^{(q)}_2u$ on $U\cap M$, for every $u\in\Omega^{0,q}(U\cap\ol M)$}
\end{equation} 
and 
\begin{equation}\label{e-gue190419sydh}
\mbox{$\Box^{(q)}_fN^{(q)}u=u+F^{(q)}_2u$ on $U\cap M$, for every $u\in H^s_{{\rm loc\,}}(U\cap\ol M,T^{*0,q}M')$, $s\in\mathbb Z$},
\end{equation}
where $F^{(q)}_1, F^{(q)}_2: \mathscr D'(U\cap M)\To\Omega^{0,q}(U\cap M)$ are properly supported continuous operators on $U\cap\ol M$ with $F^{(q)}_1\equiv0\mod\cC^\infty((U\times U)\cap(\ol M\times\ol M))$, $F^{(q)}_2\equiv0\mod\cC^\infty((U\times U)\cap(\ol M\times\ol M))$. 
\end{thm}

\begin{rem}\label{r-gue190404yyd}
Let $u\in{\rm Dom\,}\Box^{(q)}$. The equation \eqref{e-gue190403yydIz} means that for every $g\in\Omega^{0,q}_0(U\cap M)$, we have 
\begin{equation}\label{e-gue190404syd}
(\,N^{(q)}\Box^{(q)}u\,|\,g\,)_M=(\,u+F^{(q)}_1u\,|\,g\,)_M.
\end{equation}
Since $N^{(q)}$ and $F^{(q)}_1$ are properly supported operators on $U\cap\ol M$, \eqref{e-gue190404syd} makes sense. For $u\in\Omega^{0,q}(U\cap M)$, the equation \eqref{e-gue190403yydq} means that for every $g\in\Omega^{0,q}_0(U\cap M)$, we have 
\begin{equation}\label{e-gue190404sydI}
(\,\Box^{(q)}_fN^{(q)}u\,|\,g\,)_M=(\,u+F^{(q)}_2u\,|\,g\,)_M.
\end{equation}
Similarly, the meaning of \eqref{e-gue190419sydh} is the same as \eqref{e-gue190403yydq}
\end{rem}

\begin{proof}[Proof of Theorem~\ref{t-gue190403yydI}]
Let $u\in{\rm Dom\,}\Box^{(q)}$. Then, $\Box^{(q)}u\in L^2_{(0,q)}(M)\subset L^2_{{\rm loc\,}}(U\cap\ol M,T^{*0,q}M')$. Let $g\in\Omega^{0,q}_0(U\cap M)$. 
From \eqref{e-gue190404yyd} and \eqref{e-gue190404yyda}, we have 
\begin{equation}\label{e-gue190404sydII}
\begin{split}
&(\,N^{(q)}\Box^{(q)}u\,|\,g)_M=(\,(N^{(q)})^*-H^{(q)})\Box^{(q)}u\,|\,g\,)_M\\
&=(\,\Box^{(q)}u\,|\,N^{(q)}g\,)_M-(\,H^{(q)}\Box^{(q)}_fu\,|\,g\,)_M.
\end{split}
\end{equation}
Since $u\in{\rm Dom\,}\Box^{(q)}$ and by \eqref{e-gue190326yyd}, \eqref{e-gue190326yydI}, $N^{(q)}g\in{\rm Dom\,}\Box^{(q)}$, we can integrate by parts and get 
\begin{equation}\label{e-gue190404sydIII}
(\,\Box^{(q)}u\,|\,N^{(q)}g\,)_M=(\,u\,|\,\Box^{(q)}N^{(q)}g\,)_M=(\,u\,|\,(I+F^{(q)})g\,)_M=(u+(F^{(q)})^*u\,|\,g\,)_M,
\end{equation}
where $F^{(q)}$ is as in \eqref{e-gue190326yydII} and $(F^{(q)})^*$ is the formal adjoint of $F^{(q)}$. From \eqref{e-gue190404sydII} and \eqref{e-gue190404sydIII}, we have 
\begin{equation}\label{e-gue190404ycd}
(\,N^{(q)}\Box^{(q)}u\,|\,g)_M=(u+(F^{(q)})^*u-H^{(q)}\Box^{(q)}_fu\,|\,g\,)_M.
\end{equation}
From \eqref{e-gue190404ycd}, we get \eqref{e-gue190403yydIz} with $F^{(q)}_1=(F^{(q)})^*-H^{(q)}\Box^{(q)}_f$. 

Let $u\in\Omega^{0,q}(U\cap\ol M)$ and let $g\in\Omega^{0,q}_0(U\cap M)$. From \eqref{e-gue190404yyd}, \eqref{e-gue190404yyda},
\eqref{e-gue190403yydIz}, and note that $N^{(q)}$ is properly supported on $U\cap\ol M$, we have 
\begin{equation}\label{e-gue190404ycdI}
\begin{split}
&(\,\Box^{(q)}_fN^{(q)}u\,|\,g\,)_M=(\,N^{(q)}u\,|\,\Box^{(q)}_fg\,)_M=(\,u\,|\,(N^{(q)})^*\Box^{(q)}_fg\,)_M\\
&=(\,u\,|\,(N^{(q)}+H^{(q)})\Box^{(q)}_fg\,)_M=(\,u\,|\,g+F^{(q)}_1g+H^{(q)}\Box^{(q)}_fg\,)_M\\
&=(\,u+(F^{(q)}_1)^*u+(H^{(q)}\Box^{(q)}_f)^*u\,|\,g\,)_M, 
\end{split}
\end{equation}
where $(F^{(q)}_1)^*$ and $(H^{(q)}\Box^{(q)}_f)^*$ are the formal adjoints of $F^{(q)}_1$ and $H^{(q)}\Box^{(q)}_f$ respectively. From \eqref{e-gue190404ycdI}, we get 
\eqref{e-gue190403yydq} with $F^{(q)}_2=(F^{(q)}_1)^*+(H^{(q)}\Box^{(q)}_f)^*$. 

Let $u\in H^s_{{\rm loc\,}}(U\cap\ol M,T^{*0,q}M')$, $s\in\mathbb Z$. We can take $u_j\in\Omega^{0,q}_0(U\cap\ol M)$, $j=1,2,\ldots$, such that 
$u_j\To u$ in $H^s_{{\rm loc\,}}(U\cap\ol M,T^{*0,q}M')$ as $j\To\infty$. By \eqref{e-gue190403yydq}, we see that 
\begin{equation}\label{e-gue190419scda}
\Box^{(q)}_fN^{(q)}u_j=u_j+F^{(q)}_2u_j,\ \ \mbox{for every $j=1,2,\ldots$}.
\end{equation}
Note that $\Box^{(q)}_fN^{(q)}u_j\To\Box^{(q)}_fN^{(q)}u$ 
in $\mathscr D'(U\cap M,T^{*0,q}M')$ as $j\To+\infty$. 
From this observation and \eqref{e-gue190419scda}, we get 
\eqref{e-gue190419sydh}. 
\end{proof}

From Theorem~\ref{t-gue190321myyd} and Theorem~\ref{t-gue190403yydI}, 
we get the main result of this section:

\begin{thm}\label{t-gue190531syd}
Let $U$ be an open set of $M'$ with $U\cap X\neq\emptyset$. 
Suppose that the Levi form is non-degenerate of constant signature 
$(n_-, n_+)$ on $U\cap X$. 
Let $q\neq n_-$. We can find properly supported continuous operators on 
$U\cap\ol M$:
\[N^{(q)}: H^s_{{\rm loc\,}}(U\cap\ol M, T^{*0,q}M')
\To H^{s+1}_{{\rm loc\,}}(U\cap\ol M, T^{*0,q}M'),\ \ 
\mbox{for every $s\in\mathbb Z$},\]
such that  \eqref{e-gue190326yyd}, \eqref{e-gue190326yydI}, 
\eqref{e-gue190403yyd}, \eqref{e-gue190404yyda}, 
\eqref{e-gue190403yydIz}, \eqref{e-gue190403yydq} and \eqref{e-gue190419sydh} hold. 
\end{thm}

\section{Microlocal Hodge decomposition theorems for the 
$\ddbar$-Neumann Laplacian}\label{s-gue190321myyd}

We recall the following~(see page 13 of \cite{FK72} for the proof).

\begin{lem} \label{l-gue190327ycd}
For all $f\in\Omega^{0,q}(\ol M)$, $g\in\Omega^{0,q+1}(\ol M)$, we have 
\begin{equation} \label{e-gue190328yyd}
(\,g\,|\,\ddbar f\,)_M=(\,\ol{\pr}^*_fg\,|\,f\,)_M+
(\,(\ddbar\rho)^{\wedge,*}\gamma g\,|\, \gamma f\,)_X.
\end{equation}
\end{lem}

Let $D$ be a local coordinate patch of $X$ with local coordinates 
$x=(x_1,\ldots,x_{2n-1})$. Then, $\hat x:=(x_1,\ldots,x_{2n-1},\rho)$ 
are local coordinates of $M'$ defined in 
an open set $U$ of $M'$ with $U\cap X=D$. Until further notice, we work on $U$. 
\begin{lem}\label{l-gue190327ycdz}
Let $u\in\Omega^{0,q}(U\cap\ol M)$. Assume that $\ddbar\rho^{\wedge,*}\gamma u|_D=0$. Then, 
\begin{equation}\label{e-gue190327ycdII}
(\ddbar\rho)^{\wedge,*}\gamma\ol{\pr}^*_fu|_D=0.
\end{equation}
\end{lem}

\begin{proof}
Let $g\in\Omega^{0,q-2}_0(U\cap\ol M)$. From \eqref{e-gue190328yyd}, we have 
\begin{equation}\label{e-gue190327ccd}
\begin{split}
&(\,\ol{\pr}^*_fu\,|\,\ddbar g\,)_M\\
&=(\,(\ol{\pr}^*_f)^2u\,|\,g\,)_M+
(\,(\ddbar\rho)^{\wedge,*}\gamma\ol{\pr}^*_fu\,|\,\gamma g\,)_X\\
&=(\,(\ddbar\rho)^{\wedge,*}\gamma\ol{\pr}^*_fu\,|\,\gamma g\,)_X.
\end{split}
\end{equation}
On the other hand, from \eqref{e-gue190328yyd} again, we have 
\begin{equation}\label{e-gue190327ccdI}
\begin{split}
&0=(\,u\,|\,\ddbar^2g\,)_M=(\,\ol{\pr}^*_fu\,|\,\ddbar g\,)_M+
(\,(\ddbar\rho)^{\wedge,*}\gamma u\,|\,\gamma\ddbar g\,)_X\\
&=(\,\ol{\pr}^*_fu\,|\,\ddbar g\,)_M
\end{split}
\end{equation}
since $(\ddbar\rho)^{\wedge,*}\gamma u|_D=0$. 
From \eqref{e-gue190327ccd} and \eqref{e-gue190327ccdI}, we conclude that 
\[(\,(\ddbar\rho)^{\wedge,*}\gamma\ol{\pr}^*_fu\,|\,\gamma g\,)_X=0.\] 
Since $g$ is arbitrary, 
$(\ddbar\rho)^{\wedge,*}\gamma\ol{\pr}^*_fu|_D=0$.
\end{proof}

We now assume that the Levi form is non-degenerate of 
constant signature $(n_-,n_+)$ on $D=U\cap X$. 
Let $q=n_-$. Let $N^{(q+1)}$ and $N^{(q-1)}$ be as in Theorem~\ref{t-gue190531syd}. Put 
\begin{equation}\label{e-gue190327ycd}
\begin{split}
\hat N^{(q)}&:=\ol{\pr}^*_f(N^{(q+1)})^2\ddbar+
\ddbar(N^{(q-1)})^2\ol{\pr}^*_f\\
&: H^s_{{\rm loc\,}}(U\cap\ol M, T^{*0,q}M')
\To H^{s}_{{\rm loc\,}}(U\cap\ol M, T^{*0,q}M'),\ \ 
\mbox{for every $s\in\mathbb Z$}.
\end{split}
\end{equation}
Put 
\begin{equation}\label{e-gue190520yyd}
A^{0,q}(U\cap\ol M):=\set{u\in\Omega^{0,q}(U\cap\ol M);\,  
(\ddbar\rho)^{\wedge,*}\gamma u|_D=0}.
\end{equation}
We notice that if $u\in{\rm Dom\,}\ol{\pr}^*\cap\Omega^{0,q}(U\cap\ol M)$, 
then $u\in A^{0,q}(U\cap\ol M)$. We define 
\begin{equation}\label{e-gue190327ycdI}
\begin{split}
\hat\Pi^{(q)}&:=I-\ol{\pr}^*_fN^{(q+1)}\ddbar-
\ddbar N^{(q-1)}\ol{\pr}^*_f: A^{0,q}(U\cap\ol M)\To\Omega^{0,q}(U\cap\ol M),\\
&u\in A^{0,q}(U\cap\ol M)\To 
\Bigr(I-\ol{\pr}^*_fN^{(q+1)}\ddbar-\ddbar N^{(q-1)}\ol{\pr}^*_f\Bigr)u
\in\Omega^{0,q}(U\cap\ol M).
\end{split}
\end{equation}

\begin{thm}\label{t-gue190404ycd}
With the assumptions and notations above, let $q=n_-$. We have 
\begin{equation}\label{e-gue190404scd}
\mbox{$\hat N^{(q)}u\in A^{0,q}(U\cap\ol M)$, for every $u\in\Omega^{0,q}(U\cap\ol M)$}, 
\end{equation}
\begin{equation}\label{e-gue190404scdI}
\mbox{$(\ddbar\rho)^{\wedge,*}\gamma\ddbar\hat N^{(q)}u=
H^{(q)}_1u$, for every $u\in\Omega^{0,q}(U\cap\ol M)$}, 
\end{equation}
\begin{equation}\label{e-gue190404scdII}
\mbox{$\hat\Pi^{(q)}u\in A^{0,q}(U\cap\ol M)$, for every $u\in A^{0,q}(U\cap\ol M)$}, 
\end{equation}
\begin{equation}\label{e-gue190404scdIII}
\mbox{$(\ddbar\rho)^{\wedge,*}\gamma\ddbar\hat\Pi^{(q)}u
=H^{(q)}_2u$, for every $u\in A^{0,q}(U\cap\ol M)$},
\end{equation}
\begin{equation}\label{e-gue190404scdz}
\mbox{$\Box^{(q)}_f\hat N^{(q)}u+\hat\Pi^{(q)}u=
u+H^{(q)}_3u$, for every $u\in A^{0,q}(U\cap\ol M)$}
\end{equation}
and 
\begin{equation}\label{e-gue190520ycd}
\mbox{$\ddbar\hat\Pi^{(q)}u=H^{(q)}_4u$, for every 
$u\in A^{0,q}(U\cap\ol M)\cap\Omega^{0,q}_0(U\cap\ol M)$},
\end{equation}
\begin{equation}\label{e-gue190520ycdI}
\mbox{$\ol{\pr}^*_f\hat\Pi^{(q)}u=H^{(q)}_5u$ 
for every $u\in A^{0,q}(U\cap\ol M)\cap\Omega^{0,q}_0(U\cap\ol M)$},
\end{equation}
where $H^{(q)}_1\equiv0\mod\cC^\infty((U\times U)
\cap(X\times\ol M))$, $H^{(q)}_2\equiv0\mod\cC^\infty((U\times U)\cap(X\times\ol M))$, 
$H^{(q)}_3\equiv0\mod\cC^\infty((U\times U)\cap(\ol M\times\ol M))$, 
$H^{(q)}_4\equiv0\mod\cC^\infty((U\times U)\cap(\ol M\times\ol M))$ and 
$H^{(q)}_5\equiv0\mod\cC^\infty((U\times U)\cap(\ol M\times\ol M))$, 
$H^{(q)}_1, H^{(q)}_2, H^{(q)}_3, H^{(q)}_4, H^{(q)}_5$ 
are properly supported on $U\cap\ol M$. 
\end{thm}

\begin{proof}
From \eqref{e-gue190326yyd}, \eqref{e-gue190326yydI}, 
Lemma~\ref{l-gue190327ycdz} and the definitions of $\hat N^{(q)}$, 
$\hat\Pi^{(q)}$, we get \eqref{e-gue190404scd} and \eqref{e-gue190404scdII}. 

Let $u\in\Omega^{0,q}(U\cap\ol M)$. From \eqref{e-gue190403yydq} 
and \eqref{e-gue190327ycd}, we have 
\begin{equation}\label{e-gue190404yydh}
\begin{split}
&(\ddbar\rho)^{\wedge,*}\gamma\ddbar\hat N^{(q)}u=
(\ddbar\rho)^{\wedge,*}\gamma\ddbar\,\ol{\pr}^*_f(N^{(q+1)})^2\ddbar u\\
&=(\ddbar\rho)^{\wedge,*}\gamma\Box^{(q+1)}_f(N^{(q+1)})^2\ddbar u-
(\ddbar\rho)^{\wedge,*}\gamma\ol{\pr}^*_f\ddbar(N^{(q+1)})^2\ddbar u\\
&=(\ddbar\rho)^{\wedge,*}\gamma(I+F^{(q+1)}_2)N^{(q+1)}\ddbar u-
(\ddbar\rho)^{\wedge,*}\gamma\ol{\pr}^*_f\ddbar(N^{(q+1)})^2\ddbar u\\
&=(\ddbar\rho)^{\wedge,*}\gamma F^{(q+1)}_2N^{(q+1)}\ddbar u+
(\ddbar\rho)^{\wedge,*}\gamma N^{(q+1)}\ddbar u-
(\ddbar\rho)^{\wedge,*}\gamma\ol{\pr}^*_f\ddbar(N^{(q+1)})^2\ddbar u,
\end{split}
\end{equation}
where $F^{(q+1)}_2\equiv0\mod\cC^\infty((U\times U)\cap(\ol M\times\ol M))$ 
is as in \eqref{e-gue190403yydq}. Again, from \eqref{e-gue190326yyd}, 
\eqref{e-gue190326yydI} and Lemma~\ref{l-gue190327ycdz}, 
we see that $(\ddbar\rho)^{\wedge,*}\gamma N^{(q+1)}\ddbar u=0$ 
and $(\ddbar\rho)^{\wedge,*}\gamma\ol{\pr}^*_f\ddbar(N^{(q+1)})^2\ddbar u|_D=0$. 
From this observation, \eqref{e-gue190404yydh} and notice that 
\[(\ddbar\rho)^{\wedge,*}\gamma F^{(q+1)}_2N^{(q+1)}\ddbar
\equiv0\mod\cC^\infty((U\times U)\cap(X\times\ol M)),\]
we get \eqref{e-gue190404scdI}. The proof of \eqref{e-gue190404scdIII} is similar. 

Let $u\in A^{0,q}(U\cap\ol M)$. From \eqref{e-gue190403yydq}, 
\eqref{e-gue190327ycd} and \eqref{e-gue190327ycdI}, we have 
\begin{equation}\label{e-gue190404yydc}
\begin{split}
&\Box^{(q)}_f\hat N^{(q)}u=
\Box^{(q)}_f\Bigr(\ol{\pr}^*_f(N^{(q+1)})^2\ddbar+
\ddbar(N^{(q-1)})^2\ol{\pr}^*_f\Bigr)u\\
&=\ol{\pr}^*_f\Box^{(q+1)}_f(N^{(q+1)})^2\ddbar u+
\ddbar\Box^{(q-1)}_f(N^{(q-1)})^2\ol{\pr}^*_fu\\
&=\ol{\pr}^*_f(I+F^{(q+1)}_2)N^{(q+1)}\ddbar u+
\ddbar(I+F^{(q-1)}_2)N^{(q-1)}\ol{\pr}^*_fu\\
&=\ol{\pr}^*_fN^{(q+1)}\ddbar u+
\ddbar N^{(q-1)}\ol{\pr}^*_fu+
\ol{\pr}^*_fF^{(q+1)}_2N^{(q+1)}\ddbar u+
\ddbar F^{(q-1)}_2N^{(q-1)}\ol{\pr}^*_fu\\
&=(I-\hat\Pi^{(q)})u+
\Bigr(\ol{\pr}^*_fF^{(q+1)}_2N^{(q+1)}\ddbar+
\ddbar F^{(q-1)}_2N^{(q-1)}\ol{\pr}^*_f\Bigr)u, 
\end{split}
\end{equation}
where $F^{(q+1)}_2\equiv0\mod\cC^\infty((U\times U)\cap(\ol M\times\ol M))$,  
$F^{(q-1)}_2\equiv0\mod\cC^\infty((U\times U)\cap(\ol M\times\ol M))$ 
are as in \eqref{e-gue190403yydq}. It is clear that 
\[\ol{\pr}^*_fF^{(q+1)}_2N^{(q+1)}\ddbar+
\ddbar F^{(q-1)}_2N^{(q-1)}\ol{\pr}^*_f\equiv
0\mod\cC^\infty((U\times U)\cap(\ol M\times\ol M)).\]
From this observation and \eqref{e-gue190404yydc}, we get \eqref{e-gue190404scdz}. 

Let $u\in A^{0,q}(U\cap\ol M)\cap\Omega^{0,q}_0(U\cap\ol M)$, 
from \eqref{e-gue190403yydIz}, \eqref{e-gue190403yydq}, 
\eqref{e-gue190327ycd} and \eqref{e-gue190327ycdI}, we have 
\begin{equation}\label{e-gue190404yydy}
\begin{split}
&\ol{\pr}^*_f\hat\Pi^{(q)}u=\ol{\pr}^*_fu-
\ol{\pr}^*_f\Bigr(\ol{\pr}^*_fN^{(q+1)}\ddbar u-
\ddbar N^{(q-1)}\ol{\pr}^*_fu\Bigr)\\
&=\ol{\pr}^*_fu-\ol{\pr}^*_f\,\ddbar N^{(q-1)}\ol{\pr}^*_fu\\
&=\ol{\pr}^*_fu-\Bigr(\Box^{(q-1)}_f-
\ddbar\,\ol{\pr}^*_f\Bigr)N^{(q-1)}\ol{\pr}^*_fu\\
&=\ol{\pr}^*_fu-(I+F^{(q-1)}_2)\ol{\pr}^*_fu+
\ddbar\,\ol{\pr}^*_fN^{(q-1)}\ol{\pr}^*_fu\\
&=-F^{(q-1)}_2\ol{\pr}^*_fu+
\ddbar\,\ol{\pr}^*_fN^{(q-1)}\ol{\pr}^*_fu.
\end{split}
\end{equation}
For every $g\in A^{0,q}(U\cap\ol M)\cap\Omega^{0,q}_0(U\cap\ol M)$, 
from \eqref{e-gue190326yydI}, \eqref{e-gue190327ycdII} and 
\eqref{e-gue190403yydq}, we have 
\begin{equation}\label{e-gue190521yyd}
\begin{split}
&(\ddbar\rho)^{\wedge,*}\gamma\ddbar\,\ol{\pr}^*_fN^{(q-1)}\ol{\pr}^*_fg\\
&=(\ddbar\rho)^{\wedge,*}\gamma\Bigr(\Box^{(q-1)}_f-\ol{\pr}^*_f\,
\ddbar\Bigr)N^{(q-1)}\ol{\pr}^*_fg\Bigr)\\
&=(\ddbar\rho)^{\wedge,*}\gamma(I+F^{(q-1)}_2)
\ol{\pr}^*_fg-(\ddbar\rho)^{\wedge,*}\gamma\ol{\pr}^*_f\,
\ddbar N^{(q-1)}\ol{\pr}^*_fg\\
&=(\ddbar\rho)^{\wedge,*}\gamma F^{(q-1)}_2\ol{\pr}^*_fg.
\end{split}
\end{equation}
Since $(\ddbar\rho)^{\wedge,*}\gamma F^{(q-1)}_2
\equiv0\mod\cC^\infty((U\times U)\cap(X\times\ol M))$, 
we can repeat the proof of Theorem~\ref{t-gue190321myyd} 
and deduce that there is a properly supported operator 
\[\varepsilon^{(q-1)}: \mathscr D'(U\cap M,T^{*0,q-1}M')
\To\Omega^{0,q-2}(U\cap M)\] 
on $U\cap\ol M$ with $\varepsilon^{(q-1)}\equiv0
\mod\cC^\infty((U\times U)\cap(\ol M\times\ol M))$ 
such that 
\begin{equation}\label{e-gue190521yydI}
\begin{split}
&(\ddbar\rho)^{\wedge,*}\gamma\Bigr(\ol{\pr}^*_fN^{(q-1)}-
\varepsilon^{(q-1)}\Bigr)\ol{\pr}^*_fg=0,\\
&(\ddbar\rho)^{\wedge,*}\gamma\ddbar\,\Bigr(\ol{\pr}^*_fN^{(q-1)}-
\varepsilon^{(q-1)}\Bigr)\ol{\pr}^*_fg=0,
\end{split}
\end{equation}
for every $g\in A^{0,q}(U\cap\ol M)\cap\Omega^{0,q}_0(U\cap\ol M)$ and hence 
\begin{equation}\label{e-gue190521yydII}
\Bigr(\ol{\pr}^*_fN^{(q-1)}-\varepsilon^{(q-1)}\Bigr)\ol{\pr}^*_fg\in{\rm Dom\,}\Box^{(q-2)}, 
\end{equation}
for every $g\in A^{0,q}(U\cap\ol M)\cap\Omega^{0,q}_0(U\cap\ol M)$. 
From \eqref{e-gue190403yydq}, \eqref{e-gue190403yydIz}, 
\eqref{e-gue190521yydII} and \eqref{e-gue190404yydy}, we have 
\begin{equation}\label{e-gue190521yydIII}
\begin{split}
&\ol{\pr}^*_f\hat\Pi^{(q)}u\\
&=-F^{(q-1)}_2\ol{\pr}^*_fu+\ddbar\,\ol{\pr}^*_fN^{(q-1)}\ol{\pr}^*_fu\\
&=-F^{(q-1)}_2\ol{\pr}^*_fu+\ddbar\,\Bigr(\ol{\pr}^*_fN^{(q-1)}-
\varepsilon^{(q-1)}\Bigr)\ol{\pr}^*_fu+\ddbar\varepsilon^{(q-1)}\ol{\pr}^*_fu\\
&=-F^{(q-1)}_2\ol{\pr}^*_fu+\ddbar\,\Bigr(N^{(q-2)}\Box^{(q-2)}-
F^{(q-2)}_1\Bigr)\Bigr(\ol{\pr}^*_fN^{(q-1)}-\varepsilon^{(q-1)}\Bigr)\ol{\pr}^*_fu\\
&\quad+\ddbar\varepsilon^{(q-1)}\ol{\pr}^*_fu\\
&=-F^{(q-1)}_2\ol{\pr}^*_fu+\ddbar\,\Bigr(N^{(q-2)}\Box^{(q-2)}_f-
F^{(q-2)}_1\Bigr)\Bigr(\ol{\pr}^*_fN^{(q-1)}-\varepsilon^{(q-1)}\Bigr)\ol{\pr}^*_fu\\
&\quad+\ddbar\varepsilon^{(q-1)}\ol{\pr}^*_fu\\
&=-F^{(q-1)}_2\ol{\pr}^*_fu+\ddbar N^{(q-2)}
\Box^{(q-2)}_f\ol{\pr}^*_fN^{(q-1)}\ol{\pr}^*_fu-
\ddbar N^{(q-2)}\Box^{(q-2)}_f\varepsilon^{(q-1)}\ol{\pr}^*_fu\\
&\quad-\ddbar F^{(q-2)}_1\Bigr(\ol{\pr}^*_fN^{(q-1)}-
\varepsilon^{(q-1)}\Bigr)\ol{\pr}^*_fu+
\ddbar\varepsilon^{(q-1)}\ol{\pr}^*_fu\\
&=-F^{(q-1)}_2\ol{\pr}^*_fu+
\ddbar N^{(q-2)}\ol{\pr}^*_f\Box^{(q-1)}_fN^{(q-1)}\ol{\pr}^*_fu-
\ddbar N^{(q-2)}\Box^{(q-2)}_f\varepsilon^{(q-1)}\ol{\pr}^*_fu\\
&\quad-\ddbar F^{(q-2)}_1\Bigr(\ol{\pr}^*_fN^{(q-1)}-
\varepsilon^{(q-1)}\Bigr)\ol{\pr}^*_fu+\ddbar\varepsilon^{(q-1)}\ol{\pr}^*_fu\\
&=-F^{(q-1)}_2\ol{\pr}^*_fu+
\ddbar N^{(q-2)}\ol{\pr}^*_f\Bigr(I+F^{(q-1)}_2\Bigr)\ol{\pr}^*_fu-
\ddbar N^{(q-2)}\Box^{(q-2)}_f\varepsilon^{(q-1)}\ol{\pr}^*_fu\\
&\quad-\ddbar F^{(q-2)}_1\Bigr(\ol{\pr}^*_fN^{(q-1)}-
\varepsilon^{(q-1)}\Bigr)\ol{\pr}^*_fu+
\ddbar\varepsilon^{(q-1)}\ol{\pr}^*_fu\\
&=-F^{(q-1)}_2\ol{\pr}^*_fu+
\ddbar N^{(q-2)}\ol{\pr}^*_fF^{(q-1)}_2\ol{\pr}^*_fu-
\ddbar N^{(q-2)}\Box^{(q-2)}_f\varepsilon^{(q-1)}\ol{\pr}^*_fu\\
&\quad-\ddbar F^{(q-2)}_1\Bigr(\ol{\pr}^*_fN^{(q-1)}-
\varepsilon^{(q-1)}\Bigr)\ol{\pr}^*_fu+\ddbar\varepsilon^{(q-1)}\ol{\pr}^*_fu,
\end{split}
\end{equation}
where $u\in A^{0,q}(U\cap\ol M)\cap\Omega^{0,q}_0(U\cap\ol M)$.
It is clear that 
\[\begin{split}
&-F^{(q-1)}_2\ol{\pr}^*_f+
\ddbar N^{(q-2)}\ol{\pr}^*_fF^{(q-1)}_2\ol{\pr}^*_f-
\ddbar N^{(q-2)}\Box^{(q-2)}_f\varepsilon^{(q-1)}\ol{\pr}^*_f\\
&\quad-\ddbar F^{(q-2)}_1\Bigr(\ol{\pr}^*_fN^{(q-1)}-
\varepsilon^{(q-1)}\Bigr)\ol{\pr}^*_f+
\ddbar\varepsilon^{(q-1)}\ol{\pr}^*_f\equiv0
\mod\cC^\infty((U\times U)\cap(\ol M\times\ol M)).
\end{split}\]
From this observation and \eqref{e-gue190521yydIII}, 
we get \eqref{e-gue190520ycdI}. The proof of \eqref{e-gue190520ycd} 
is similar but simpler and therefore we omit the details. 
\end{proof}

From \eqref{e-gue190520ycd} and \eqref{e-gue190520ycdI}, we get 

\begin{equation}\label{e-gue190404scdq}
\mbox{$\Box^{(q)}_f\hat\Pi^{(q)}u=H^{(q)}_6u$, for every $u\in A^{0,q}(U\cap\ol M)\cap\Omega^{0,q}_0(U\cap\ol M)$},
\end{equation}
where $H^{(q)}_6\equiv0\mod\cC^\infty((U\times U)\cap(X\times\ol M))$, $H^{(q)}_6$ is properly supported on $U\cap\ol M$. 

We need 

\begin{lem}\label{l-gue190521yyd}
With the assumptions and notations above, let $q=n_-$. We have that 
\[(\,\hat N^{(q)}u\,|\,v\,)_M=(\,u\,|\,\hat N^{(q)}v\,)_M+(\,u\,|\,\hat\Gamma^{(q)}v\,)_M,\]
for every $u\in L^2_{{\rm comp\,}}(U\cap\ol M,T^{*0,q}M')$, $v\in L^2_{{\rm loc\,}}(U\cap\ol M,T^{*0,q}M')$, where $\hat\Gamma^{(q)}\equiv0\mod\cC^\infty((U\times U)\cap(\ol M\times\ol M))$ and $\hat\Gamma^{(q)}$ is properly supported on $U\cap\ol M$.
\end{lem}

\begin{proof}
Let $u\in L^2_{{\rm comp\,}}(U\cap\ol M, T^{*0,q}M')$, $v\in L^2_{{\rm loc\,}}(U\cap\ol M, T^{*0,q}M')$. 
Let $u_j\in\Omega^{0,q}_0(U\cap\ol M)$, $v_j\in\Omega^{0,q}_0(U\cap\ol M)$, $j=1,2,\ldots$, such that
$u_j\To u$ in $L^2_{{\rm comp\,}}(U\cap\ol M, T^{*0,q}M')$ as $j\To+\infty$ and $v_j\To v$ in $L^2_{{\rm loc\,}}(U\cap\ol M, T^{*0,q}M')$ as $j\To+\infty$. 
From \eqref{e-gue190327ycd}, we see that 
\begin{equation}\label{e-gue190718yyd}
(\,\hat N^{(q)}u\,|\,v\,)_M=\lim_{j\To+\infty}(\,\hat N^{(q)}u_j\,|\,v_j\,)_M.
\end{equation}
From \eqref{e-gue190403yyd} and \eqref{e-gue190404yyda}, we can check that 
\begin{equation}\label{e-gue190718yydI}
(\,\hat N^{(q)}u_j\,|\,v_j\,)_M=(\,u_j\,|\,\hat N^{(q)}v_j\,)_M+(\,u_j\,|\,\hat\Gamma^{(q)}v_j\,)_M,\ \ \mbox{for every $j=1,2,\ldots$}, 
\end{equation}
where $\hat\Gamma^{(q)}\equiv0\mod\cC^\infty((U\times U)\cap(\ol M\times\ol M))$ 
and $\hat\Gamma^{(q)}$ is properly supported on $U\cap\ol M$. 
From \eqref{e-gue190718yyd} and \eqref{e-gue190718yydI}, the lemma follows. 
\end{proof}

\begin{lem}\label{l-gue190521yydI}
With the assumptions and notations used above, let $q=n_-$. 
Fix an open set $W\subset U$ with $\ol W$ is a compact subset of $U$. 
There is a constant $C_W>0$ such that 
\begin{equation}\label{e-gue190520yydI}
\norm{\hat\Pi^{(q)}u}_M\leq C_W\norm{u}_M,\ \ 
\mbox{for every $u\in A^{0,q}(U\cap\ol M)\cap\Omega^{0,q}_0(W\cap\ol M)$}.
\end{equation}
\end{lem}

\begin{proof}
Let $u\in A^{0,q}(U\cap\ol M)\cap\Omega^{0,q}_0(W\cap\ol M)$. 
From \eqref{e-gue190404scdz}, we have 
\begin{equation}\label{e-gue190521syd}
\begin{split}
&(\,\hat\Pi^{(q)}u\,|\,\hat\Pi^{(q)}u\,)_M=
(\,\hat\Pi^{(q)}u\,|\,u\,)_M-(\,\hat\Pi^{(q)}u\,|\,(I-\hat\Pi^{(q)})u)_M\\
&=(\,\hat\Pi^{(q)}u\,|\,u\,)_M-
(\,\hat\Pi^{(q)}u\,|\,(\Box^{(q)}_f\hat N^{(q)}-H^{(q)}_3)u\,)_M. 
\end{split}
\end{equation}
From \eqref{e-gue190404scd} and \eqref{e-gue190404scdI}, 
we can repeat the proof of Theorem~\ref{t-gue190321myyd} 
and deduce that there is a properly supported operator 
$N^{(q)}:  \mathscr D'(U\cap M,T^{*0,q}M')\To\Omega^{0,q}(U\cap M)$ 
on $U\cap\ol M$ with
$N^{(q)}-\hat N^{(q)}\equiv0\mod\cC^\infty((U\times U)\cap(\ol M\times\ol M))$ 
such that 
\begin{equation}\label{e-gue190521sydI}
N^{(q)}g\in{\rm Dom\,}\Box^{(q)},
\end{equation}
for every $g\in A^{0,q}(U\cap\ol M)\cap\Omega^{0,q}_0(W\cap\ol M)$. 
From \eqref{e-gue190521syd}, \eqref{e-gue190521sydI}, 
\eqref{e-gue190404scdz}, \eqref{e-gue190404scdII}, 
\eqref{e-gue190520ycd}, \eqref{e-gue190520ycdI},
we have 
\begin{equation}\label{e-gue190419sydII}
\begin{split}
&(\,\hat\Pi^{(q)}u\,|\,\hat\Pi^{(q)}u\,)_M\\
&=(\,\hat\Pi^{(q)}u\,|\,u\,)_M-(\,\hat\Pi^{(q)}u\,|\,(\Box^{(q)}_f\hat N^{(q)}-
H^{(q)}_3)u\,)_M\\
&=(\,\hat\Pi^{(q)}u\,|\,u\,)_M-(\,\hat\Pi^{(q)}u\,|\,(\Box^{(q)}_fN^{(q)}-
H^{(q)}_3)u\,)_M+(\,\hat\Pi^{(q)}u\,|\,\Box^{(q)}(N^{(q)}-
\hat N^{(q)})u\,)_M\\
&=(\,\hat\Pi^{(q)}u\,|\,u\,)_M-
(\,\ddbar\hat\Pi^{(q)}u\,|\,\ddbar\,N^{(q)}u\,)_M-
(\,\ol{\pr}^*_f\hat\Pi^{(q)}u\,|\,\ol{\pr}^*N^{(q)}u\,)_M\\
&\quad+(\,\hat\Pi^{(q)}u\,|\,H^{(q)}_3u\,)_M+
(\,\hat\Pi^{(q)}u\,|\,\Box^{(q)}(N^{(q)}-\hat N^{(q)})u\,)_M\\
&=(\,\hat\Pi^{(q)}u\,|\,u\,)_M-(\,H^{(q)}_4u\,|\,\ddbar\,N^{(q)}u\,)_M-
(\,H^{(q)}_5u\,|\,\ol{\pr}^*N^{(q)}u\,)_M\\
&\quad+(\,\hat\Pi^{(q)}u\,|\,H^{(q)}_3u\,)_M+
(\,\hat\Pi^{(q)}u\,|\,\Box^{(q)}(N^{(q)}-\hat N^{(q)})u\,)_M\\
&=(\,\hat\Pi^{(q)}u\,|\,u\,)_M-
(\,u\,|\,\Bigr((H^{(q)}_4)^*\ddbar\,N^{(q)}+
(H^{(q)}_5)^*\ol{\pr}^*N^{(q)}\Bigr)u\,)_M\\
&\quad+(\,\hat\Pi^{(q)}u\,|\,H^{(q)}_3u\,)_M+
(\,\hat\Pi^{(q)}u\,|\,\Box^{(q)}(N^{(q)}-\hat N^{(q)})u\,)_M,
\end{split}
\end{equation}
where 
\[H^{(q)}_3, H^{(q)}_4, H^{(q)}_5\equiv0
\mod\cC^\infty((U\times U)\cap(\ol M\times\ol M))\]
are as in \eqref{e-gue190404scdz}, \eqref{e-gue190520ycd}, 
\eqref{e-gue190520ycdI}, and 
$(H^{(q)}_4)^*$ and $(H^{(q)}_5)^*$ are the 
formal adjoints of $H^{(q)}_4$ and $H^{(q)}_5$ respectively. 
Note that
\[\begin{split}
&(H^{(q)}_4)^*\ddbar N^{(q)}+
(H^{(q)}_5)^*\ol{\pr}^*N^{(q)}, H^{(q)}_3, 
\Box^{(q)}(N^{(q)}-\hat N^{(q)})\\
&: L^2_{{\rm loc\,}}(U\cap\ol M, T^{*0,q}M')\To 
L^2_{{\rm loc\,}}(U\cap\ol M, T^{*0,q}M'),
\end{split}\]
are continuous. From this observation and \eqref{e-gue190419sydII}, we deduce that 
\begin{equation}\label{e-gue190419sydIII}
(\,\hat\Pi^{(q)}u\,|\,\hat\Pi^{(q)}u\,)_M\leq
\hat C\Bigr(\norm{\hat\Pi^{(q)}u}_M\norm{u}_M+\norm{u}^2_M\Bigr),
\end{equation}
where $\hat C>0$ is a constant independent of $u$. 
From \eqref{e-gue190419sydIII}, we get \eqref{e-gue190520yydI}. 
\end{proof}

\begin{rem}\label{r-gue190521syda}
Since $N^{(q-1)}$ and $N^{(q+1)}$ are properly supported on 
$U\cap\ol M$, $\hat\Pi$ is properly supported on $U\cap\ol M$. Hence
for every $\chi\in\cC^\infty_0(U\cap\ol M)$, there are 
$\chi_1\in\cC^\infty_0(U\cap\ol M)$, $\chi_2\in\cC^\infty_0(U\cap\ol M)$, such that 
\[\hat\Pi^{(q)}\chi u=\chi_2\hat\Pi^{(q)}u,\ \ \mbox{for every $u\in A^{0,q}(U\cap\ol M)$},\]
and 
\[\chi\hat\Pi^{(q)}u=\hat\Pi^{(q)}\chi_1u,\ \ \mbox{for every $u\in A^{0,q}(U\cap\ol M)$}.\]
\end{rem}

From Lemma~\ref{l-gue190521yydI}, we extend 
$\hat\Pi^{(q)}$ to $L^2_{{\rm comp\,}}(U\cap\ol M,T^{*0,q}M')$ 
by density. More precisely, let $u\in L^2_{{\rm comp\,}}(U\cap\ol M,T^{*0,q}M')$. 
Suppose that ${\rm Supp\,}u\subset W$, where $W\subset U$ is an open set with 
$\ol W\Subset U$. Take any sequence 
$u_j\in A^{0,q}(U\cap\ol M)\cap\Omega^{0,q}_0(W\cap\ol M)$, 
$j=1,2,\ldots$, with $\lim_{j\To+\infty}\norm{u_j-u}_M=0$. 
Since $\hat\Pi^{(q)}$ is properly supported on $U\cap\ol M$, we have   
\begin{equation}\label{e-gue190521sydb}
\mbox{$\hat\Pi^{(q)}u:=\lim_{j\To+\infty}\hat\Pi^{(q)}u_j$ in 
$L^2_{{\rm comp\,}}(U\cap\ol M, T^{*0,q}M')$}. 
\end{equation}
By using  $\hat\Pi^{(q)}$ is properly supported on $U\cap\ol M$, 
we can extend $\hat\Pi^{(q)}$ to $L^2_{{\rm loc\,}}(U\cap\ol M,T^{*0,q}M')$ 
and we have
 \begin{equation}\label{e-gue190521sydc}
\begin{split}
&\hat\Pi^{(q)}: L^2_{{\rm comp\,}}(U\cap\ol M, T^{*0,q}M')\To 
L^2_{{\rm comp\,}}(U\cap\ol M, T^{*0,q}M'),\\
&\hat\Pi^{(q)}: L^2_{{\rm loc\,}}(U\cap\ol M, T^{*0,q}M')\To 
L^2_{{\rm loc\,}}(U\cap\ol M, T^{*0,q}M'),
\end{split}
\end{equation}
is continuous. 

\begin{lem}\label{l-gue190521sydu}
With the assumptions and notations above, let $q=n_-$. We have that 
\begin{equation}\label{e-gue190521sydu}
(\,\hat\Pi^{(q)}u\,|\,v\,)_M=(\,u\,|\,\hat\Pi^{(q)}v\,)_M+(\,u\,|\,\hat\Gamma^{(q)}_1v\,)_M,
\end{equation}
for every $u\in L^2_{{\rm comp\,}}(U\cap\ol M, T^{*0,q}M')$, 
$v\in L^2_{{\rm loc\,}}(U\cap\ol M, T^{*0,q}M')$, where 
$\hat\Gamma^{(q)}_1\equiv0
\mod\cC^\infty((U\times U)\cap(\ol M\times\ol M))$ 
is a properly supported continuous operator on $U\cap\ol M$. 
\end{lem}

\begin{proof}
From \eqref{e-gue190403yyd}, \eqref{e-gue190404yyda} 
and \eqref{e-gue190327ycdI}, we get \eqref{e-gue190521sydu} 
for $u, v\in A^{0,q}(U\cap\ol M)\cap\Omega^{0,q}_0(U\cap\ol M)$. 
By using density argument and notice that $\hat\Pi^{(q)}$ 
is properly supported on $U\cap\ol M$, we get \eqref{e-gue190521sydu}. 
\end{proof}

\begin{thm}\label{t-gue190521scda}
We have that 
\begin{equation}\label{e-gue190521scdh}
\mbox{$\chi\hat\Pi^{(q)}u\in{\rm Dom\,}\ol{\pr}^*$
for every $\chi\in\cC^\infty_0(U\cap\ol M)$, every 
$u\in L^2_{{\rm loc\,}}(U\cap\ol M,T^{*0,q}M')$}, 
\end{equation}
\begin{equation}\label{e-gue190521scdi}
\mbox{$\ddbar\hat\Pi^{(q)}u=H^{(q)}_4u$, 
for every $u\in L^2_{{\rm loc\,}}(U\cap\ol M,T^{*0,q}M')$}, 
\end{equation}
\begin{equation}\label{e-gue190521scdj}
\mbox{$\ol{\pr}^*_f\hat\Pi^{(q)}u=H^{(q)}_5u$
for every $u\in L^2_{{\rm loc\,}}(U\cap\ol M,T^{*0,q}M')$}, 
\end{equation}
\begin{equation}\label{e-gue190521scdk}
\mbox{$\Box^{(q)}_f\hat N^{(q)}u+\hat\Pi^{(q)}u=
u+H^{(q)}_3u$, for every $u\in\Omega^{0,q}(U\cap\ol M)$},
\end{equation}
where $H^{(q)}_j\equiv0\mod\cC^\infty((U\times U)\cap(\ol M\times\ol M))$, 
$j=3,4,5$, are as in Theorem~\ref{t-gue190404ycd}. 
\end{thm}

\begin{proof}
Let $u\in L^2_{{\rm loc\,}}(U\cap\ol M,T^{*0,q}M')$ 
and let $\chi\in\cC^\infty_0(U\cap\ol M)$. 
Since $\hat\Pi^{(q)}$ is properly supported 
on $U\cap\ol M$ (see Remark~\ref{r-gue190521syda}), 
there is a $\chi_1\in\cC^\infty(U\cap\ol M)$ such that 
$\chi\hat\Pi^{(q)}=\hat\Pi^{(q)}\chi_1$ 
on $L^2_{{\rm loc\,}}(U\cap\ol M,T^{*0,q}M')$. 
Let $g\in{\rm Dom\,}\ddbar\cap L^2_{(0,q)}(\ol M)$. 
Let $u_j\in A^{0,q}(U\cap\ol M)\cap\Omega^{0,q}_0(U\cap\ol M)$, 
$j=1,2,\ldots$, with $\lim_{j\To+\infty}\norm{u_j-\chi_1u}_M=0$. Then, 
\begin{equation}\label{e-gue190521scdy}
\begin{split}
&(\,\chi\hat\Pi^{(q)}u\,|\,\ddbar g\,)_M=
(\,\hat\Pi^{(q)}\chi_1u\,|\,\ddbar g\,)_M=
\lim_{j\To+\infty}(\,\hat\Pi^{(q)}u_j\,|\,\ddbar g\,)_M\\
&=\lim_{j\To+\infty}(\,\ol{\pr}^*\hat\Pi^{(q)}u_j\,|\,g\,)_M=
\lim_{j\To+\infty}(\,H^{(q)}_5u_j\,|\,g\,)_M=(\,H^{(q)}_5u\,|\,g)_M,
\end{split}
\end{equation}
where $H^{(q)}_5\equiv0\mod\cC^\infty((U\times U)\cap(\ol M\times\ol M))$ 
is as in \eqref{e-gue190520ycdI}. 

From \eqref{e-gue190521scdy}, we deduce that 
$\chi\hat\Pi^{(q)}u\in{\rm Dom\,}\ol{\pr}^*$, 
we get \eqref{e-gue190521scdh} and we also get 
\eqref{e-gue190521scdj}. The proof of \eqref{e-gue190521scdi} is similar. 
We now prove \eqref{e-gue190521scdk}. 

Let $u\in\Omega^{0,q}(U\cap\ol M)$ and let 
$g\in\Omega^{0,q}_0(U\cap M)$. 
Since $\hat\Pi^{(q)}$, $\hat N^{(q)}$ and $H^{(q)}_3$ 
are properly supported on $U\cap\ol M$, there is a 
$\tau\in\cC^\infty_0(U\cap\ol M)$ such that 
\begin{equation}\label{e-gue190521sch}
\begin{split}
&(\,(\Box^{(q)}_f\hat N^{(q)}+\hat\Pi^{(q)})u\,|\,g\,)_M=
(\Box^{(q)}_f\hat N^{(q)}+\hat\Pi^{(q)})\tau u\,|\,g\,)_M,\\
&(\,(I+H^{(q)}_3)u\,|\,g\,)_M=(\,(I+H^{(q)}_3)\tau u\,|\,g\,)_M.
\end{split}
\end{equation}
 Let $u_j\in A^{0,q}(U\cap\ol M)\cap\Omega^{0,q}_0(U\cap\ol M)$, 
 $j=1,2,\ldots$, with $\lim_{j\To+\infty}\norm{u_j-\tau u}_M=0$. 
 From \eqref{e-gue190404scdz} and \eqref{e-gue190521sch}, we have 
 \begin{equation}\label{e-gue190523yyd}
 \begin{split}
&(\,(\Box^{(q)}_f\hat N^{(q)}+\hat\Pi^{(q)})u\,|\,g\,)_M= 
(\,(\Box^{(q)}_f\hat N^{(q)}+\hat\Pi^{(q)})\tau u\,|\,g\,)_M\\
&=(\,\hat N^{(q)}\tau u\,|\,\Box^{(q)}_fg\,)_M+
(\,\hat\Pi^{(q)}\tau u\,|\,g\,)_M\\
&=\lim_{j\To+\infty}\Bigr((\,\hat N^{(q)}u_j\,|\,\Box^{(q)}_fg\,)_M+
(\,\hat\Pi^{(q)}u_j\,|\,g\,)_M\Bigr)\\
&=\lim_{j\To+\infty}\Bigr((\,(\Box^{(q)}_f\hat N^{(q)}+
\hat\Pi^{(q)})u_j\,|\,g\,)_M=\lim_{j\To+\infty}(\,(I+H^{(q)}_3)u_j\,|\,g\,)_M\\
&=(\,(I+H^{(q)}_3)\tau u\,|\,g\,)_M=
(\,(I+H^{(q)}_3)u\,|\,g\,)_M.
 \end{split}
 \end{equation}
 Let $h\in\Omega^{0,q}_0(U\cap\ol M)$. 
 Take $h_j\in\Omega^{0,q}_0(U\cap M)$, 
 $j=1,2,\ldots$, so that $\lim_{j\To+\infty}\norm{h_j-h}_M=0$. 
 From \eqref{e-gue190523yyd} and \eqref{e-gue190521sydc}, we have 
 \begin{equation}\label{e-gue190723yyd}
 \begin{split}
 &(\,(\Box^{(q)}_f\hat N^{(q)}+\hat\Pi^{(q)})u\,|\,h\,)_M=
 \lim_{j\To+\infty}(\,(\Box^{(q)}_f\hat N^{(q)}+
 \hat\Pi^{(q)})u\,|\,h_j\,)_M\\
 &=\lim_{j+\infty}(\,(I+H^{(q)}_3)u\,|\,h_j\,)_M=
 (\,(I+H^{(q)}_3)u\,|\,h\,)_M. 
 \end{split}
 \end{equation}
 From \eqref{e-gue190723yyd}, we get \eqref{e-gue190521scdk}. 
\end{proof}

From Theorem~\ref{t-gue190404ycd} and 
Theorem~\ref{t-gue190521scda}, we can repeat 
the procedure in the proof of Theorem~\ref{t-gue190321myyd} and conclude that 

\begin{thm}\label{t-gue190404hyyd}
With the assumptions and notations used above, let $q=n_-$. 
We can find properly supported continuous operators on $U\cap\ol M$, 
\begin{equation}\label{e-gue190417ycd}
\begin{split}
&N^{(q)}: H^s_{{\rm loc\,}}(U\cap\ol M, T^{*0,q}M')\To 
H^{s}_{{\rm loc\,}}(U\cap\ol M, T^{*0,q}M'),\ \ 
\mbox{for every $s\in\mathbb Z$},\\
&\Pi^{(q)}: L^2_{{\rm loc\,}}(U\cap\ol M, T^{*0,q}M')\To 
L^2_{{\rm loc\,}}(U\cap\ol M, T^{*0,q}M'),
\end{split}\end{equation}
such that  
\begin{equation}\label{e-gue190419yyd}
\begin{split}
&N^{(q)}-\hat N^{(q)}\equiv0\mod
\cC^\infty((U\times U)\cap(\ol M\times\ol M)),\\
&\Pi^{(q)}-\hat\Pi^{(q)}\equiv0
\mod\cC^\infty((U\times U)\cap(\ol M\times\ol M)),
\end{split}
\end{equation}
\begin{equation}\label{e-gue190404hyydII}
\begin{split}
&\mbox{$\Box^{(q)}_fN^{(q)}u+\Pi^{(q)}u=
u+R^{(q)}_0u$, for every  $u\in \Omega^{0,q}(U\cap\ol M)$},\\
&\mbox{$\Box^{(q)}_f\Pi^{(q)}u=R^{(q)}_1u$, 
for every  $u\in L^2_{{\rm loc\,}}(U\cap M)$},\\
&\mbox{$\ddbar\Pi^{(q)}u=R^{(q)}_2u$, for every 
$u\in L^2_{{\rm loc\,}}(U\cap\ol M,T^{*0,q}M')$}, \\
&\mbox{$\ol{\pr}^*_f\Pi^{(q)}u=R^{(q)}_3u$, for every $u\in L^2_{{\rm loc\,}}(U\cap\ol M,T^{*0,q}M')$},
\end{split}
\end{equation}
\begin{equation}\label{e-gue190404hyyd}
\begin{split}
&\mbox{$(\ddbar\rho)^{\wedge,*}\gamma N^{(q)}u|_D=
0$, for every $u\in\Omega^{0,q}(U\cap\ol M)$}, \\
&\mbox{$\chi\Pi^{(q)}u\in{\rm Dom\,}\ol{\pr}^*$
for every $\chi\in\cC^\infty_0(U\cap\ol M)$ and every 
$u\in L^2_{{\rm loc\,}}(U\cap\ol M,T^{*0,q}M')$},
\end{split}
\end{equation}
\begin{equation}\label{e-gue190404hyydI}
\begin{split}
&\mbox{$(\ddbar\rho)^{\wedge,*}\gamma\ddbar N^{(q)}u|_D=
0$, for every $u\in\Omega^{0,q}(U\cap\ol M)$},\\
&\mbox{$(\ddbar\rho)^{\wedge,*}\gamma\ddbar\Pi^{(q)}u|_D=
0$, for every $u\in L^2_{{\rm loc\,}}(U\cap\ol M,T^{*0,q}M')$},
\end{split}
\end{equation}
where $R^{(q)}_j: \mathscr D'(U\cap M)\To\Omega^{0,q}(U\cap M)$ 
is a properly supported continuous operator on 
$U\cap\ol M$ with $R^{(q)}_j\equiv0
\mod\cC^\infty((U\times U)\cap(\ol M\times\ol M))$, $j=0,1,2,3$. 
\end{thm}

From Lemma~\ref{l-gue190521yyd} and  Lemma~\ref{l-gue190521sydu}, we get: 

\begin{thm}\label{t-gue190523ycd}
With the assumptions and notations used above, let $q=n_-$. We have
\begin{equation}\label{e-gue190417yyd}
(\,N^{(q)}u\,|\,v\,)_M=(u\,|\,N^{(q)}v\,)_M+(\,u\,|\,\Gamma^{(q)}v\,)_M,
\end{equation}
and 
\begin{equation}\label{e-gue190419scd}
(\,\Pi^{(q)}u\,|\,v\,)_M=(u\,|\,\Pi^{(q)}v\,)_M+(\,u\,|\,\Gamma^{(q)}_1v\,)_M,
\end{equation}
for every $u\in L^2_{{\rm comp\,}}(U\cap\ol M,T^{*0,q}M')$, 
$v\in L^2_{{\rm loc\,}}(U\cap\ol M,T^{*0,q}M')$, where $N^{(q)}$ 
and $\Pi^{(q)}$ are as in Theorem~\ref{t-gue190404hyyd}, 
$\Gamma^{(q)}, \Gamma^{(q)}_1\equiv0
\mod\cC^\infty((U\times U)\cap(\ol M\times\ol M))$, 
$\Gamma^{(q)}$ and $\Gamma^{(q)}_1$ are properly 
supported on $U\cap\ol M$.
\end{thm}

We can now prove 

\begin{thm}\label{t-gue190419ycdb}
With the assumptions and notations used above, let $q=n_-$. 
Let $N^{(q)}$ and $\Pi^{(q)}$ be as in Theorem~\ref{t-gue190404hyyd}. 
We have 
\begin{equation}\label{e-gue190419ycdb}
\mbox{$\Pi^{(q)}\Box^{(q)}u=\Lambda^{(q)}_0u$ on $U\cap\ol M$, 
for every $u\in{\rm Dom\,}\Box^{(q)}$}
\end{equation}
and 
\begin{equation}\label{e-gue190419ycdc}
\mbox{$N^{(q)}\Box^{(q)}u+\Pi^{(q)}u=u+
\Lambda^{(q)}u$ on $U\cap\ol M$, for every $u\in{\rm Dom\,}\Box^{(q)}$},
\end{equation}
where $\Lambda^{(q)}_0, \Lambda^{(q)}\equiv0
\mod\cC^\infty((U\times U)\cap(\ol M\times\ol M))$, 
$\Lambda^{(q)}_0$, $\Lambda^{(q)}$ are properly supported on $U\cap\ol M$.
\end{thm}

\begin{proof}
Let $u\in{\rm Dom\,}\Box^{(q)}$ and let 
$v\in\Omega^{0,q}_0(U\cap M)$. From \eqref{e-gue190404hyyd}, 
\eqref{e-gue190404hyydI}, \eqref{e-gue190404hyydII} and 
\eqref{e-gue190419scd}, we have 
\begin{equation}\label{e-gue190419ycdd}
\begin{split}
&(\,\Pi^{(q)}\Box^{(q)}u\,|\,v\,)_M=
(\,\Box^{(q)}u\,|\,\Pi^{(q)}v\,)_M+(\,\Box^{(q)}u\,|\,\Gamma^{(q)}_1v\,)_M\\
&=(\,u\,|\,\Box^{(q)}\Pi^{(q)}v\,)_M+(\,\Box^{(q)}u\,|\,\Gamma^{(q)}_1v\,)_M\\
&=(\,u\,|\,R^{(q)}_1v\,)_M+(\,\Box^{(q)}u\,|\,\Gamma^{(q)}_1v\,)_M\\
&=(\,\bigr((R^{(q)}_1)^*+(\Gamma^{(q)}_1)^*\Box^{(q)}_f\bigr)u\,|\,v\,)_M, 
\end{split}
\end{equation}
where $R^{(q)}_1$, $\Gamma^{(q)}_1$ are as in 
\eqref{e-gue190404hyydII} and \eqref{e-gue190419scd} respectively and 
$(R^{(q)}_1)^*$ and $(\Gamma^{(q)}_1)^*$ are the formal adjoints 
of $R^{(q)}_1$ and $\Gamma^{(q)}_1$ with respect to 
$(\,\cdot\,|\,\cdot\,)_M$ respectively. It is clear that 
$(R^{(q)}_1)^*+(\Gamma^{(q)}_1)^*\Box^{(q)}_f\equiv0
\mod\cC^\infty((U\times U)\cap(\ol M\times\ol M))$. 
From this observation and \eqref{e-gue190419ycdd}, 
we get \eqref{e-gue190419ycdb}. 

Let $u\in{\rm Dom\,}\Box^{(q)}$and let $v\in\Omega^{0,q}_0(U\cap M)$. 
From \eqref{e-gue190404hyyd}, \eqref{e-gue190404hyydI}, \eqref{e-gue190404hyydII}, 
\eqref{e-gue190419scd} and \eqref{e-gue190417yyd}, we have 
\begin{equation}\label{e-gue190419ycde}
\begin{split}
&(\,N^{(q)}\Box^{(q)}u+\Pi^{(q)}u\,|\,v\,)_M\\
&=(\,\Box^{(q)}u\,|\,N^{(q)}v\,)_M+(\,\Box^{(q)}u\,|\,\Gamma^{(q)}v\,)_M+(\,u\,|\,\Pi^{(q)}v\,)_M+(\,u\,|\,\Gamma^{(q)}_1v\,)_M\\
&=(\,u\,|\,\Box^{(q)}N^{(q)}v\,)_M+(\,(\Gamma^{(q)})^*\Box^{(q)}_fu\,|\,v\,)_M+(\,u\,|\,\Pi^{(q)}v\,)_M+(\,u\,|\,\Gamma^{(q)}_1v\,)_M\\
&=(\,u\,|\,(\Box^{(q)}N^{(q)}+\Pi^{(q)})v\,)_M+(\,(\Gamma^{(q)})^*\Box^{(q)}_fu\,|\,v\,)_M+(\,u\,|\,\Gamma^{(q)}_1v\,)_M\\
&=(\,u\,|\,R^{(q)}_0v\,)_M+(\,(\Gamma^{(q)})^*\Box^{(q)}_fu\,|\,v\,)_M+(\,u\,|\,\Gamma^{(q)}_1)v\,)_M\\
&=(\,((R^{(q)}_0+\Gamma^{(q)}_1)^*+(\Gamma^{(q)})^*\Box^{(q)}_f)u\,|\,v\,)_M, 
\end{split}
\end{equation}
where $R^{(q)}_0$, $\Gamma^{(q)}$, $\Gamma^{(q)}_1$ are as in \eqref{e-gue190404hyydII}, \eqref{e-gue190417yyd} and \eqref{e-gue190419scd} respectively, 
$(\Gamma^{(q)})^*$ is the formal adjoint of $\Gamma^{(q)}$ with respect to  $(\,\cdot\,|\,\cdot\,)_M$ 
and $(R^{(q)}_0+\Gamma^{(q)}_1)^*$ is the formal adjoint of $R^{(q)}_0+\Gamma^{(q)}_1$ with respect to $(\,\cdot\,|\,\cdot\,)_M$. It is clear that 
$(\Gamma^{(q)})^*\Box^{(q)}_f\equiv0\mod\cC^\infty((U\times U)\cap(\ol M\times\ol M))$ and 
$(R^{(q)}_0+\Gamma^{(q)}_1)^*\equiv0\mod\cC^\infty((U\times U)\cap(\ol M\times\ol M))$. From this observation and \eqref{e-gue190419ycde}, 
we get \eqref{e-gue190419ycdc}. 
\end{proof}

In the rest of this section, we will study regularity property and distribution kernel of $\Pi^{(q)}$. Let $[\,\cdot\,|\,\cdot\,]_X$ be the $L^2$ inner product on $H^{-\frac{1}{2}}(X,T^{*0,q}M')$ given by 
\begin{equation}\label{e-gue190515syd}
[\,u\,|\,v\,]_X:=(\,\tilde Pu\,|\,\tilde Pv\,)_M,\ \ u, v\in H^{-\frac{1}{2}}(X, T^{*0,q}M').
\end{equation}
Recall that $\tilde P$ is the Poisson operator given by \eqref{e-gue190312scdIII}. 
Let $\tilde P^*: \Omega^{0,q}(\ol M)\To\cC^\infty(X,T^{*0,q}M')$ be as in \eqref{e-gue190515yyda}. Then, 
\[\tilde P^*\tilde P: \cC^\infty(X, T^{*0,q}M')\To\cC^\infty(X, T^{*0,q}M')\]
is an injective continuous operator. Let 
\[(\tilde P^*\tilde P)^{-1}: \cC^\infty(X, T^{*0,q}M')\To\cC^\infty(X, T^{*0,q}M')\]
be the inverse of $\tilde P^*\tilde P$. It is well-known that (see~\cite{B71}) $(\tilde P^*\tilde P)^{-1}$ is a classical pseudodifferential operator of order one on $X$. 

Let 
\[{\rm Ker\,}(\ddbar\rho)^{\wedge,*}:=\set{u\in H^{-\frac{1}{2}}(X, T^{*0,q}M');\, (\ddbar\rho)^{\wedge,*}u=0}.\]
It is easy to see that ${\rm Ker\,}(\ddbar\rho)^{\wedge,*}=H^{-\frac{1}{2}}(X,T^{*0,q}X)$. Let 
\begin{equation}\label{e-gue190514yyd}
Q^{(q)}: H^{-\frac{1}{2}}(X,T^{*0,q}M')\To{\rm Ker\,}(\ddbar\rho)^{\wedge,*}
\end{equation}
be the orthogonal projection with respect to $[\,\cdot\,|\,\cdot\,]_X$. The following is well-known (see~\cite{Hsiao08})

\begin{thm}\label{t-gue190514yyd}
We have that $Q^{(q)}$ is a classical pseudodifferential operator of order $0$ with principal symbol $2(\ddbar\rho)^{\wedge,*}(\ddbar\rho)^{\wedge}$. 
Moreover, 
\begin{equation}\label{e-gue190514yydI}
I-Q^{(q)}=(\tilde P^*\tilde P)^{-1}(\ddbar\rho)^{\wedge}R,
\end{equation}
where $R: \cC^\infty(X, T^{*0,q}M')\To\cC^\infty(X,T^{*0,q-1}M')$ is a classical pseudodifferential operator of order $-1$. 
\end{thm}

Let $u\in\Omega^{0,q}_0(U\cap M)$. From Theorem~\ref{t-gue190321myyd}, \eqref{e-gue190327ycdI} and Theorem~\ref{t-gue190404hyyd}, we see that 
$\Pi^{(q)}u\in\Omega^{0,q}_0(U\cap\ol M)$ and $\gamma\Pi^{(q)}u\in\cC^\infty(X,T^{*0,q}M')$.
We need

\begin{thm}\label{t-gue190418ycd}
With the assumptions and notations used before, we have 
\begin{equation}\label{e-gue190418ycd}
\Pi^{(q)}u=\tilde P\gamma\Pi^{(q)}u+\epsilon^{(q)}u,\ \ \mbox{for every $u\in\Omega^{0,q}_0(U\cap M)$}, 
\end{equation}
where $\epsilon^{(q)}\equiv0\mod\cC^\infty((U\times U)\cap(\ol M\times\ol M))$. 
\end{thm}

\begin{proof}
Let $u\in\Omega^{0,q}_0(U\cap M)$. Since $\Pi^{(q)}$ is properly supported on $U\cap\ol M$, 
\[\Pi^{(q)}u\in\Omega^{0,q}_0(U\cap\ol M)\subset\Omega^{0,q}(\ol M).\]
From \eqref{e-gue190418yyda}, we have 
\begin{equation}\label{e-gue190418yydb}
D^{(q)}\tilde \Box_f^{(q)}\Pi^{(q)}u+\tilde P\gamma\Pi^{(q)}u=\Pi^{(q)}u.
\end{equation}
From \eqref{e-gue190404hyydII} and $\tilde\Box^{(q)}_f-\Box^{(q)}_f\equiv0\mod\cC^\infty(\ol M\times\ol M)$, we see that 
$D^{(q)}\tilde\Box^{(q)}_f\Pi^{(q)}\equiv0\mod\cC^\infty((U\times U)\cap(\ol M\times\ol M))$. From this observation and \eqref{e-gue190418yydb}, we get 
\eqref{e-gue190418ycd}. 
\end{proof}

From \eqref{e-gue190418ycd}, we have 
\begin{equation}\label{e-gue190418scdh}
\begin{split}
&(\tilde P^*\tilde P)^{-1}\tilde P^*\Pi^{(q)}u=(\tilde P^*\tilde P)^{-1}\tilde P^*\tilde P\gamma\Pi^{(q)}u+(\tilde P^*\tilde P)^{-1}\tilde P^*\epsilon^{(q)}u\\
&=\gamma\Pi^{(q)}u+(\tilde P^*\tilde P)^{-1}\tilde P^*\epsilon^{(q)}u
\end{split}
\end{equation}
and 
\begin{equation}\label{e-gue190524yydI}
\Pi^{(q)}u=\tilde P(\tilde P^*\tilde P)^{-1}\tilde P^*\Pi^{(q)}u+\epsilon^{(q)}_1u,
\end{equation}
for every $u\in\Omega^{0,q}_0(U\cap M)$, where $\epsilon^{(q)}_1=-\tilde P(\tilde P^*\tilde P)^{-1}\tilde P^*\epsilon^{(q)}u+\epsilon^{(q)}\equiv0\mod\cC^\infty((U\times U)\cap(\ol M\times\ol M))$. 
From \eqref{e-gue190313ad} and \eqref{e-gue190515yyd}, we see that $\tilde P(\tilde P^*\tilde P)^{-1}\tilde P^*\Pi^{(q)}$ is well-defined as a continuous operator 
\[\tilde P(\tilde P^*\tilde P)^{-1}\tilde P^*\Pi^{(q)}: L^2_{{\rm comp\,}}(U\cap\ol M, T^{*0,q}M')\To L^2_{{\rm loc\,}}(U\cap\ol M, T^{*0,q}M').\]
From this observation, \eqref{e-gue190524yydI} and by using density argument, we conclude that 
\begin{equation}\label{e-gue190524yydII}
\Pi^{(q)}-\tilde P(\tilde P^*\tilde P)^{-1}\tilde P^*\Pi^{(q)}\equiv0\mod\cC^\infty((U\times U)\cap(\ol M\times\ol M)). 
\end{equation}
Similarly, from \eqref{e-gue190313ad} and \eqref{e-gue190515yyd}, we see that $\Pi^{(q)}\tilde P(\tilde P^*\tilde P)^{-1}\tilde P^*$ is well-defined as a continuous operator 
\[\Pi^{(q)}\tilde P(\tilde P^*\tilde P)^{-1}\tilde P^*: L^2(\ol M, T^{*0,q}M')\To L^2_{{\rm loc\,}}(U\cap\ol M, T^{*0,q}M').\]
We need 

\begin{lem}\label{l-gue190524yyd}
With the assumptions and notations used before, we have 
\[\Pi^{(q)}\tilde P(\tilde P^*\tilde P)^{-1}\tilde P^*-\Pi^{(q)}\equiv0\mod\cC^\infty((U\times U)\cap(\ol M\times\ol M)).\]
\end{lem}

\begin{proof}
Let $u\in L^2_{(0,q)}(\ol M)$ and let $v\in\Omega^{0,q}_0(U\cap\ol M)$. From \eqref{e-gue190419scd} and \eqref{e-gue190524yydI}, we have 
\begin{equation}\label{e-gue190524yydIII}
\begin{split}
&(\,\Pi^{(q)}u\,|\,v\,)_M=(\,u\,|\,\Pi^{(q)}v\,)_M+(\,u\,|\,\Gamma^{(q)}_1v\,)_M\\
&=(\,u\,|\,\tilde P(\tilde P^*\tilde P)^{-1}\tilde P^*\Pi^{(q)}v\,)_M+(\,u\,|\,\epsilon^{(q)}_1v\,)_M+(\,u\,|\,\Gamma^{(q)}_1v\,)_M\\
&=(\,\tilde P(\tilde P^*\tilde P)^{-1}\tilde P^*u\,|\,\Pi^{(q)}v\,)_M+(\,u\,|\,\epsilon^{(q)}_1v\,)_M+(\,u\,|\,\Gamma^{(q)}_1v\,)_M\\
&=(\,\Pi^{(q)}\tilde P(\tilde P^*\tilde P)^{-1}\tilde P^*u\,|\,v\,)_M-(\,\tilde P(\tilde P^*\tilde P)^{-1}\tilde P^*u\,|\,\Gamma^{(q)}_1v\,)_M+(\,u\,|\,\epsilon^{(q)}_1v\,)_M+(\,u\,|\,\Gamma^{(q)}_1v\,)_M\\
&=(\,\Pi^{(q)}\tilde P(\tilde P^*\tilde P)^{-1}\tilde P^*u\,|\,v\,)_M-(\,(\Gamma^{(q)}_1)^*\tilde P(\tilde P^*\tilde P)^{-1}\tilde P^*u\,|\,v\,)_M+(\,(\epsilon^{(q)}_1+\Gamma^{(q)}_1)^*u\,|\,v\,)_M,
\end{split}
\end{equation}
where $(\Gamma^{(q)}_1)^*$ and $(\epsilon^{(q)}_1+\Gamma^{(q)}_1)^*$ are the formal adjoints of $\Gamma^{(q)}_1$ and $\epsilon^{(q)}_1+\Gamma^{(q)}_1$ respectively. 
Note that 
\[(\Gamma^{(q)}_1)^*\tilde P(\tilde P^*\tilde P)^{-1}\tilde P^*, (\epsilon^{(q)}_1+\Gamma^{(q)}_1)^*\equiv0\mod\cC^\infty((U\times U)\cap(\ol M\times\ol M)).\]
From this observation and \eqref{e-gue190524yydIII}, the lemma follows. 
\end{proof}

We can prove 

\begin{thm}\label{t-gue190524ycd}
With the assumptions and notations used before, we have 
\begin{equation}\label{e-gue190524ycd}
\Pi^{(q)}-\Pi^{(q)}\tilde PQ^{(q)}(\tilde P^*\tilde P)^{-1}\tilde P^*\equiv0\mod\cC^\infty((U\times U)\cap(\ol M\times\ol M))
\end{equation}
and 
\begin{equation}\label{e-gue190524ycdI}
\Pi^{(q)}-\tilde PQ^{(q)}(\tilde P^*\tilde P)^{-1}\tilde P^*\Pi^{(q)}\equiv0\mod\cC^\infty((U\times U)\cap(\ol M\times\ol M)).
\end{equation}
\end{thm} 

\begin{proof}
Let $u\in L^2_{(0,q)}(\ol M)$ and let $v\in\Omega^{0,q}_0(U\cap M)$. From \eqref{e-gue190419scd} and \eqref{e-gue190418ycd}, we have 
\begin{equation}\label{e-gue190524ycdII}
\begin{split}
&(\,\Pi^{(q)}\tilde P(I-Q^{(q)})(\tilde P^*\tilde P)^{-1}\tilde P^*u\,|\,v\,)_M\\
&=(\,\tilde P(I-Q^{(q)})(\tilde P^*\tilde P)^{-1}\tilde P^*u\,|\,\Pi^{(q)}v\,)_M+(\,\tilde P(I-Q^{(q)})(\tilde P^*\tilde P)^{-1}\tilde P^*u\,|\,\Gamma^{(q)}_1v\,)_M\\
&=(\,\tilde P(I-Q^{(q)})(\tilde P^*\tilde P)^{-1}\tilde P^*u\,|\,\tilde P\gamma\Pi^{(q)}v\,)_M+(\,\tilde P(I-Q^{(q)})(\tilde P^*\tilde P)^{-1}\tilde P^*u\,|\,\epsilon^{(q)}v\,)_M\\
&\quad+(\,\tilde P(I-Q^{(q)})(\tilde P^*\tilde P)^{-1}\tilde P^*u\,|\,\Gamma^{(q)}_1v\,)_M\\
&=[\,(I-Q^{(q)})(\tilde P^*\tilde P)^{-1}\tilde P^*u\,|\,\gamma\Pi^{(q)}v\,]_X+(\,(\epsilon^{(q)})^*\tilde P(I-Q^{(q)})(\tilde P^*\tilde P)^{-1}\tilde P^*u\,|\,v\,)_M\\
&\quad+(\,(\Gamma^{(q)}_1)^*\tilde P(I-Q^{(q)})(\tilde P^*\tilde P)^{-1}\tilde P^*u\,|\,v\,)_M,
\end{split}
\end{equation}
where $(\epsilon^{(q)})^*$ and $(\Gamma^{(q)}_1)^*$ are the formal adjoints of $\epsilon^{(q)}$ and $\Gamma^{(q)}_1$ respectively. 
From \eqref{e-gue190404hyydI}, we see that $[\,(I-Q^{(q)})(\tilde P^*\tilde P)^{-1}\tilde P^*u\,|\,\gamma\Pi^{(q)}v\,]_X=0$. From this observation, \eqref{e-gue190524ycdII} and notice that 
\[(\epsilon^{(q)})^*\tilde P(I-Q^{(q)})(\tilde P^*\tilde P)^{-1}\tilde P^*, (\Gamma^{(q)}_1)^*\tilde P(I-Q^{(q)})(\tilde P^*\tilde P)^{-1}\tilde P^*\equiv0\mod\cC^\infty((U\times U)\cap(\ol M\times\ol M)),\]
we get 
\begin{equation}\label{e-gue190524syda}
\Pi^{(q)}\tilde P(I-Q^{(q)})(\tilde P^*\tilde P)^{-1}\tilde P^*\equiv0\mod\cC^\infty((U\times U)\cap(\ol M\times\ol M)).
\end{equation}
From \eqref{e-gue190524syda} and Lemma~\ref{l-gue190524yyd}, we get \eqref{e-gue190524ycd}. 

Let $u\in L^2_{{\rm comp\,}}(\ol M)$ and let $v\in\Omega^{0,q}_0(U\cap M)$. From \eqref{e-gue190419scd}, we have 
\begin{equation}\label{e-gue190524sydb}
\begin{split}
&(\,\tilde P(I-Q^{(q)})(\tilde P^*\tilde P)^{-1}\tilde P^*\Pi^{(q)}u\,|\,v\,)_M\\
&=(\,(I-Q^{(q)})(\tilde P^*\tilde P)^{-1}\tilde P^*\Pi^{(q)}u\,|\,\tilde P^*v\,)_X\\
&=(\,(I-Q^{(q)})(\tilde P^*\tilde P)^{-1}\tilde P^*\Pi^{(q)}u\,|\,(\tilde P^*\tilde P)(\tilde P^*\tilde P)^{-1}\tilde P^*v\,)_X\\
&=[\,(I-Q^{(q)})(\tilde P^*\tilde P)^{-1}\tilde P^*\Pi^{(q)}u\,|\,(\tilde P^*\tilde P)^{-1}\tilde P^*v\,]_X\\
&=[\,(\tilde P^*\tilde P)^{-1}\tilde P^*\Pi^{(q)}u\,|\,(I-Q^{(q)})(\tilde P^*\tilde P)^{-1}\tilde P^*v\,]_X\\
&=(\,\Pi^{(q)}u\,|\,\tilde P(I-Q^{(q)})(\tilde P^*\tilde P)^{-1}\tilde P^*v\,)_M\\
&=(\,u\,|\,\Pi^{(q)}\tilde P(I-Q^{(q)})(\tilde P^*\tilde P)^{-1}\tilde P^*v\,)_M+(\,u\,|\,\Gamma^{(q)}_1\tilde P(I-Q^{(q)})(\tilde P^*\tilde P)^{-1}\tilde P^*v\,)_M.
\end{split}
\end{equation}
From \eqref{e-gue190524sydb} and \eqref{e-gue190524syda}, we deduce that 
\[\tilde P(I-Q^{(q)})(\tilde P^*\tilde P)^{-1}\tilde P^*\Pi^{(q)}\equiv0\mod\cC^\infty((U\times U)\cap(\ol M\times\ol M).\]
From this observation and \eqref{e-gue190524yydII}, we get \eqref{e-gue190524ycdI}.
\end{proof}

We can now prove the following regularity property for $\Pi^{(q)}$. 

\begin{thm}\label{t-gue190524sydh}
With the assumptions and notations used before, $\Pi^{(q)}$ can be continuously extend to 
\begin{equation}\label{e-gue190524sydh}
\begin{split}
&\Pi^{(q)}: H^s_{{\rm loc\,}}(U\cap\ol M, T^{*0,q}M')\To H^{s-1}_{{\rm loc\,}}(U\cap\ol M, T^{*0,q}M'),\ \ \mbox{for every $s\in\mathbb Z$},\\
&\Pi^{(q)}: H^s_{{\rm comp\,}}(U\cap\ol M, T^{*0,q}M')\To H^{s-1}_{{\rm comp\,}}(U\cap\ol M, T^{*0,q}M'),\ \ \mbox{for every $s\in\mathbb Z$}.
\end{split}
\end{equation}
\end{thm}

\begin{proof}
Let $u\in\Omega^{0,q}_0(U\cap\ol M)$. From \eqref{e-gue190524ycd}, we see that 
\begin{equation}\label{e-gue190524sydj}
\Pi^{(q)}u=\Pi^{(q)}\tilde PQ^{(q)}(\tilde P^*\tilde P)^{-1}\tilde P^*u+\gamma^{(q)}u,
\end{equation}
where $\gamma^{(q)}: \Omega^{0,q}_0(U\cap M)\To\mathscr D'(U\cap M, T^{*0,q}M')$ is a continuous opeator with $\gamma^{(q)}\equiv0\mod\cC^\infty((U\times U)\cap(\ol M\times\ol M))$. From Theorem~\ref{t-gue190321myyd}, \eqref{e-gue190327ycdI}, Theorem~\ref{t-gue190404hyyd} and notice that 
\[\tilde PQ^{(q)}(\tilde P^*\tilde P)^{-1}\tilde P^*u\in A^{0,q}(U\cap\ol M),\]
we conclude that 
\begin{equation}\label{e-gue190524sydk}
\Pi^{(q)}u=(I-\ol{\pr}^*_fN^{(q+1)}\ddbar-\ddbar N^{(q-1)}\ol{\pr}^*_f)\tilde PQ^{(q)}(\tilde P^*\tilde P)^{-1}\tilde P^*u+\gamma^{(q)}_1u,
\end{equation}
where $\gamma^{(q)}_1: \Omega^{0,q}_0(U\cap M)\To\mathscr D'(U\cap M, T^{*0,q}M')$ is a continuous operator with $\gamma^{(q)}_1\equiv0\mod\cC^\infty((U\times U)\cap(\ol M\times\ol M))$. From \eqref{e-gue190524sydk},
\[
\begin{split}
&N^{(q-1)}: H^s_{{\rm comp\,}}(U\cap\ol M, T^{*0,q}M')\To H^{s+1}_{{\rm comp\,}}(U\cap\ol M, T^{*0,q}M'),\ \ \mbox{for every $s\in\mathbb Z$},\\
&N^{(q+1)}: H^s_{{\rm comp\,}}(U\cap\ol M, T^{*0,q}M')\To H^{s+1}_{{\rm comp\,}}(U\cap\ol M, T^{*0,q}M'),\ \ \mbox{for every $s\in\mathbb Z$},
\end{split}\]
are continuous and note that $\Omega^{0,q}_0(U\cap\ol M)$ is dense in $H^s_{{\rm comp\,}}(U\cap \ol M, T^{*0,q}M')$, for every $s\in\mathbb Z$, we get 
\eqref{e-gue190524sydh}. 
\end{proof}

We pause and recall the operators $\ddbar_\beta$ and $\Box^{(q)}_\beta$ introduced in~\cite{Hsiao08}.  Put
\begin{equation} \label{e-gue190527yyd}
\ddbar_\beta=Q^{(q+1)}\gamma\ddbar\tilde P: \Omega^{0,q}(X)\To\Omega^{0,q+1}(X). 
\end{equation}
We recall that $Q^{(q+1)}$ is given by \eqref{e-gue190514yyd}. It is well-known that $\ddbar_\beta$ is a classical
pseudodifferential operator of order one (see Chapter 5 in~\cite{Hsiao08}).  Let
\begin{equation} \label{e-gue190527yydI}
\ol{\pr_\beta}^\dagger: \Omega^{0,q+1}(X)\To\Omega^{0,q}(X)
\end{equation}
be the formal adjoint of $\ddbar_\beta$ with respect to $[\,\cdot\,|\,\cdot\,]_X$, that is
$[\,\ddbar_\beta f\,|\,h\,]=[\,f\,|\,\ol{\pr_\beta}^\dagger h\,]_X$,
$f\in\Omega^{0,q}(X)$, $h\in\Omega^{0,q+1}(X)$. It is well-known that
$\ol{\pr}^\dagger_\beta$ is a classical
pseudodifferential operator of order one and we have 
\begin{equation}\label{e-gue190527yydII}
\mbox{ $\ol{\pr}^\dagger_\beta=\gamma\ol{\pr}^*_f\Td P$ on $\Omega^{0,q}(X)$, for $q=1,\ldots,n-1$}.
\end{equation}
(See Chapter 5 in~\cite{Hsiao08}.)
Set
\begin{equation} \label{e-gue190527yydIII}
\Box^{(q)}_\beta=\ol{\pr}^\dagger_\beta\,\ddbar_\beta+\ddbar_\beta\,\ol{\pr}^\dagger_\beta: 
\mathscr D'(X, T^{*0,q}X)\To\mathscr D'(X,T^{*0,q}X).
\end{equation}
It was shown in~\cite[Chapter 5]{Hsiao08} that $\Box^{(q)}_\beta$ is a classical
pseudodifferential operator of order two and the characteristic manifold of $\Box^{(q)}_\beta$ is given by
$\Sigma=\Sigma^+\bigcup\Sigma^-$,
where $\Sigma^+$, $\Sigma^-$ are as in \eqref{e-gue190313sydI}. 
As before, let $D$ be a local coordinate patch of $X$ with local coordinates $x=(x_1,\ldots,x_{2n-1})$ and we assume that the Levi form is non-degenerate of constant signature $(n_-, n_+)$ on $D$. Let $H\in L^{-1}_{{\rm cl\,}}(D,T^{*0,q}X\boxtimes(T^{*0,q}X)^*)$ be a properly supported pseudodifferential operator of order $-1$ on $D$ such that 
\begin{equation}\label{e-gue190528ycda}
H-\tilde P^*\tilde P\equiv 0\ \ \mbox{on $D$}. 
\end{equation}
The following was established in~\cite[Theorem 6.15]{Hsiao08}.

\begin{thm}\label{t-gue190527syd}
With the assumptions and notations above, let $q=n_-$.  Then there exist properly supported operators 
\[A\in L^{-1}_{\frac{1}{2},
\frac{1}{2}}(D, T^{*0,q}X\boxtimes(T^{*0,q}X)^*),\ \ 
S_-, S_+\in L^{0}_{\frac{1}{2},
\frac{1}{2}}(D, T^{*0,q}X\boxtimes(T^{*0,q}X)^*)\]
such that
\begin{equation} \label{e-gue190527syd} \begin{split}
{\rm WF\,}'(S_-(x,y)) &={\rm diag\,}\Bigr((\Sigma^-\cap T^*D)\times(\Sigma^-\cap T^*D)\Bigr), \\
{\rm WF\,}'(S_+(x,y)) &\subset{\rm diag\,}\Bigr((\Sigma^+\cap T^*D)\times(\Sigma^+\cap T^*D)\Bigr)
\end{split}\end{equation}
and
\begin{gather}
A\Box^{(q)}_\beta+S_-+S_+= I, \label{e-gue190527sydI} \\
\ddbar_\beta S_-\equiv0,\ \ol{\pr}^\dagger_\beta S_-\equiv0, \label{e-gue190527sydII} \\
S_-\equiv S_-^\dagger\equiv S_-^2, \label{e-gue190527sydIII}\\
S_+\equiv0\ \ \mbox{if $q\neq n_+$},\label{e-gue190527sydr}
\end{gather}
where 
\begin{equation}\label{e-gue190528ycdb}
S_-^\dagger:=2Q^{(q)}(\tilde P^*\tilde P)^{-1}S^*_-(\ddbar\rho)^{\wedge,*}(\ddbar\rho)^\wedge H: \Omega^{0,q}_0(D)\To\Omega^{0,q}(X),
\end{equation}
$H$ is given by \eqref{e-gue190528ycda}, $S^*_-$ is the formal adjoint of $S_-$ with respect to $(\,\cdot\,|\,\cdot\,)_X$, 
\[{\rm WF\,}'(S_-(x,y))=\set{(x, \xi, y, \eta)\in T^*X\times T^*X;\, (x, \xi, y, -\eta)\in{\rm WF}(S_-(x,y))}.\]
Here ${\rm WF\,}(S_-(x,y))$ is the wave front set of $S_-(x,y)$ in the sense of H\"{o}rmander. 

Moreover, the kernel $S_-(x,y)$ satisfis
\[S_-(x, y)\equiv\int^{\infty}_{0}\!\! e^{i\varphi_-(x, y)t}a(x, y, t)dt\]
with
\begin{equation}  \label{e-gue190527syda}\begin{split}
&a(x, y, t)\in S^{n-1}_{1, 0}(D\times D\times]0, \infty[,T^{*0,q}X\boxtimes(T^{*0,q}X)^*), \\
&a(x, y, t)\sim\sum^\infty_{j=0}a_j(x, y)t^{n-1-j}\text{ in }  S^{n-1}_{1, 0}(D\times D\times]0, \infty[,T^{*0,q}X\boxtimes(T^{*0,q}X)^*),\\
&a_0(x, x)\neq 0 \text{ for every } x\in D
\end{split}\end{equation}
(a formula for $a_0(x, x)$ will be given in Theorem~\ref{t-gue190528yyd} below),
where 
\[a_j(x, y)\in\cC^\infty(D\times D; T^{*0,q}X\boxtimes(T^{*0,q}X)^*),\ \ j=0,1,\ldots,\]
the phase function $\varphi_-$ is the same as the phase function appearing in the description of the singularities of the Szeg\H{o} kernels for lower energy forms in~\cite{HM17} (we refer the reader to~\cite[Theorems 3.3, 3.4]{HM17} for more properties for $\varphi_-$), in particular, we have 
\begin{gather}
\varphi_-(x, y)\in\cC^\infty(X\times X),\ \ {\rm Im\,}\varphi_-(x, y)\geq0, \label{e-gue190527sydh} \\
\varphi_-(x, x)=0,\ \ \varphi_-(x, y)\neq0\ \ \mbox{if}\ \ x\neq y, \label{e-gue190527sydi} \\
d_x\varphi_-\neq0,\ \ d_y\varphi_-\neq0\ \ \mbox{where}\ \ {\rm Im\,}\varphi_-=0, \label{e-gue190527sydj} \\
d_x\varphi_-(x, y)|_{x=y}=-\omega_0(x), \ \ d_y\varphi_-(x, y)|_{x=y}=\omega_0(x),\label{e-gue190527sydk} \\
\varphi_-(x, y)=-\ol\varphi_-(y, x). \label{e-gue190527sykl}
\end{gather}
\end{thm} 

For a given point $x_0\in D$, let $\{W_j\}_{j=1}^{n-1}$ be an
orthonormal frame of $(T^{1,0}X,\langle\,\cdot\,|\,\cdot\,\rangle)$ near $x_0$, for which the Levi form
is diagonal at $x_0$. Put
\begin{equation}\label{e-gue190528yyd}
\mathcal{L}_{x_0}(W_j,\ol W_\ell)=\mu_j(x_0)\delta_{j\ell}\,,\;\; j,\ell=1,\ldots,n-1\,.
\end{equation}
We will denote by
\begin{equation}\label{e-gue190528yydI}
\det\mathcal{L}_{x_0}=\prod_{j=1}^{n-1}\mu_j(x_0)\,.
\end{equation}
Let $\{e_j\}_{j=1}^{n-1}$ denote the basis of $T^{*0,1}X$, dual to $\{\ol W_j\}^{n-1}_{j=1}$. We assume that
$\mu_j(x_0)<0$ if\, $1\leq j\leq n_-$ and $\mu_j(x_0)>0$ if\, $n_-+1\leq j\leq n-1$. Put
\begin{equation}\label{e-gue190528yydII}
\mathcal{N}(x_0,n_-):=\set{ce_1(x_0)\wedge\ldots\wedge e_{n_-}(x_0);\, c\in\Complex},
\end{equation}
and let
\begin{equation}\label{e-gue190528yydIII}
\tau_{x_0,n_-}:T^{*0,q}_{x_0}X\To\mathcal{N}(x_0,n_-)
\end{equation}
be the orthogonal projection onto $\mathcal{N}(x_0,n_-)$ 
with respect to $\langle\,\cdot\,|\,\cdot\,\rangle$. The following was obtained in~\cite[Proposition 6.17]{Hsiao08}

\begin{thm} \label{t-gue190528yyd}
With the notations and assumptions used in Theorem~\ref{t-gue190527syd}, for $a_0(x,y)$ in \eqref{e-gue190527syda}, we have 
\begin{equation}\label{e-gue190528yyda}
a_0(x, x)=\frac{1}{2}\pi^{-n}\abs{\det\mathcal{L}_{x}}\tau_{x,n_-},\ \ \mbox{for every $x\in D$}, 
\end{equation}
where $\det\mathcal{L}_{x}$ and $\tau_{x,n_-}$ are given by \eqref{e-gue190528yydI} and \eqref{e-gue190528yydIII} respectively. 
\end{thm}

We come
back to our situation. In view of Lemma~\ref{l-gue190429yyd}, 
Lemma~\ref{l-gue190429yydI} and Theorem~\ref{t-gue190514yyd}, 
we see that $Q^{(q)}(\tilde P^*\tilde P)^{-1}\tilde P^*$ is smoothing away 
the diagonal. Hence, there is a continuous operator 
$L^{(q)}: \Omega^{0,q}_0(U\cap\ol M)\To\Omega^{0,q}(D)$ such that 
\begin{equation}\label{e-gue190524syds}
L^{(q)}-Q^{(q)}(\tilde P^*\tilde P)^{-1}\tilde P^*\equiv0
\mod\cC^\infty((U\times U)\cap(X\times\ol M))
\end{equation}
and $L^{(q)}$ is properly supported on $U\cap\ol M$, that is, 
for every $\chi\in\cC^\infty_0(U\cap\ol M)$, there is a 
$\tau\in\cC^\infty_0(D)$ such that $L^{(q)}\chi=\tau L^{(q)}$ on 
$\Omega^{0,q}_0(U\cap\ol M)$ and for every $\tau_1\in\cC^\infty_0(D)$, 
there is a $\chi_1\in\cC^\infty_0(U\cap\ol M)$ such that 
$\tau_1L^{(q)}=L^{(q)}\chi_1$ 
on $\Omega^{0,q}_0(U\cap\ol M)$. We can extend $L^{(q)}$ to a continuous opeator
\[\begin{split}
&L^{(q)}: \Omega^{0,q}(U\cap\ol M)\To\Omega^{0,q}(D),\\
&L^{(q)}: \Omega^{0,q}_0(U\cap\ol M)\To\Omega^{0,q}_0(D).
\end{split}\]
From Theorem~\ref{t-gue190524ycd}, we have 
\begin{equation}\label{e-gue190524sydp}
\Pi^{(q)}-\tilde PL^{(q)}\Pi^{(q)}\equiv0
\mod\cC^\infty((U\times U)\cap(\ol M\times\ol M)). 
\end{equation}

\begin{lem}\label{l-gue190524sydp}
With the notations and assumptions above, we have 
\begin{equation}\label{e-gue190524sydq}
S_+L^{(q)}\Pi^{(q)}\equiv0\mod\cC^\infty((U\times U)\cap(X\times\ol M)), 
\end{equation}
where $S_+$ is as in Theorem~\ref{t-gue190527syd}. 
\end{lem}

\begin{proof}
Since ${\rm WF\,}'(S_+(x,y))\subset{\rm diag\,}\Bigr((\Sigma^+\cap T^*D)
\times(\Sigma^+\cap T^*D)\Bigr)$ and $\Box^{(q)}_-$ is elliptic near 
$\Sigma^+$ (see Theorem~\ref{t-gue180313syd}), there is a classical 
pseudodifferential operator 
$E^{(q)}\in L^{-1}_{{\rm cl\,}}(D,T^{*0,q}X\boxtimes(T^{*0,q}X)^*)$ 
such that 
\begin{equation}\label{e-gue190524scdp}
S_+-S_+E^{(q)}\Box^{(q)}_-\equiv0.
\end{equation}
From \eqref{e-gue190404hyydII} and \eqref{e-gue190524sydp}, we deduce that 
\begin{equation}\label{e-gue190524scdq}
\Box^{(q)}_-L^{(q)}\Pi^{(q)}\equiv0\mod\cC^\infty((U\times U)\cap(X\times\ol M)).
\end{equation}
From \eqref{e-gue190524scdp} and \eqref{e-gue190524scdq}, we get \eqref{e-gue190524sydq}. 
\end{proof}

We can now prove the following 

\begin{thm}\label{t-gue190524yydp}
With the notations and assumptions above, we have 
\begin{gather}
S_-L^{(q)}\Pi^{(q)}-L^{(q)}\Pi^{(q)}\equiv0
\mod\cC^\infty((U\times U)\cap(X\times\ol M)), \label{e-gue190524yydp}\\
\tilde PS_-L^{(q)}\Pi^{(q)}-\Pi^{(q)}\equiv0
\mod\cC^\infty((U\times U)\cap(\ol M\times\ol M)), 
\label{e-gue190524yydq}\\
\Pi^{(q)}\tilde PS_-L^{(q)}-\Pi^{(q)}\equiv0
\mod\cC^\infty((U\times U)\cap(\ol M\times\ol M)).\label{e-gue190524yydr}
\end{gather}
\end{thm}

\begin{proof}
From \eqref{e-gue190404hyydII} and \eqref{e-gue190524sydp}, we see that 
\begin{equation}\label{e-gue190524yydz}
\Box^{(q)}_\beta L^{(q)}\Pi^{(q)}\equiv0\mod\cC^\infty((U\times U)\cap(X\times\ol M)).
\end{equation}
From \eqref{e-gue190524yydz}, \eqref{e-gue190524sydq} and \eqref{e-gue190527sydI}, we have 
\[L^{(q)}\Pi^{(q)}=(A\Box^{(q)}_\beta+S_-+S_+)L^{(q)}\Pi^{(q)}
\equiv S_-L^{(q)}\Pi^{(q)}\mod\cC^\infty((U\times U)\cap(X\times\ol M))\]
and we get \eqref{e-gue190524yydp}. 

From \eqref{e-gue190524yydp} and \eqref{e-gue190524sydp}, 
we get \eqref{e-gue190524yydq}. 

We now prove \eqref{e-gue190524yydr}. Put 
\[\begin{split}
&\gamma^{(q)}:=\Pi^{(q)}-\tilde PL^{(q)}\Pi^{(q)}: 
\Omega^{0,q}_0(U\cap\ol M)\To\Omega^{0,q}(\ol M),\\
&\gamma^{(q)}_0:=\tilde P^*\tilde P-H: \Omega^{0,q}_0(D)\To\Omega^{0,q}(X),\\
&\gamma^{(q)}_1:=S^\dagger_--S_-: \Omega^{0,q}_0(D)\To\Omega^{0,q}(X),\\
&\gamma^{(q)}_2:=L^{(q)}-Q^{(q)}(\tilde P^*\tilde P)^{-1}\tilde P^*: 
\Omega^{0,q}_0(U\cap\ol M)\To\Omega^{0,q}(X),\\
&\gamma^{(q)}_3:=S_-L^{(q)}\Pi^{(q)}-L^{(q)}\Pi^{(q)}: 
\Omega^{0,q}_0(U\cap\ol M)\To\Omega^{0,q}_0(D),
\end{split}\]
where $S^\dagger_-$ is given by \eqref{e-gue190528ycdb}. 
From \eqref{e-gue190524sydp}, \eqref{e-gue190527sydIII}, 
\eqref{e-gue190524syds} and \eqref{e-gue190524yydp}, we see that 
\begin{equation}\label{e-gue190528sydh}
\begin{split}
&\gamma^{(q)}\equiv0\mod\cC^\infty((U\times U)\cap(\ol M\times\ol M)),\ \ 
\gamma^{(q)}_2\equiv0\mod\cC^\infty((U\times U)\cap(X\times\ol M)),\\
&\gamma^{(q)}_3\equiv0\mod\cC^\infty((U\times U)\cap(X\times\ol M)),\ \ 
\gamma^{(q)}_1\equiv0,\ \ \gamma^{(q)}_0\equiv0.
\end{split}
\end{equation}
Let 
\[(\gamma^{(q)})^*: \Omega^{0,q}(\ol M)\To\Omega^{0,q}(U\cap\ol M)\]
be the formal adjoint of $\gamma^{(q)}$ with respect to $(\,\cdot\,|\,\cdot\,)_M$ and let 
\[(\gamma^{(q)}_2)^*: \Omega^{0,q}(X)\To\Omega^{0,q}(U\cap\ol M)\]
be the formal adjoint of $\gamma^{(q)}_2$ with respect to 
$(\,\cdot\,|\,\cdot\,)_M$ and $(\,\cdot\,|\,\cdot\,)_X$, that is, 
\[(\,\gamma^{(q)}_2u\,|\,v\,)_X=(\,u\,|\,(\gamma^{(q)}_2)^*v\,)_M,\]
for every $u\in\Omega^{0,q}_0(U\cap\ol M)$, $v\in\Omega^{0,q}(X)$. It is obvious that 
\begin{equation}\label{e-gue190528sydi}
(\gamma^{(q)})^*\equiv0\mod\cC^\infty((U\times U)\cap(\ol M\times\ol M)),\ \ 
(\gamma^{(q)}_2)^*\equiv0\mod\cC^\infty((U\times U)\cap(\ol M\times X)).
\end{equation}
Let $u, v\in\Omega^{0,q}_0(U\cap\ol M)$. From \eqref{e-gue190419scd},  
it is straightforward to check that
\begin{equation}\label{e-gue190528syd}
\begin{split}
&(\,\Pi^{(q)}\tilde PS_-L^{(q)}u\,|\,v\,)_M\\
&=(\,\tilde PS_-L^{(q)}u\,|\,\Pi^{(q)}v\,)_M+
(\,\tilde PS_-L^{(q)}u\,|\,\Gamma^{(q)}_1v\,)_M\\
&=(\,\tilde PS_-L^{(q)}u\,|\,\tilde PL^{(q)}\Pi^{(q)}v\,)_M+
(\,\tilde PS_-L^{(q)}u\,|\,\gamma^{(q)}v\,)_M+
(\,\tilde PS_-L^{(q)}u\,|\,\Gamma^{(q)}_1v\,)_M\\
&=(\,S_-L^{(q)}u\,|\,HL^{(q)}\Pi^{(q)}v\,)_X+
(\,S_-L^{(q)}u\,|\,\gamma^{(q)}_0L^{(q)}\Pi^{(q)}v\,)_X\\
&\quad+(\,\tilde PS_-L^{(q)}u\,|\,\gamma^{(q)}v\,)_M+
(\,\tilde PS_-L^{(q)}u\,|\,\Gamma^{(q)}_1v\,)_M\\
&=[\,L^{(q)}u\,|\,S^\dagger_-L^{(q)}\Pi^{(q)}v\,]_X+
(\,S_-L^{(q)}u\,|\,\gamma^{(q)}_0L^{(q)}\Pi^{(q)}v\,)_X\\
&\quad+(\,\tilde PS_-L^{(q)}u\,|\,\gamma^{(q)}v\,)_M+
(\,\tilde PS_-L^{(q)}u\,|\,\Gamma^{(q)}_1v\,)_M\\
&=[\,L^{(q)}u\,|\,S_-L^{(q)}\Pi^{(q)}v\,]_X+
[\,L^{(q)}u\,|\,\gamma^{(q)}_1L^{(q)}\Pi^{(q)}v\,]_X+
(\,S_-L^{(q)}u\,|\,\gamma^{(q)}_0L^{(q)}\Pi^{(q)}v\,)_X\\
&\quad+(\,\tilde PS_-L^{(q)}u\,|\,\gamma^{(q)}v\,)_M+
(\,\tilde PS_-L^{(q)}u\,|\,\Gamma^{(q)}_1v\,)_M\\
&=[\,Q^{(q)}(\tilde P^*\tilde P)^{-1}\tilde P^*u\,|\,S_-L^{(q)}\Pi^{(q)}v\,]_X+
[\,\gamma^{(q)}_2u\,|\,S_-L^{(q)}\Pi^{(q)}v\,]_X+
[\,L^{(q)}u\,|\,\gamma^{(q)}_1L^{(q)}\Pi^{(q)}v\,]_X\\
&\quad+(\,S_-L^{(q)}u\,|\,\gamma^{(q)}_0L^{(q)}\Pi^{(q)}v\,)_X+
(\,\tilde PS_-L^{(q)}u\,|\,\gamma^{(q)}v\,)_M+
(\,\tilde PS_-L^{(q)}u\,|\,\Gamma^{(q)}_1v\,)_M\\
&=(\,u\,|\,\tilde PS_-L^{(q)}\Pi^{(q)}v\,)_M+
(\,\tilde P\gamma^{(q)}_2u\,|\,\tilde PS_-L^{(q)}\Pi^{(q)}v\,)_M+
(\,\tilde PL^{(q)}u\,|\,\tilde P\gamma^{(q)}_1L^{(q)}\Pi^{(q)}v\,)_M\\
&\quad+(\,\tilde PS_-L^{(q)}u\,|\,\tilde P(\tilde P^*\tilde P)^{-1}
\gamma^{(q)}_0L^{(q)}\Pi^{(q)}v\,)_M+
(\,\tilde PS_-L^{(q)}u\,|\,\gamma^{(q)}v\,)_M+
(\,\tilde PS_-L^{(q)}u\,|\,\Gamma^{(q)}_1v\,)_M\\
&=(\,u\,|\,\tilde PL^{(q)}\Pi^{(q)}v\,)_M+
(\,u\,|\,\tilde P\gamma^{(q)}_3v\,)_M+
(\,u\,|\,(\gamma^{(q)}_2)^*\tilde P^*\tilde PS_-L^{(q)}\Pi^{(q)}v\,)_M\\
&\quad+(\,u\,|\,(L^{(q)})^*\tilde P^*\tilde P\gamma^{(q)}_1L^{(q)}\Pi^{(q)}v\,)_M+
(\,u\,|\,(L^{(q)})^*S^*_-\gamma^{(q)}_0L^{(q)}\Pi^{(q)}v\,)_M\\
&\quad+(\,u\,|\,(L^{(q)})^*(S_-)^*(\tilde P)^*\gamma^{(q)}v\,)_M+
(\,u\,|\,(L^{(q)})^*(S_-)^*(\tilde P)^*\Gamma^{(q)}_1v\,)_M\\
&=(\,u\,|\,\Pi^{(q)}v\,)_M-(\,u\,|\,\gamma^{(q)}v\,)_M+
(\,u\,|\,\tilde P\gamma^{(q)}_3v\,)_M+
(\,u\,|\,(\gamma^{(q)}_2)^*\tilde P^*\tilde PS_-L^{(q)}\Pi^{(q)}v\,)_M\\
&\quad+(\,u\,|\,(L^{(q)})^*\tilde P^*\tilde P\gamma^{(q)}_1L^{(q)}\Pi^{(q)}v\,)_M+
(\,u\,|\,(L^{(q)})^*S^*_-\gamma^{(q)}_0L^{(q)}\Pi^{(q)}v\,)_M\\
&\quad+(\,u\,|\,(L^{(q)})^*(S_-)^*(\tilde P)^*\gamma^{(q)}v\,)_M+
(\,u\,|\,(L^{(q)})^*(S_-)^*(\tilde P)^*\Gamma^{(q)}_1v\,)_M,
\end{split}
\end{equation}
where $(L^{(q)})^*: \Omega^{0,q}(D)\To\Omega^{0,q}(U\cap\ol M)$
is the formal adjoint of $L^{(q)}$ with respect to $(\,\cdot\,|\,\cdot\,)_M$ 
and $(\,\cdot\,|\,\cdot\,)_X$. Note that $(L^{(q)})^*$ is properly supported. 
From \eqref{e-gue190528syd}, we conclude that there is a continuous operator 
$\varepsilon^{(q)}: \Omega^{0,q}_0(U\cap\ol M)\To\Omega^{0,q}(U\cap\ol M)$ with 
$\varepsilon^{(q)}\equiv0\mod\cC^\infty((U\times U)\cap(\ol M\times\ol M))$ such that 
\begin{equation}\label{e-gue190528sydw}
(\,\Pi^{(q)}\tilde PS_-L^{(q)}u\,|\,v\,)_M=(\,u\,|\,\Pi^{(q)}v\,)_M+(\,u\,|\,\varepsilon^{(q)}v\,)_M,
\end{equation}
for every $u, v\in\Omega^{0,q}_0(U\cap\ol M)$. 
From \eqref{e-gue190528sydw} and \eqref{e-gue190419scd}, we get 
\begin{equation}\label{e-gue190528sydz}
(\,\Pi^{(q)}\tilde PS_-L^{(q)}u\,|\,v\,)_M=
(\,\Pi^{(q)}u\,|\,v\,)_M-(\,(\Gamma^{(q)}_1)^*u\,|\,v\,)_M+
(\,(\varepsilon^{(q)})^*u\,|\,v\,)_M,
\end{equation}
for every $u, v\in\Omega^{0,q}_0(U\cap\ol M)$, where 
$(\Gamma^{(q)}_1)^*, (\varepsilon^{(q)})^*: 
\Omega^{0,q}_0(U\cap\ol M)\To\Omega^{0,q}(U\cap\ol M)$ 
are the formal adjoints of $\Gamma^{(q)}_1$ and $\varepsilon^{(q)}$ 
with respect to $(\,\cdot\,|\,\cdot\,)_M$ respectively. 
Note that $(\Gamma^{(q)}_1)^*, (\varepsilon^{(q)})^*\equiv0
\mod\cC^\infty((U\times U)\cap(\ol M\times\ol M))$. 
From this observation and \eqref{e-gue190528sydz}, we get \eqref{e-gue190524yydr}.
\end{proof}

The following was proved in~\cite[Proposition 6.18]{Hsiao08}.

\begin{thm}\label{t-gue190528sydk}
With the notations and assumptions used above, we have 
\begin{equation}\label{e-gue190528sydk}
\ddbar\tilde PS_-L^{(q)}\equiv0\mod\cC^\infty((U\times U)\cap(\ol M\times\ol M)). 
\end{equation}
\end{thm} 

Let $\delta^{(q)}:=2\rho\tilde P\Bigr((\ddbar\rho)^{\wedge,*}
\gamma\ddbar\tilde PS_-L^{(q)}\Bigr): \Omega^{0,q}_0(U\cap\ol M)\To
\Omega^{0,q}(\ol M)$. 
From \eqref{e-gue190528sydk}, we see that 
\begin{equation}\label{e-gue190529yyd}
\delta^{(q)}\equiv0\mod\cC^\infty((U\times U)\cap(\ol M\times\ol M)). 
\end{equation}
Moreover, it is easy to check that 
\begin{equation}\label{e-gue190529yydI}
\mbox{($\tilde PS_-L^{(q)}-\delta^{(q)})u
\in{\rm Dom\,}\Box^{(q)}\cap\Omega^{0,q}(\ol M)$, 
for every $u\in\Omega^{0,q}_0(U\cap\ol M)$.}
\end{equation}

\begin{thm}\label{t-gue190529yyd}
With the notations and assumptions used above, we have
\begin{equation}\label{e-gue190529yydII}
\Pi^{(q)}\tilde PS_-L^{(q)}-\tilde PS_-L^{(q)}\equiv0
\mod\cC^\infty((U\times U)\cap(\ol M\times\ol M)).
\end{equation}
\end{thm}

\begin{proof}
From \eqref{e-gue190419ycdc} and \eqref{e-gue190529yydI}, we have 
\begin{equation}\label{e-gue190529yyda}
N^{(q)}\Box^{(q)}(\tilde PS_-L^{(q)}-\delta^{(q)})u+
\Pi^{(q)}(\tilde PS_-L^{(q)}-\delta^{(q)})u=
(\tilde PS_-L^{(q)}-\delta^{(q)})u+
\Lambda^{(q)}(\tilde PS_-L^{(q)}-\delta^{(q)})u, 
\end{equation}
for every $u\in\Omega^{0,q}_0(U\cap\ol M)$, where 
$\Lambda^{(q)}\equiv0\mod\cC^\infty((U\times U)\cap(\ol M\times\ol M))$ 
is as in \eqref{e-gue190419ycdc}. 
From \eqref{e-gue190527sydII}, \eqref{e-gue190528sydk} and \eqref{e-gue190529yyd}, we have 
\begin{equation}\label{e-gue190529yydb}
N^{(q)}\Box^{(q)}(\tilde PS_-L^{(q)}-\delta^{(q)})u=
N^{(q)}\Box^{(q)}_f(\tilde PS_-L^{(q)}-\delta^{(q)})u=\Lambda^{(q)}_1u,
\end{equation}
for every $u\in\Omega^{0,q}_0(U\cap\ol M)$, where 
$\Lambda^{(q)}_1\equiv0\mod\cC^\infty((U\times U)\cap(\ol M\times\ol M))$. 
From \eqref{e-gue190529yydb}, \eqref{e-gue190529yyda} and 
\eqref{e-gue190529yyd}, we get \eqref{e-gue190529yydII}. 
\end{proof}

From \eqref{e-gue190529yydII} and \eqref{e-gue190524yydr}, we conclude that 
\begin{equation}\label{e-gue190529ycd}
\Pi^{(q)}-\tilde PS_-L^{(q)}\equiv0\mod\cC^\infty((U\times U)\cap(\ol M\times\ol M)).
\end{equation}
Note that $S_-\in L^0_{\frac{1}{2},\frac{1}{2}}(D,T^{*0,q}X
\boxtimes(T^{*0,q}X)^*)$. From this observation and the classical 
result of Calderon and Vaillancourt 
(see \eqref{e-gue190322yyd}), \eqref{e-gue190313ad}, 
\eqref{e-gue190515yyd} and \eqref{e-gue190529ycd}, 
we can improve Theorem~\ref{t-gue190524sydh} as follows.

\begin{thm}\label{t-gue190529ycdh}
With the notations used above, $\Pi^{(q)}$ can be continuously extended to 
\begin{equation}\label{e-gue190529ycdI}
\begin{split}
&\Pi^{(q)}: H^s_{{\rm loc\,}}(U\cap\ol M, T^{*0,q}M')\To 
H^{s}_{{\rm loc\,}}(U\cap\ol M, T^{*0,q}M'),\ \ \mbox{for every $s\in\mathbb Z$},\\
&\Pi^{(q)}: H^s_{{\rm comp\,}}(U\cap\ol M, T^{*0,q}M')\To 
H^{s}_{{\rm comp\,}}(U\cap\ol M, T^{*0,q}M'),\ \ \mbox{for every $s\in\mathbb Z$}.
\end{split}
\end{equation}
\end{thm}

We introduce some notations. Let $m\in\Real$ and let $U$ be an open set in $M'$. Let 
\[S^{m}_{1, 0}(((U\times U)\cap(\ol M\times\ol M))\times]0, \infty[,
T^{*0,q}M'\boxtimes(T^{*0,q}M')^*)\]
denote the space of restrictions to $U\cap\ol M$ of elements in 
\[S^{m}_{1, 0}(U\times U\times]0, \infty[,T^{*0,q}M'\boxtimes(T^{*0,q}M')^*).\]\
Let
\[a_j\in S^{m_j}_{1, 0}(((U\times U)\cap(\ol M\times\ol M))\times]0, \infty[,
T^{*0,q}M'\boxtimes(T^{*0,q}M')^*),\ \ j=0,1,2,\dots,\] 
with $m_j\searrow -\infty$, $j\To \infty$.
Then there exists
\[a\in S^{m_0}_{1, 0}(((U\times U)\cap(\ol M\times\ol M))\times]0, \infty[,
T^{*0,q}M'\boxtimes(T^{*0,q}M')^*)\]
such that
\[a-\sum^{k-1}_{j=0}a_j\in 
S^{m_k}_{1, 0}((U\times U)\cap(\ol M\times\ol M)\times]0, \infty[,
T^{*0,q}M'\boxtimes(T^{*0,q}M')^*),\]
for every $k=1,2,\ldots$. If $a$ and $a_j$ have the properties above, we write
\[a\sim\sum^\infty_{j=0}a_j \text{ in }
S^{m_0}_{1, 0}(((U\times U)\cap(\ol M\times\ol M))\times]0, \infty[,
T^{*0,q}M'\boxtimes(T^{*0,q}M')^*).\]

We can repeat the procedure in the proof of~\cite[Proposition 7.8 of Part II]{Hsiao08} 
and deduce that the distribution kernel of $\tilde PS_-L^{(q)}$ is of the form 
\begin{equation}\label{e-gue190529ycdII}
\tilde PS_-L^{(q)}(z, w)\equiv\int^\infty_0e^{i\phi(z, w)t}b(z, w, t)dt
\mod\cC^\infty((U\times U)\cap(\ol M\times\ol M)),
\end{equation}
where $\phi(z, w)\in\cC^\infty((U\times U)\cap(\ol M\times\ol M))$ 
is as in~\cite[Theorem 1.4 of Part II]{Hsiao08}, 
\[\mbox{$b(z, w, t)\sim\sum^\infty_{j=0}b_j(z, w)t^{n-j}$ in 
$S^{n}_{1, 0}((U\times U)\cap(\ol M\times\ol M)\times]0, \infty[, 
T^{*0,q}M'\boxtimes(T^{*0,q}M')^*)$}\]
and the leading term $b_0(z,z)$ is given by~\cite[Proposition 1.6 of Part II]{Hsiao08}. 
From Theorem~\ref{t-gue190404hyyd}, 
Theorem~\ref{t-gue190523ycd}, Theorem~\ref{t-gue190419ycdb}, 
Theorem~\ref{t-gue190529ycdh}, \eqref{e-gue190529ycd} and 
\eqref{e-gue190529ycdII}, we get the main result of this section:

\begin{thm}\label{t-gue190529ycdi}
Let $U$ be an open set of $M'$ with $U\cap X\neq\emptyset$. Suppose that the Levi form is non-degenerate of constant signature $(n_-, n_+)$ on $U\cap X$. 
Let $q=n_-$. We can find properly supported continuous operators on $U\cap\ol M$, 
\begin{equation}\label{e-gue190529syd}
\begin{split}
&N^{(q)}: H^s_{{\rm loc\,}}(U\cap\ol M, T^{*0,q}M')\To 
H^{s}_{{\rm loc\,}}(U\cap\ol M, T^{*0,q}M'),\ \ 
\mbox{for every $s\in\mathbb Z$},\\
&\Pi^{(q)}: H^s_{{\rm loc\,}}(U\cap\ol M, T^{*0,q}M')\To 
H^s_{{\rm loc\,}}(U\cap\ol M, T^{*0,q}M'),\ \ \mbox{for every $s\in\mathbb Z$},
\end{split}\end{equation}
such that  
\begin{equation}\label{e-gue190529sydI}
\begin{split}
&N^{(q)}u\in{\rm Dom\,}\Box^{(q)},\ \ \mbox{for every $u\in\Omega^{0,q}_0(U\cap\ol M)$},\\
&\Pi^{(q)}u\in{\rm Dom\,}\Box^{(q)},\ \ \mbox{for every $u\in\Omega^{0,q}_0(U\cap\ol M)$},
\end{split}
\end{equation}
and on $U\cap\ol M$, we have
\begin{equation}\label{e-gue190531yyd}
\begin{split}
&\mbox{$\Box^{(q)}_fN^{(q)}u+\Pi^{(q)}u=u+r^{(q)}_0u$ 
for every  $u\in \Omega^{0,q}(U\cap M)$},\\
&\mbox{$N^{(q)}\Box^{(q)}u+\Pi^{(q)}u=u+r^{(q)}_1u$ 
for every $u\in{\rm Dom\,}\Box^{(q)}$},\\
&\mbox{$\ddbar\Pi^{(q)}u=r^{(q)}_2u$ for every 
$u\in L^2_{{\rm loc\,}}(U\cap\ol M,T^{*0,q}M')$}, \\
&\mbox{$\ol{\pr}^*_f\Pi^{(q)}u=r^{(q)}_3u$, for every 
$u\in L^2_{{\rm loc\,}}(U\cap\ol M,T^{*0,q}M')$},\\
&\mbox{$\Pi^{(q)}\Box^{(q)}u=r^{(q)}_4u$, for every $u\in{\rm Dom\,}\Box^{(q)}$},\\
&\mbox{$\Box^{(q)}_f\Pi^{(q)}u=r^{(q)}_5u$, for every $u\in\Omega^{0,q}(U\cap\ol M)$},\\
&\mbox{$(\Pi^{(q)})^2u-\Pi^{(q)}u=r^{(q)}_6u$, for every $u\in\Omega^{0,q}(U\cap\ol M)$},
\end{split}
\end{equation}
where $r^{(q)}_j$ is properly supported on $U\cap\ol M $ with 
$r^{(q)}_j\equiv0\mod\cC^\infty((U\times U)\cap(\ol M\times\ol M))$, 
for every $j=0,1,2,3,4,5,6$, and the 
distribution kernel of $\Pi^{(q)}$ satisfies 
\begin{equation}\label{e-gue190531yydI}
\Pi^{(q)}(z, w)\equiv\int^\infty_0e^{i\phi(z, w)t}b(z, w, t)dt 
\mod\cC^\infty((U\times U)\cap(\ol M\times\ol M))
\end{equation}
(for the precise meaning of the oscillatory integral 
$\int^\infty_0e^{i\phi(z, w)t}b(z, w, t)dt$, 
see Remark~\ref{r-gue190531yyd} below) with
\begin{equation}\label{e-gue190531yydII}
\begin{split}
&b(z, w, t)\in S^{n}_{1, 0}((U\times U)\cap(\ol M\times\ol M)\times]0, 
\infty[,T^{*0,q}M'\boxtimes(T^{*0,q}M')^*),\\
&\mbox{$b(z, w, t)\sim\sum^\infty_{j=0}b_j(z, w)t^{n-j}$ in 
$S^{n}_{1, 0}((U\times U)\cap(\ol M\times\ol M)\times]0, 
\infty[,T^{*0,q}M'\boxtimes(T^{*0,q}M')^*)$},\\
\end{split}
\end{equation}
with $b_0(z,z)$ given by \eqref{e-gue190531yyda} below. Moreover, 
\begin{equation}\label{e-gue190531yydIIa}
\begin{split}
&\phi(z, w)\in\cC^\infty((U\times U)\cap(\ol M\times\ol M)),\ \ 
{\rm Im\,}\phi\geq0, \\
&\phi(z, z)=0,\ \ z\in U\cap X,\ \ \phi(z, w)\neq0\ \ \mbox{if}\ \ 
(z, w)\notin{\rm diag\,}((U\times U)\cap(X\times X)), \\
&{\rm Im\,}\phi(z, w)>0\ \ \mbox{if}\ \ (z, w)\notin(U\times U)\cap(X\times X), \\
&\phi(z, w)=-\ol\phi(w, z),\\
&\mbox{$d_z\phi(x,x)=-\omega_0(x)-id\rho(x)$, for every $x\in U\cap X$},
\end{split}\end{equation}
 $\phi(z, w)\in\cC^\infty((U\times U)\cap(\ol M\times\ol M))$ 
 is as in~\cite[Theorem 1.4 of Part II]{Hsiao08} 
 and  $\phi(x, y)=\varphi_-(x, y)$ if $x, y\in U\cap X$, where 
 $\varphi_-(x, y)\in\cC^\infty((U\times U)\cap(X\times X))$ 
 is as in Theorem~\ref{t-gue190527syd}.
\end{thm}

\begin{rem} \label{r-gue190531yyd}
Let $\phi$ and $b(z, w, t)$ be as in Theorem~\ref{t-gue190529ycdi}. Let
$y=(y_1,\ldots,y_{2n-1})$
be local coordinates on $X$ and extend $y_1,\ldots,y_{2n-1}$ 
to real smooth functions in some neighborhood of $X$.
We work with local coordinates
\[w=(y_1,\ldots,y_{2n-1},\rho)\]
defined on some neighborhood $U$ of $p\in X$ in $M'$
Let $u\in\cC^\infty_0(U\cap\ol M)$. Choose a cut-off function $\chi(t)\in\cC^\infty(\Real)$
so that $\chi(t)=1$ when $\abs{t}<1$ and $\chi(t)=0$ when $\abs{t}>2$. Set
\[(\Pi^{(q)}_\epsilon u)(z):=
\int_M\int^\infty_0e^{i\phi(z, w)t}b(z, w, t)\chi(\epsilon t)u(w)dtdv_M(w).
\]
Since $d_y\phi\neq0$ where ${\rm Im\,}\phi=0$ (see \eqref{e-gue190527sydj}),
we can integrate by parts in $y$ and $t$ and obtain
$\lim_{\epsilon\To0}(\Pi^{(q)}_\epsilon u)(z)\in\Omega^{0,q}_0(U\cap\ol M)$.
This means that
$$\Pi^{(q)}=\lim_{\epsilon\To0}\Pi^{(q)}_{\epsilon}: 
\Omega^{0,q}(U\cap\ol M)\To\Omega^{0,q}(U\cap\ol M)$$
is continuous. Then the distribution kernel $\Pi^{(q)}(z, w)$ of $\Pi^{(q)}$ can be
written formally,
$$\Pi^{(q)}(z, w)=\int^\infty_0e^{i\phi(z, w)t}b(z, w, t)dt.$$
\end{rem}
The following was known~\cite[Theorem 1.4 of Part II]{Hsiao08} 
\begin{thm} \label{t-gue190531yyda}
Under the assumptions and notations of Theorem~\ref{t-gue190529ycdi}, fix $p\in U\cap X$.
We choose local holomorphic coordinates
$z=(z_1,\ldots,z_n)$, $z_j=x_{2j-1}+ix_{2j}$, $j=1,\ldots,n$,
vanishing at $p$ such that the metric on $T^{1,0}M'$ is
$\sum^n_{j=1}dz_j\otimes d\ol z_j$ at $p$
and $\rho(z)=\sqrt{2}{\rm Im\,}z_n+\sum^{n-1}_{j=1}\lambda_j\abs{z_j}^2+O(\abs{z}^3)$,
where $\lambda_j$, $j=1,\ldots,n-1$, are the eigenvalues of $\mathcal{L}_p$.
We also write $w=(w_1,\ldots,w_n)$, $w_j=y_{2j-1}+iy_{2j}$, $j=1,\ldots,n$.
Then, we can take $\phi(z, w)$ in \eqref{e-gue190531yydI} so that
\begin{equation} \label{e-gue190531yydh}\begin{split}
\phi(z, w)&=-\sqrt{2}x_{2n-1}+\sqrt{2}y_{2n-1}-
i\rho(z)\Bigr(1+\sum^{2n-1}_{j=1}a_jx_j+\frac{1}{2}a_{2n}x_{2n}\Bigr) \\
&\quad-i\rho(w)\Bigr(1+\sum^{2n-1}_{j=1}\ol a_jy_j+
\frac{1}{2}\ol a_{2n}y_{2n}\Bigr)+
i\sum^{n-1}_{j=1}\abs{\lambda_j}\abs{z_j-w_j}^2 \\
&\quad+\sum^{n-1}_{j=1}i\lambda_j(\ol z_jw_j-z_j\ol w_j)+
O\bigr(\abs{(z, w)}^3\bigr)
\end{split}\end{equation}
in some neighborhood of $(p, p)$ in $M'\times M'$, where
$a_j=\dfrac{1}{2}\dfrac{\pr\sigma(\Box^{(q)}_f)}{\pr x_j}(p, -\omega_0(p)
-id\rho(p))$, $j=1,\ldots,2n$, and $\sigma(\Box^{(q)}_f)$
denotes the principal symbol of  $\Box^{(q)}_f$.
\end{thm}

The following was also known~\cite[Proposition 1.6 of Part II]{Hsiao08}

\begin{thm} \label{t-gue190531yydb}
Under the assumptions and notations of Theorem~\ref{t-gue190529ycdi}, fix $p\in U\cap X$. 
For $b_0(z,w)$ in \eqref{e-gue190531yydII}, we have 
\begin{equation}\label{e-gue190531yyda}
b_0(x, x)=\pi^{-n}\abs{\det\mathcal{L}_{x}}\tau_{x,n_-}
\circ(\ddbar\rho(p))^{\wedge, *}(\ddbar\rho(p))^\wedge,\ \ \mbox{for every $x\in U\cap X$}, 
\end{equation}
where $\det\mathcal{L}_{x}$, $\tau_{x,n_-}$ are given by 
\eqref{e-gue190528yydI}, \eqref{e-gue190528yydIII} respectively and 
$(\ddbar\rho(p))^{\wedge, *}$ is given by \eqref{e-gue190312mscdI}. 
\end{thm}

\section{Microlocal spectral theory for the $\ddbar$-Neumann Laplacian} 
\label{s-gue190531yyda}

In this section, we will apply the microlocal Hodge decomposition theorems 
for $\Box^{(q)}$ from
Section~\ref{s-gue190531yyd} and Section~\ref{s-gue190321myyd} 
to study the singularities for the kernel $B^{(q)}_{\leq\lambda}(x,y)$ near 
the non-degenerate part of the Levi form. 
In particular, we give the proof of Theorem \ref{t-gue190708yyd}.

Until further notice, we fix $\lambda>0$. It is clear that there is a continuous operator
\[A^{(q)}_\lambda:L^2_{(0,q)}(M)\To{\rm Dom\,}\Box^{(q)}\]
such that
\begin{equation}\label{e-gue190531syda}
\begin{split}
&\mbox{$\Box^{(q)}A^{(q)}_\lambda+B^{(q)}_{\leq\lambda}=I$ on $L^2_{(0,q)}(M)$},\\
&\mbox{$A^{(q)}_\lambda\Box^{(q)}+B^{(q)}_{\leq\lambda}=I$ on ${\rm Dom\,}\Box^{(q)}$}.
\end{split}
\end{equation}

Let $U$ be an open set of $M'$ with $U\cap X\neq\emptyset$. 
Suppose that the Levi form is non-degenerate of constant signature 
$(n_-, n_+)$ on $U\cap X$. Until further notice, we let $q=n_-$.

\begin{lem}\label{l-gue190531ycdh}
The map 
\begin{equation}\label{e-gue190531ycdh}
\ddbar B^{(q)}_{\leq\lambda}: L^2_{(0,q)}(M)\To 
H^s_{{\rm loc\,}}(U\cap\ol M, T^{*0,q+1}M')
\end{equation}
is continuous for every $s\in\mathbb N$. 
\end{lem}

\begin{proof}
Let $u\in L^2(M,T^{*0,q}M')$. Since $B^{(q)}_{\leq\lambda}u\in{\rm Dom\,}\Box^{(q)}$, $\ddbar B^{(q)}_{\leq\lambda}u\in L^2_{(0,q+1)}(M)$. We claim that 
\begin{equation}\label{e-gue190531ycdi}
\ddbar B^{(q)}_{\leq\lambda}u\in{\rm Dom\,}\Box^{(q+1)}.
\end{equation}
It is clear that $\ddbar B^{(q)}_{\leq\lambda}u
\in{\rm Dom\,}\ddbar\cap{\rm Dom\,}\ol{\pr}^*$ and 
$\ddbar^2B^{(q)}_{\leq\lambda}u=0$. Hence, 
$\ddbar^2B^{(q)}_{\leq\lambda}u\in{\rm Dom\,}\ol{\pr}^*$. 
We only need to show that 
$\ol{\pr}^*\,\ddbar B^{(q)}_{\leq\lambda}u\in{\rm Dom\,}\ddbar$. 
We have 
\begin{equation}\label{e-gue190531ycdj}
\ol{\pr}^*\,\ddbar B^{(q)}_{\leq\lambda}u=
\Box^{(q)}B^{(q)}_{\leq\lambda}u-\ddbar\,\ol{\pr}^*B^{(q)}_{\leq\lambda}u.
\end{equation}
By spectral theory, we see that 
$\Box^{(q)}B^{(q)}_{\leq\lambda}u\in{\rm Dom\,}\Box^{(q)}$ 
and hence $\Box^{(q)}B^{(q)}_{\leq\lambda}u\in{\rm Dom\,}\ddbar$. 
Note that $\ddbar^2\ol{\pr}^*B^{(q)}_{\leq\lambda}u=0$, 
$\ddbar\,\ol{\pr}^*B^{(q)}_{\leq\lambda}u\in{\rm Dom\,}\ddbar$. 
From this observation and \eqref{e-gue190531ycdj}, we get \eqref{e-gue190531ycdi}. 

From \eqref{e-gue190403yydIz}, we have 
\begin{equation}\label{e-gue190531scdl}
N^{(q+1)}\Box^{(q+1)}\ddbar B^{(q)}_{\leq\lambda}u=
\ddbar B^{(q)}_{\leq\lambda}u+F^{(q+1)}_1\ddbar B^{(q)}_{\leq\lambda}u,
\end{equation}
where $F^{(q+1)}_1\equiv0\mod\cC^\infty((U\times U)\cap(\ol M\times\ol M))$ 
is as in \eqref{e-gue190403yydIz}. It is clear that 
\begin{equation}\label{e-gue190531scd}
F^{(q+1)}_1\ddbar: L^2_{{\rm loc\,}}(U\cap\ol M, T^{*0,q}M')\To 
H^s_{{\rm loc\,}}(U\cap\ol M, T^{*0,q+1}M')
\end{equation}
is continuous, for every $s\in\mathbb Z$. We have 
\begin{equation}\label{e-gue190531scdI}
\mbox{$N^{(q+1)}\Box^{(q+1)}\ddbar B^{(q)}_{\leq\lambda}=
N^{(q+1)}\ddbar\Box^{(q)}B^{(q)}_{\leq\lambda}$ on $L^2_{(0,q)}(M)$}. 
\end{equation}
By spectral theory, 
\begin{equation}\label{e-gue190531scdII}
\Box^{(q)}B^{(q)}_{\leq\lambda}: L^2_{(0,q)}(M)\To L^2_{(0,q)}(M)
\end{equation}
is continuous. In view of Theorem~\ref{t-gue190321myyd}, we see that 
\begin{equation}\label{e-gue190531ycdk}
N^{(q+1)}\ddbar: H^s_{{\rm loc\,}}(U\cap\ol M, T^{*0,q}M')\To 
H^s_{{\rm loc\,}}(U\cap\ol M, T^{*0,q+1}M')
\end{equation}
is continuous, for every $s\in\mathbb Z$. From \eqref{e-gue190531scdl}, 
\eqref{e-gue190531scd}, \eqref{e-gue190531scdI}, \eqref{e-gue190531scdII} 
and \eqref{e-gue190531ycdk}, we deduce that 
\begin{equation}\label{e-gue190531ycdl}
\ddbar B^{(q)}_{\leq\lambda}: L^2_{(0,q)}(M)\To L^2_{{\rm loc\,}}(U\cap\ol M, T^{*0,q+1}M')
\end{equation}
is continuous. We have 
\[N^{(q+1)}\Box^{(q+1)}\ddbar B^{(q)}_{\leq\lambda}u=
N^{(q+1)}\ddbar\Box^{(q)}B^{(q)}_{\leq\lambda}u=
N^{(q+1)}\ddbar B^{(q)}_{\leq\lambda}\Box^{(q)}B^{(q)}_{\leq\lambda}u.\]
From this observation and \eqref{e-gue190531scdl}, we have 
\begin{equation}\label{e-gue190531scdm}
N^{(q+1)}\ddbar B^{(q)}_{\leq\lambda}\Box^{(q)}B^{(q)}_{\leq\lambda}u=
\ddbar B^{(q)}_{\leq\lambda}u+F^{(q)}_1\ddbar B^{(q)}_{\leq\lambda}u.
\end{equation}
From \eqref{e-gue190531scdm}, \eqref{e-gue190531scd}, 
\eqref{e-gue190531scdII}, \eqref{e-gue190531ycdl} and note that 
\begin{equation}\label{e-gue190531yydz}
N^{(q+1)}: H^s_{{\rm loc\,}}(U\cap\ol M, T^{*0,q+1}M')\To 
H^{s+1}_{{\rm loc\,}}(U\cap\ol M, T^{*0,q+1}M')
\end{equation}
is continuous, for every $s\in\mathbb Z$, we deduce that 
\begin{equation}\label{e-gue190531yydp}
\ddbar B^{(q)}_{\leq\lambda}: L^2_{(0,q)}(M)\To 
H^1_{{\rm loc\,}}(U\cap\ol M, T^{*0,q+1}M')
\end{equation}
is continuous. Continuing in this way, we get the lemma. 
\end{proof}

We can repeat the proof of Lemma~\ref{l-gue190531ycdh} and deduce:

\begin{lem}\label{l-gue190603yyd}
The operator
\begin{equation}\label{e-gue190603yyd}
\ol{\pr}^*B^{(q)}_{\leq\lambda}: L^2_{(0,q)}(M)\To 
H^s_{{\rm loc\,}}(U\cap\ol M, T^{*0,q}M'),
\end{equation}
is continuous for every $s\in\mathbb N$. 
\end{lem}

From \eqref{e-gue190531ycdh} and \eqref{e-gue190603yyd}, we get: 

\begin{thm}\label{t-gue190603yyd}
The operator
\begin{equation}\label{e-gue190603yydI}
\Box^{(q)}B^{(q)}_{\leq\lambda}: L^2_{(0,q)}(M)\To 
H^s_{{\rm loc\,}}(U\cap\ol M, T^{*0,q}M'),
\end{equation}
is continuous for every $s\in\mathbb N$. 
\end{thm}
\begin{lem}\label{l-gue190603yydI} 
For every $m\in\mathbb N$, the operator 
\[B^{(q)}_{\leq\lambda}\ddbar(\Box^{(q-1)}_f)^m: \Omega^{0,q-1}_0(M)\To L^2_{(0,q)}(M)\]
can be continuously extended to 
\begin{equation}\label{e-gue190603yydII}
B^{(q)}_{\leq\lambda}\ddbar(\Box^{(q-1)}_f)^m: L^2_{(0,q-1)}(M)\To L^2_{(0,q)}(M).
\end{equation}
\end{lem} 

\begin{proof}
Let $u\in\Omega^{0,q-1}_0(M)$, $v\in L^2_{(0,q)}(M)$. we have 
\begin{equation}\label{e-gue190604yydIII}
(\,B^{(q)}_{\leq\lambda}\ddbar(\Box^{(q-1)}_f)^mu\,|\,v\,)_M
=(\,B^{(q)}_{\leq\lambda}(\Box^{(q)}_f)^m\ddbar u\,|\,v\,)_M=
(\,u\,|\,\ol{\pr}^*(\Box^{(q})^mB^{(q)}_{\leq\lambda}v\,)_M.
\end{equation}
We have 
\begin{equation}\label{e-gue190604yyda}
\begin{split}
&\norm{\ol{\pr}^*(\Box^{(q)})^mB^{(q)}_{\leq\lambda}v}^2_M
\leq\norm{\ol{\pr}^*(\Box^{(q)})^m
B^{(q)}_{\leq\lambda}v}^2_M+
\norm{\ddbar(\Box^{(q)})^mB^{(q)}_{\leq\lambda}v}^2_M\\
&=\Big(\,(\Box^{(q)})^{m+1}B^{(q)}_{\leq\lambda}v\,|\,
(\Box^{(q)})^mB^{(q)}_{\leq\lambda}v\,\Big)_M
\leq\lambda^{2m+1}\norm{v}^2_M.
\end{split}
\end{equation}
From \eqref{e-gue190604yydIII}, \eqref{e-gue190604yyda} 
and take $v=B^{(q)}_{\leq\lambda}\ddbar(\Box^{(q-1)}_f)^mu$, 
it is straightforward to see that 
\begin{equation}\label{e-gue190604yydb}
\norm{B^{(q)}_{\leq\lambda}\ddbar(\Box^{(q-1)}_f)^mu}_M\leq
\lambda^{m+\frac{1}{2}}\norm{u}_M. 
\end{equation}
From \eqref{e-gue190604yydb} and notice that 
$\Omega^{0,q-1}_0(M)$ is dense in $L^2_{(0,q-1)}(M)$, the lemma follows. 
\end{proof} 

\begin{lem}\label{l-gue190603syd}
The operator $B^{(q)}_{\leq\lambda}\ddbar: 
\Omega^{0,q-1}_0(U\cap\ol M)\To L^2_{(0,q)}(M)$ can be continuously extended to 
\begin{equation}\label{e-gue190603syd}
B^{(q)}_{\leq\lambda}\ddbar: 
H^{-s}_{{\rm comp\,}}(U\cap\ol M, T^{*0,q-1}M')\To L^2_{(0,q)}(M),\ \ 
\mbox{for every $s\in\mathbb N$}.
\end{equation}
\end{lem}

\begin{proof}
Let $u\in\Omega^{0,q-1}_0(U\cap\ol M)$. From \eqref{e-gue190403yydq}, we have 
\begin{equation}\label{e-gue190603sydI}
B^{(q)}_{\leq\lambda}\ddbar\Box^{(q-1)}_fN^{(q-1)}u=
B^{(q)}_{\leq\lambda}\ddbar u+B^{(q)}_{\leq\lambda}\ddbar F^{(q-1)}_2u, 
\end{equation}
where $F^{(q-1)}_2\equiv0\mod\cC^\infty((U\times U)\cap(\ol M\times\ol M))$ 
is as in \eqref{e-gue190403yydq}. From \eqref{e-gue190603yydII}, 
\eqref{e-gue190603sydI} and notice that 
\begin{equation}\label{e-gue190603sydII}
N^{(q-1)}: H^{s}_{{\rm comp\,}}(U\cap\ol M, T^{*0,q-1}M')\To 
H^{s+1}_{{\rm comp\,}}(U\cap\ol M, T^{*0,q-1}M')
\end{equation}
is continuous, for every $s\in\mathbb Z$, we deduce that  
$B^{(q)}_{\leq\lambda}\ddbar$ can be continuously extended to 
\begin{equation}\label{e-gue190603sydIII}
B^{(q)}_{\leq\lambda}\ddbar: 
H^{-1}_{{\rm comp\,}}(U\cap\ol M, T^{*0,q-1}M')\To L^2_{(0,q)}(M).
\end{equation}
From Lemma~\ref{l-gue190603yydI}, we can repeat the proof of 
\eqref{e-gue190603sydIII} and deduce that 
$B^{(q)}_{\leq\lambda}\ddbar\Box^{(q-1)}_f$ can be continuously extended to 
\begin{equation}\label{e-gue190603yydh}
B^{(q)}_{\leq\lambda}\ddbar\Box^{(q-1)}_f: 
H^{-1}_{{\rm comp\,}}(U\cap\ol M, T^{*0,q-1}M')\To L^2_{(0,q)}(M).
\end{equation}
From \eqref{e-gue190603sydI}, \eqref{e-gue190603sydII} 
and \eqref{e-gue190603yydh}, we deduce that 
$B^{(q)}_{\leq\lambda}\ddbar$ can be continuously extended to 
\[B^{(q)}_{\leq\lambda}\ddbar: 
H^{-2}_{{\rm comp\,}}(U\cap\ol M, T^{*0,q-1}M')\To L^2_{(0,q)}(M).\]
Continuing in this way, we get the lemma.
\end{proof}

We can repeat the proof of Lemma~\ref{l-gue190603syd} and deduce 

\begin{lem}\label{l-gue190603yydh}
The operator $B^{(q)}_{\leq\lambda}\ol{\pr}^*_f: 
\Omega^{0,q+1}_0(U\cap\ol M)\To L^2_{(0,q)}(M)$ can be continuously extended to 
\begin{equation}\label{e-gue190603syds}
B^{(q)}_{\leq\lambda}\ol{\pr}^*_f: 
H^{-s}_{{\rm comp\,}}(U\cap\ol M, T^{*0,q+1}M')\To 
L^2_{(0,q)}(M),\ \ \mbox{for every $s\in\mathbb N$}.
\end{equation}
\end{lem}

From Lemma~\ref{l-gue190603syd} and Lemma~\ref{l-gue190603yydh}, we get 

\begin{thm}\label{t-gue190603sydk}
The operator $B^{(q)}_{\leq\lambda}\Box^{(q)}_f:
\Omega^{0,q}_0(U\cap\ol M)\To L^2_{(0,q)}(M)$ can be continuously extended to 
\begin{equation}\label{e-gue190603sydl}
B^{(q)}_{\leq\lambda}\Box^{(q)}_f: 
H^{-s}_{{\rm comp\,}}(U\cap\ol M, T^{*0,q}M')\To 
L^2_{(0,q)}(M),\ \ \mbox{for every $s\in\mathbb N$}. 
\end{equation}
\end{thm}

We consider 
\[\begin{split}
&\Box^{(q)}B^{(q)}_{\leq\lambda}\Box^{(q)}_f: 
\Omega^{0,q}_0(U\cap\ol M)\To L^2_{(0,q)}(M)\subset 
L^2_{{\rm loc\,}}(U\cap\ol M, T^{*0,q}M'),\\
&(\Box^{(q)})^2B^{(q)}_{\leq\lambda}: 
\Omega^{0,q}_0(U\cap\ol M)\To L^2_{(0,q)}(M)\subset 
L^2_{{\rm loc\,}}(U\cap\ol M, T^{*0,q}M').
\end{split}\]

\begin{thm}\label{t-gue190604yyd}
We have 
\begin{equation}\label{e-gue190604syd}
\Box^{(q)}B^{(q)}_{\leq\lambda}\Box^{(q)}_f\equiv0
\mod\cC^\infty((U\times U)\cap(\ol M\times\ol M))
\end{equation}
and 
\begin{equation}\label{e-gue190604sydI}
(\Box^{(q)})^2B^{(q)}_{\leq\lambda}\equiv0
\mod\cC^\infty((U\times U)\cap(\ol M\times\ol M)).
\end{equation}
\end{thm} 

\begin{proof}
From \eqref{e-gue190603yydI} and \eqref{e-gue190603sydl}, we have 
\[\Box^{(q)}B^{(q)}_{\leq\lambda}\Box^{(q)}_f: 
H^{-s}_{{\rm comp\,}}(U\cap\ol M, T^{*0,q}M')\To 
H^s_{{\rm loc\,}}(U\cap\ol M,T^{*0,q}M'),\]
for every $s\in\mathbb N$. We get \eqref{e-gue190604syd}. 
Let $u\in L^2_{(0,q)}(M)$. Take $u_j\in\Omega^{0,q}_0(M)$, 
$j=1,2,\ldots$, so that $\lim_{j\To+\infty}\norm{u_j-u}_M=0$. 
Since $(\Box^{(q)})^2B^{(q)}_{\leq\lambda}$ is $L^2$ continuous, we have 
\begin{equation}\label{e-gue190604yydj}
(\Box^{(q)})^2B^{(q)}_{\leq\lambda}u=
\lim_{j\To+\infty}(\Box^{(q)})^2B^{(q)}_{\leq\lambda}u_j\ \ 
\mbox{in $L^2_{(0,q)}(M)$}. 
\end{equation}
From the fact that $u_j\in{\rm Dom\,}\Box^{(q)}$, for every 
$j=1,2,\ldots$, we can check that 
\begin{equation}\label{e-gue190604yydk}
(\Box^{(q)})^2B^{(q)}_{\leq\lambda}u_j=
\Box^{(q)}B^{(q)}_{\leq\lambda}\Box^{(q)}u_j=
\Box^{(q)}B^{(q)}_{\leq\lambda}\Box^{(q)}_fu_j,\ \ 
\mbox{for every $j=1,2,\ldots$}. 
\end{equation}
From \eqref{e-gue190604yydk} and \eqref{e-gue190604yydj}, we conclude that 
\begin{equation}\label{e-gue190604yydl}
(\Box^{(q)})^2B^{(q)}_{\leq\lambda}=
\Box^{(q)}B^{(q)}_{\leq\lambda}\Box^{(q)}_f\ \ 
\mbox{on $L^2_{(0,q)}(M)$}. 
\end{equation}
From \eqref{e-gue190604yydl} and \eqref{e-gue190604syd}, 
we get \eqref{e-gue190604sydI}. 
\end{proof} 

\begin{lem}\label{l-gue190604sydk}
The operator $B^{(q)}_{\leq\lambda}$ can be continuously extended to 
\begin{equation}\label{e-gue190604ycd}
B^{(q)}_{\leq\lambda}: 
H^{-s}_{{\rm comp\,}}(U\cap\ol M, T^{*0,q}M')\To 
H^{-s}_{{\rm loc\,}}(U\cap\ol M, T^{*0,q}M'), 
\end{equation}
for every $s\in\mathbb N$.
\end{lem}

\begin{proof}
Let $u\in\Omega^{0,q}_0(U\cap\ol M)$. From \eqref{e-gue190531yyd}, we have 
\begin{equation}\label{e-gue190604ycdI}
B^{(q)}_{\leq\lambda}\Box^{(q)}_fN^{(q)}u+B^{(q)}_{\leq\lambda}\Pi^{(q)}u=
B^{(q)}_{\leq\lambda}u+B^{(q)}_{\leq\lambda}r^{(q)}_0u,
\end{equation}
where $r^{(q)}_0\equiv0\mod\cC^\infty((U\times U)\cap(\ol M\times\ol M))$ 
is as in \eqref{e-gue190531yyd}. From \eqref{e-gue190604ycdI}, 
\eqref{e-gue190529syd}, \eqref{e-gue190603sydl} and note that 
$\Omega^{0,q}_0(U\cap\ol M)$ is dense in 
$H^{-s}_{{\rm comp\,}}(U\cap\ol M, T^{*0,q}M')$, 
for every $s\in\mathbb N$, we deduce that 
$B^{(q)}_{\leq\lambda}-B^{(q)}_{\leq\lambda}\Pi^{(q)}$ 
can be continuously extended to 
\begin{equation}\label{e-gue190604ycdII}
B^{(q)}_{\leq\lambda}-B^{(q)}_{\leq\lambda}\Pi^{(q)}: 
H^{-s}_{{\rm comp\,}}(U\cap\ol M, T^{*0,q}M')\To L^2_{(0,q)}(M),\ \ 
\mbox{for every $s\in\mathbb N$}.
\end{equation}
On the other hand, from \eqref{e-gue190531syda} and \eqref{e-gue190531yyd}, 
we have 
\begin{equation}\label{e-gue190606yyd}
\begin{split}
&\Pi^{(q)}u=(A^{(q)}_{\lambda}\Box^{(q)}+B^{(q)}_{\leq\lambda})\Pi^{(q)}u\\
&=A^{(q)}_{\lambda}\Box^{(q)}_f\Pi^{(q)}u+B^{(q)}_{\leq\lambda}\Pi^{(q)}u\\
&=A^{(q)}_\lambda r^{(q)}_5u+B^{(q)}_{\leq\lambda}\Pi^{(q)}u,
\end{split}
\end{equation}
for every $u\in\Omega^{0,q}_0(U\cap\ol M)$, where 
$r^{(q)}_5\equiv0\mod\cC^\infty((U\times U)\cap(\ol M\times\ol M))$ 
is as in \eqref{e-gue190531yyd}.
From \eqref{e-gue190606yyd}, we conclude that 
$\Pi^{(q)}-B^{(q)}_{\leq\lambda}\Pi^{(q)}$ can be continuously extended to 
\begin{equation}\label{e-gue190606yydI}
\Pi^{(q)}-B^{(q)}_{\leq\lambda}\Pi^{(q)}: 
H^{-s}_{{\rm comp\,}}(U\cap\ol M, T^{*0,q}M')\To L^2_{(0,q)}(M),\ \ 
\mbox{for every $s\in\mathbb N$}.
\end{equation}
From \eqref{e-gue190604ycdII} and \eqref{e-gue190606yydI}, 
we deduce that $\Pi^{(q)}-B^{(q)}_{\leq\lambda}$ can be continuously extended to 
\begin{equation}\label{e-gue190606yydII}
\Pi^{(q)}-B^{(q)}_{\leq\lambda}: 
H^{-s}_{{\rm comp\,}}(U\cap\ol M, T^{*0,q}M')\To L^2_{(0,q)}(M),\ \ 
\mbox{for every $s\in\mathbb N$}.
\end{equation}
From \eqref{e-gue190606yydII} and \eqref{e-gue190529syd}, we get \eqref{e-gue190604ycd}.
\end{proof}
\begin{thm}\label{t-gue190606yyd}
We have 
\begin{equation}\label{e-gue190606yydIII}
\Box^{(q)}B^{(q)}_{\leq\lambda}\equiv0\mod\cC^\infty((U\times U)\cap(\ol M\times\ol M)).
\end{equation}
\end{thm}
\begin{proof}
Let $\varepsilon^{(q)}:=(\Box^{(q)})^2B^{(q)}_{\leq\lambda}$. 
In view of \eqref{e-gue190604sydI}, we see that 
$\varepsilon^{(q)}\equiv0\mod\cC^\infty((U\times U)\cap(\ol M\times\ol M))$. 
Let $u\in\Omega^{0,q}_0(U\cap\ol M)$. 
From second equation in \eqref{e-gue190531yyd}, we have 
\begin{equation}\label{e-gue190606yyda}
\begin{split}
&\Box^{(q)}B^{(q)}_{\leq\lambda}u\\
&=N^{(q)}(\Box^{(q)})^2B^{(q)}_{\leq\lambda}u+
\Pi^{(q)}\Box^{(q)}B^{(q)}_{\leq\lambda}u-r^{(q)}_1\Box^{(q)}B^{(q)}_{\leq\lambda}u\\
&=N^{(q)}\varepsilon^{(q)}u+r^{(q)}_4B^{(q)}_{\leq\lambda}u-r^{(q)}_1
\Box^{(q)}B^{(q)}_{\leq\lambda}u,
\end{split}
\end{equation}
where $r^{(q)}_4\equiv0\mod\cC^\infty((U\times U)\cap(\ol M\times\ol M))$, 
$r^{(q)}_1\equiv0\mod\cC^\infty((U\times U)\cap(\ol M\times\ol M))$, 
are as in \eqref{e-gue190531yyd}. From \eqref{e-gue190604ycd}, we see that 
\[r^{(q)}_1\Box^{(q)}B^{(q)}_{\leq\lambda}, r^{(q)}_4B^{(q)}_{\leq\lambda}: 
H^{-s}_{{\rm comp\,}}(U\cap\ol M, T^{*0,q}M')\To 
H^{s}_{{\rm loc\,}}(U\cap\ol M, T^{*0,q}M')\]
are continuous, for every $s\in\mathbb N$, 
and hence $r^{(q)}_1\Box^{(q)}B^{(q)}_{\leq\lambda}, 
r^{(q)}_4B^{(q)}_{\leq\lambda}\equiv0\mod\cC^\infty((U\times U)\cap(\ol M\times\ol M))$. 
From this observation and \eqref{e-gue190606yyda}, we get \eqref{e-gue190606yydIII}. 
\end{proof}
We can now prove one of the main results of this work. 
\begin{thm}\label{t-gue190606ycd}
Let $U$ be an open set of $M'$ with $U\cap X\neq\emptyset$. 
Suppose that the Levi form is non-degenerate of constant signature 
$(n_-, n_+)$ on $U\cap X$. Let $q=n_-$ and
fix $\lambda>0$. We have 
\[B^{(q)}_{\leq\lambda}-\Pi^{(q)}\equiv0
\mod\cC^\infty((U\times U)\cap(\ol M\times\ol M)),\]
where $\Pi^{(q)}$ is as in Theorem~\ref{t-gue190529ycdi}. 
\end{thm}

\begin{proof}
From the second equation in \eqref{e-gue190531yyd}, we have 
\begin{equation}\label{e-gue190606ycdk}
N^{(q)}\Box^{(q)}B^{(q)}_{\leq\lambda}u+
\Pi^{(q)}B^{(q)}_{\leq\lambda}u=r^{(q)}_1B^{(q)}_{\leq\lambda}u+
B^{(q)}_{\leq\lambda}u,
\end{equation}
for every $u\in\Omega^{0,q}_0(U\cap\ol M)$, where 
$r^{(q)}_1\equiv0\mod\cC^\infty((U\times U)\cap(\ol M\times\ol M))$ 
is as in \eqref{e-gue190531yyd}. 
From \eqref{e-gue190606ycdk}, \eqref{e-gue190606yydIII} and 
\eqref{e-gue190604ycd}, we deduce that 
\begin{equation}\label{e-gue190606syda}
B^{(q)}_{\leq\lambda}-\Pi^{(q)}B^{(q)}_{\leq\lambda}
=:\epsilon^{(q)}\equiv0\mod\cC^\infty((U\times U)\cap(\ol M\times\ol M)). 
\end{equation}
Similarly, from first equation in \eqref{e-gue190531yyd}, we have 
\begin{equation}\label{e-gue190606sydb}
B^{(q)}_{\leq\lambda}\Box^{(q)}_fN^{(q)}u+B^{(q)}_{\leq\lambda}\Pi^{(q)}u
=B^{(q)}_{\leq\lambda}u+B^{(q)}_{\leq\lambda}r^{(q)}_0u, 
\end{equation}
for every $u\in\Omega^{0,q}_0(U\cap\ol M)$, 
where $r^{(q)}_0\equiv0\mod\cC^\infty((U\times U)\cap(\ol M\times\ol M))$ 
is as in \eqref{e-gue190531yyd}. Since $N^{(q)}u\in{\rm Dom\,}\Box^{(q)}$, 
we have 
$$B^{(q)}_{\leq\lambda}\Box^{(q)}_fN^{(q)}u=
B^{(q)}_{\leq\lambda}\Box^{(q)}N^{(q)}u=
\Box^{(q)}B^{(q)}_{\leq\lambda}N^{(q)}u,$$ 
for every $u\in\Omega^{0,q}_0(U\cap\ol M)$. From this observation and \eqref{e-gue190606sydb}, we deduce that 
\begin{equation}\label{e-gue190606sydc}
\Box^{(q)}B^{(q)}_{\leq\lambda}N^{(q)}u+B^{(q)}_{\leq\lambda}\Pi^{(q)}u=B^{(q)}_{\leq\lambda}u+B^{(q)}_{\leq\lambda}r^{(q)}_0u, 
\end{equation}
for every $u\in\Omega^{0,q}_0(U\cap\ol M)$. From \eqref{e-gue190606yydIII}, \eqref{e-gue190604ycd} and \eqref{e-gue190606sydc}, we deduce that 
\begin{equation}\label{e-gue190606sydd}
B^{(q)}_{\leq\lambda}-B^{(q)}_{\leq\lambda}\Pi^{(q)}=:\epsilon^{(q)}_1\equiv0\mod\cC^\infty((U\times U)\cap(\ol M\times\ol M)).
\end{equation}
Let $u\in\Omega^{0,q}_0(U\cap\ol M)$. From \eqref{e-gue190531syda}, we have 
\begin{equation}\label{e-gue190606syde}
\Pi^{(q)}\Box^{(q)}A^{(q)}_\lambda u+\Pi^{(q)}B^{(q)}_{\leq\lambda}u=\Pi^{(q)}u\ \ \mbox{on $U\cap\ol M$}
\end{equation}
and 
\begin{equation}\label{e-gue190606sydf}
A^{(q)}_\lambda\Box^{(q)}\Pi^{(q)}u+B^{(q)}_{\leq\lambda}\Pi^{(q)}u=\Pi^{(q)}u\ \ \mbox{on $U\cap\ol M$}.
\end{equation}
From \eqref{e-gue190531yyd}, we have 
\begin{equation}\label{e-gue190606sydg}
\Pi^{(q)}\Box^{(q)}A^{(q)}_\lambda u=r^{(q)}_4A^{(q)}_\lambda u\ \ \mbox{on $U\cap X$},
\end{equation}
\begin{equation}\label{e-gue190606sydh}
A^{(q)}_\lambda\Box^{(q)}\Pi^{(q)}u=A^{(q)}_\lambda r^{(q)}_5u\ \ \mbox{on $U\cap X$},
\end{equation}
where $r^{(q)}_4\equiv0\mod\cC^\infty((U\times U)\cap(\ol M\times\ol M))$ and 
$r^{(q)}_5\equiv0\mod\cC^\infty((U\times U)\cap(\ol M\times\ol M))$ are as in \eqref{e-gue190531yyd}. 
From \eqref{e-gue190606sydg}, \eqref{e-gue190606sydh} and \eqref{e-gue190606sydf}, we deduce that 
\begin{equation}\label{e-gue190606sydl}
\begin{split}
&\Pi^{(q)}-\Pi^{(q)}B^{(q)}_{\leq\lambda}=r^{(q)}_4A^{(q)}_\lambda,\\
&\Pi^{(q)}-B^{(q)}_{\leq\lambda}\Pi^{(q)}=A^{(q)}_\lambda r^{(q)}_5.
\end{split}
\end{equation}
From \eqref{e-gue190606syda}, \eqref{e-gue190606sydd} and \eqref{e-gue190606sydl}, we get 
\begin{equation}\label{e-gue190606sydm}
\begin{split}
&\Pi^{(q)}-B^{(q)}_{\leq\lambda}=r^{(q)}_4A^{(q)}_\lambda-\epsilon^{(q)},\\
&\Pi^{(q)}-B^{(q)}_{\leq\lambda}=A^{(q)}_\lambda r^{(q)}_5-\epsilon^{(q)}_1.
\end{split}
\end{equation}
From \eqref{e-gue190606sydm}, we have 
\begin{equation}\label{e-gue190606scd}
\begin{split}
&(\Pi^{(q)}-B^{(q)}_{\leq\lambda})(\Pi^{(q)}-B^{(q)}_{\leq\lambda})\\
&=(r^{(q)}_4A^{(q)}_\lambda-\epsilon^{(q)})(A^{(q)}_\lambda r^{(q)}_5-\epsilon^{(q)}_1)\\
&=r^{(q)}_4(A^{(q)}_\lambda)^2r^{(q)}_5-r^{(q)}_4A^{(q)}_\lambda\epsilon^{(q)}_1-\epsilon^{(q)}A^{(q)}_\lambda r^{(q)}_5+\epsilon^{(q)}\epsilon^{(q)}_1\ \ \mbox{on $\Omega^{0,q}_0(U\cap\ol M)$}. 
\end{split}\end{equation}
Note that $r^{(q)}_5$ and $r^{(q)}_4$ are properly supported on $U\cap\ol M$ and $r^{(q)}_4(A^{(q)}_\lambda)^2r^{(q)}_5$, $r^{(q)}_4A^{(q)}_\lambda\epsilon^{(q)}_1$, 
$\epsilon^{(q)}A^{(q)}_\lambda r^{(q)}_5$,  $\epsilon^{(q)}\epsilon^{(q)}_1$, are well-defined as continuous operators: $\Omega^{0,q}_0(U\cap\ol M)\To\Omega^{0,q}(U\cap\ol M)$. Now, 
\[\begin{split}
r^{(q)}_4(A^{(q)}_\lambda)^2r^{(q)}_5:&H^{-s}_{{\rm comp\,}}(U\cap\ol M, T^{*0,q}M')\To H^s_{{\rm comp\,}}(U\cap\ol M, T^{*0,q}M')\subset L^2_{(0,q)}(M)\\
&\To L^2_{(0,q)}(M)\To H^s_{{\rm loc\,}}(U\cap\ol M, T^{*0,q}M')\end{split}\]
is continuous, for every $s\in\mathbb N$. Hence, $r^{(q)}_4(A^{(q)}_\lambda)^2r^{(q)}_5\equiv0\mod\cC^\infty((U\times U)\cap(\ol M\times\ol M))$. Similarly, 
$r^{(q)}_4A^{(q)}_\lambda\epsilon^{(q)}_1$, $\epsilon^{(q)}A^{(q)}_\lambda r^{(q)}_5$, $\epsilon^{(q)}\epsilon^{(q)}_1\equiv0\mod\cC^\infty((U\times U)\cap(\ol M\times\ol M))$. From this observation and \eqref{e-gue190606scd}, we get 
\begin{equation}\label{e-gue190606yydz}
(\Pi^{(q)}-B^{(q)}_{\leq\lambda})(\Pi^{(q)}-B^{(q)}_{\leq\lambda})\equiv0\mod\cC^\infty((U\times U)\cap(\ol M\times\ol M)). 
\end{equation}
Now, 
\begin{equation}\label{e-gue190606sydz}
\begin{split}
&(\Pi^{(q)}-B^{(q)}_{\leq\lambda})(\Pi^{(q)}-B^{(q)}_{\leq\lambda})\\
&=(\Pi^{(q)})^2-\Pi^{(q)}B^{(q)}_{\leq\lambda}-
B^{(q)}_{\leq\lambda}\Pi^{(q)}+(B^{(q)}_{\leq\lambda})^2\\
&=\Pi^{(q)}+r^{(q)}_6-B^{(q)}_{\leq\lambda}+
\epsilon^{(q)}-B^{(q)}_{\leq\lambda}+\epsilon^{(q)}_1+
B^{(q)}_{\leq\lambda}\\
&=\Pi^{(q)}-B^{(q)}_{\leq\lambda}+r^{(q)}_6+
\epsilon^{(q)}+\epsilon^{(q)}_1,
\end{split}
\end{equation}
where $r^{(q)}_6, \epsilon^{(q)}, \epsilon^{(q)}_1\equiv0
\mod\cC^\infty((U\times U)\cap(\ol M\times\ol M))$ are as in 
\eqref{e-gue190531yyd}, \eqref{e-gue190606syda} and 
\eqref{e-gue190606sydd} respectively. From \eqref{e-gue190606yydz} 
and \eqref{e-gue190606sydz}, the theorem follows. 
\end{proof} 

By using Theorem~\ref{t-gue190531syd}, we can repeat the proof of 
Theorem~\ref{t-gue190606ycd} with minor change and deduce 

\begin{thm}\label{t-gue190810yyd}
Let $U$ be an open set of $M'$ with $U\cap X\neq\emptyset$. 
Suppose that the Levi form is non-degenerate of constant signature 
$(n_-, n_+)$ on $U\cap X$. Let $q\neq n_-$. Fix $\lambda>0$. We have 
\[B^{(q)}_{\leq\lambda}\equiv0
\mod\cC^\infty((U\times U)\cap(\ol M\times\ol M)).\]
\end{thm}
\begin{proof}[Proof of Theorem \ref{t-gue190708yyd}]
This follows immediately from Theorem~\ref{t-gue190606ycd} and 
Theorem~\ref{t-gue190810yyd}. 
\end{proof}
We remind the reader that the local closed range condition is given by Definition~\ref{d-gue190609yyd}. 
The following is our second main result. 

\begin{thm}\label{t-gue190609yyd}
Let $U$ be an open set of $M'$ with $U\cap X\neq\emptyset$. 
Assume that the Levi form is non-degenerate of constant signature $(n_-, n_+)$ 
on $U\cap X$. 
Let $q=n_-$. Suppose that $\Box^{(q)}$ has local closed range in $U$. Then
the Bergman projection $B^{(q)}$ satisfies
\[B^{(q)}-\Pi^{(q)}\equiv0\mod\cC^\infty((U\times U)\cap(\ol M\times\ol M)),\]
where $\Pi^{(q)}$ is as in Theorem~\ref{t-gue190529ycdi}. 
\end{thm}

\begin{proof}
Let $W$ be any open set of $U$ with $W\cap U\neq\emptyset$, 
$\ol W\Subset U$. Since $\Pi^{(q)}$ is properly supported on $U\cap\ol M$, 
there is an open set $W'\subset U$ with $W'\cap X\neq\emptyset$, 
$\ol{W'}\Subset U$, such that 
$\Pi^{(q)}u\in\Omega^{0,q}_0(W'\cap\ol M)\cap{\rm Dom\,}\Box^{(q)}$
for every $u\in\Omega^{0,q}_0(W\cap\ol M)$. 
Since $\Box^{(q)}$ has local closed range on $U$, there is a constant 
$C_{W'}>0$ such that 
\begin{equation}\label{e-gue190609yyd}
\norm{(I-B^{(q)})\Pi^{(q)}u}_M\leq 
C_{W'}\norm{\Box^{(q)}\Pi^{(q)}u}_M=\norm{r^{(q)}_5u}_M,\ \ 
\mbox{for every $u\in\Omega^{0,q}_0(W\cap\ol M)$},
\end{equation}
where $r^{(q)}_5\equiv0\mod\cC^\infty((U\times U)\cap(\ol M\times\ol M))$ 
is as in \eqref{e-gue190531yyd}. 
Let $u\in H^{-s}_{{\rm comp\,}}(W\cap\ol M, T^{*0,q}M')$. 
Take $u_j\in\Omega^{0,q}_0(W\cap\ol M)$, $j=1,2,\ldots$, 
$\lim_{j\To+\infty}u_j=u$ in $H^{-s}_{{\rm comp\,}}(W\cap\ol M, T^{*0,q}M')$. 
From \eqref{e-gue190609yyd}, we have 
\[\norm{(I-B^{(q)})(\Pi^{(q)}(u_j-u_k))}_M\leq 
C_{W'}\norm{r^{(q)}_5(u_j-u_k)}_M,\ \ 
\mbox{for every $j, k=1,2,\ldots$}.\]
Since $r^{(q)}_5\equiv0\mod\cC^\infty((U\times U)\cap(\ol M\times\ol M))$, 
$\lim_{j,k\To+\infty}\norm{r^{(q)}_5(u_j-u_k)}_M=0$. Hence, 
$(\Pi^{(q)}-B^{(q)}\Pi^{(q)})u_j$ converges in $L^2_{(0,q)}(M)$, 
as $j\To+\infty$. Thus, $\Pi^{(q)}-B^{(q)}\Pi^{(q)}$ 
can be extended continuously to $u$ and 
$(\Pi^{(q)}-B^{(q)}\Pi^{(q)})u\in L^2_{(0,q)}(M)$. 
We conclude that $\Pi^{(q)}-B^{(q)}\Pi^{(q)}$ can be extended continuously to
\begin{equation}\label{e-gue190609yydI}
\Pi^{(q)}-B^{(q)}\Pi^{(q)}: H^{-s}_{{\rm comp\,}}(U\cap\ol M, T^{*0,q}M')
\To L^2_{(0,q)}(M),\ \ \mbox{for every $s\in\mathbb N$}.
\end{equation}
From the first two equations in \eqref{e-gue190531yyd}, we have 
\begin{equation}\label{e-gue190609yydII}
\begin{split}
&\mbox{$\Pi^{(q)}B^{(q)}u=N^{(q)}\Box^{(q)}B^{(q)}u+
\Pi^{(q)}B^{(q)}u=B^{(q)}u+r^{(q)}_1B^{(q)}u$ on $U\cap X$ 
for every $u\in L^2_{(0,q)}(M)$},\\
&\mbox{$B^{(q)}\Pi^{(q)}u=B^{(q)}\Box^{(q)}N^{(q)}u+
B^{(q)}\Pi^{(q)}u=B^{(q)}u+B^{(q)}r^{(q)}_0u$ on $U\cap X$, 
for every $u\in\Omega^{0,q}_0(U\cap\ol M)$},
\end{split}
\end{equation}
where $r^{(q)}_0, r^{(q)}_1\equiv0\mod\cC^\infty((U\times U)\cap(\ol M\times\ol M))$ 
are as in \eqref{e-gue190531yyd}. From \eqref{e-gue190609yydII}, we conclude that 
$B^{(q)}-\Pi^{(q)}B^{(q)}$ and $B^{(q)}-B^{(q)}\Pi^{(q)}$ 
can be extended continuously to 
\begin{equation}\label{e-gue190609yyda}
B^{(q)}-\Pi^{(q)}B^{(q)}: L^2_{(0,q)}(M)
\To H^s_{{\rm loc\,}}(U\cap\ol M, T^{*0,q}M'),\ \ \mbox{for every $s\in\mathbb N$}, 
\end{equation}
and 
\begin{equation}\label{e-gue190609yydb}
B^{(q)}-B^{(q)}\Pi^{(q)}: H^{-s}_{{\rm comp\,}}(U\cap\ol M, T^{*0,q}M')
\To L^2_{(0,q)}(M),\ \ \mbox{for every $s\in\mathbb N$}.
\end{equation}
Form \eqref{e-gue190609yydI} and \eqref{e-gue190609yydb}, 
we deduce that $\Pi^{(q)}-B^{(q)}$ can be extended continuously to 
\begin{equation}\label{e-gue190609ycd}
\Pi^{(q)}-B^{(q)}: H^{-s}_{{\rm comp\,}}(U\cap\ol M, T^{*0,q}M')
\To L^2_{(0,q)}(M),\ \ \mbox{for every $s\in\mathbb N$}.
\end{equation}
Since $\Pi^{(q)}:  H^{s}_{{\rm comp\,}}(U\cap\ol M, T^{*0,q}M')
\To H^{s}_{{\rm comp\,}}(U\cap\ol M, T^{*0,q}M')$ is continuous
for every $s\in\mathbb Z$, we deduce that 
$B^{(q)}$ can be extended continuously to 
\begin{equation}\label{e-gue190609ycdI}
B^{(q)}:  H^{-s}_{{\rm comp\,}}(U\cap\ol M, T^{*0,q}M')
\To H^{-s}_{{\rm loc\,}}(U\cap\ol M, T^{*0,q}M'),\ \ 
\mbox{for every $s\in\mathbb N$}.
\end{equation}
From \eqref{e-gue190609ycdI}, we deduce that 
\[r^{(q)}_1B^{(q)}, (r^{(q)}_0)^*B^{(q)}: 
H^{-s}_{{\rm comp\,}}(U\cap\ol M, T^{*0,q}M')
\To H^s_{{\rm loc\,}}(U\cap\ol M, T^{*0,q}M'),\ \ 
\mbox{for every $s\in\mathbb N$},\]
where $r^{(q)}_1, r^{(q)}_0$ are as in \eqref{e-gue190531yyd} 
and $(r^{(q)}_0)^*$ is the formal adjoint of $r^{(q)}_0$ 
with respect to $(\,\cdot\,|\,\cdot\,)_M$. Hence, 
\begin{equation}\label{e-gue190609ycdh}
r^{(q)}_1B^{(q)}, (r^{(q)}_0)^*B^{(q)}\equiv0
\mod\cC^\infty((U\times U)\cap(\ol M\times\ol M)).
\end{equation}
By taking adjoint of $(r^{(q)}_0)^*B^{(q)}$, we get 
\begin{equation}\label{e-gue190609ycdi}
B^{(q)}r^{(q)}_0\equiv0\mod\cC^\infty((U\times U)\cap(\ol M\times\ol M)). 
\end{equation}
From \eqref{e-gue190609ycdh}, \eqref{e-gue190609ycdi} 
and \eqref{e-gue190609yydII}, we get 
\begin{equation}\label{e-gue190609yydh}
\begin{split}
&\mbox{$\Pi^{(q)}B^{(q)}u-B^{(q)}u=f^{(q)}_1u$ on 
$U\cap X$, for every $u\in L^2_{(0,q)}(M)$},\\
&\mbox{$B^{(q)}\Pi^{(q)}u-B^{(q)}u=f^{(q)}_2u$ 
on $U\cap X$, for every $u\in\Omega^{0,q}_0(U\cap\ol M)$},
\end{split}
\end{equation}
where $f^{(q)}_1: L^2_{(0,q)}(M)\To\Omega^{0,q}(U\cap\ol M)$, 
$f^{(q)}_1\equiv0\mod\cC^\infty((U\times U)\cap(\ol M\times\ol M))$, 
$f^{(q)}_2: \Omega^{0,q}_0(U\cap\ol M)\To L^2_{(0,q)}(M)$, 
$f^{(q)}_2\equiv0\mod\cC^\infty((U\times U)\cap(\ol M\times\ol M))$. 
Taking adjoint in \eqref{e-gue190609ycd}, we conclude that 
$(\Pi^{(q)})^*-B^{(q)}$ can be extended continuously to 
\begin{equation}\label{e-gue190609syd}
(\Pi^{(q)})^*-B^{(q)}: L^2_{(0,q)}(M)
\To H^s_{{\rm loc\,}}(U\cap\ol M, T^{*0,q}M'),\ \ 
\mbox{for every $s\in\mathbb N$}, 
\end{equation}
where $(\Pi^{(q)})^*$ is the formal adjoint of $\Pi^{(q)}$ 
with respect to $(\,\cdot\,|\,\cdot\,)_M$. From \eqref{e-gue190419scd}, we see that 
\[(\Pi^{(q)})^*=\Pi^{(q)}+\Gamma^{(q)}_1,\]
where $\Gamma^{(q)}_1\equiv0\mod\cC^\infty((U\times U)\cap(\ol M\times\ol M))$. 
From this observation and \eqref{e-gue190609syd}, we deduce that 
$\Pi^{(q)}-B^{(q)}$ can be extended continuously to 
\begin{equation}\label{e-gue190609sydI}
\Pi^{(q)}-B^{(q)}: L^2_{(0,q)}(M)
\To H^s_{{\rm loc\,}}(U\cap\ol M, T^{*0,q}M'),\ \ 
\mbox{for every $s\in\mathbb N$}. 
\end{equation}
From \eqref{e-gue190609ycd} and \eqref{e-gue190609sydI}, we get 
\[(\Pi^{(q)}-B^{(q)})(\Pi^{(q)}-B^{(q)}): 
H^{-s}_{{\rm comp\,}}(U\cap\ol M, T^{*0,q}M')
\To H^s_{{\rm loc\,}}(U\cap\ol M, T^{*0,q}M')\]
is continuous, for every $s\in\mathbb N$. Hence, 
\begin{equation}\label{e-gue190609sydII}
(\Pi^{(q)}-B^{(q)})(\Pi^{(q)}-B^{(q)})\equiv0
\mod\cC^\infty((U\times U)\cap(\ol M\times\ol M)). 
\end{equation}
On the other hand, we have 
\begin{equation}\label{e-gue190609sydIII}
\begin{split}
&(\Pi^{(q)}-B^{(q)})(\Pi^{(q)}-B^{(q)})u\\
&=(\Pi^{(q)})^2u-\Pi^{(q)}B^{(q)}u-B^{(q)}\Pi^{(q)}u+(B^{(q)})^2u\\
&=\Pi^{(q)}u-B^{(q)}u-B^{(q)}u+B^{(q)}u+((\Pi^{(q)})^2-\Pi^{(q)})u\\
&\quad+(B^{(q)}-\Pi^{(q)}B^{(q)})u+(B^{(q)}-B^{(q)}\Pi^{(q)})u\\
&=\Pi^{(q)}u-B^{(q)}u+r^{(q)}_6-f^{(q)}_1u-f^{(q)}_2u,\ \ 
\mbox{for every $u\in\Omega^{0,q}_0(U\cap\ol M)$}, 
\end{split}
\end{equation}
where $r^{(q)}_6\equiv0\mod\cC^\infty((U\times U)\cap(\ol M\times\ol M))$ 
is as in \eqref{e-gue190531yyd}, 
$f^{(q)}_1, f^{(q)}_2\equiv0\mod\cC^\infty((U\times U)\cap(\ol M\times\ol M))$ 
are as in \eqref{e-gue190609yydh}. From \eqref{e-gue190609sydII} and 
\eqref{e-gue190609sydIII}, the theorem follows. 
\end{proof}

\section{Proof of Theorem \ref{t-gue190909yyd}}
To prove Theorem~\ref{t-gue190905ycd}, we need a result of 
Takegoshi~\cite{Takegoshi83}, which is a generalization
of \cite{Ko73}. Consider an open relatively compact subset 
$M_0:=\set{z\in M';\, \rho(z)<0}$ with smooth boundary 
$X_0$ of $M'$.
We have the following (see~\cite[Section 3, Theorem N]{Takegoshi83}).
\begin{thm}\label{t-gue190905ycdq}
Let $M_0$ be a pseudoconvex domain with smooth boundary $X_0$
in a complex manifold $M'$ and let $L$ be a holomorphic line bundle
on $M'$ which is positive on a neighborhood of $X_0$.
Then there exists $k_0\in\mathbb N$, such that the following statement 
holds for every $k\in\mathbb N$, $k\geq k_0$: 
For every $f\in L^2_{(0,1)}(M_0,L^k)$ with $\ddbar f=0$ on $M_0$ 
there exists $g\in L^2(M_0,L^k)$ such that $\ddbar g=f$ on $M_0$ and 
\begin{equation}\label{e-gue190905yydj}
\int_{M_0}\abs{g}^2_{h^{L^k}}dv_{M'}\leq 
C_k\int_{M_0}\abs{f}^2_{h^{L^k}}dv_{M'},
\end{equation}
where $C_k>0$ is a constant independent of $f$ and 
$g$ and $\abs{\cdot}_{h^{L^k}}$ denotes the pointwise norm 
on $\oplus^n_{q=0}T^{*0,q}M'\otimes L^k$ induced by the given 
Hermitian metric $\langle\,\cdot\,|\,\cdot\,\rangle$ on $\Complex TM'$ and $h^L$.
\end{thm}
\begin{proof}[Proof of Theorem~\ref{t-gue190905ycd}]
Let $k_0\in\mathbb N$ be as in Theorem~\ref{t-gue190905ycdq}. 
Let $k\geq k_0$, $k\in\mathbb N$ and let $U$ be any open set of 
$X_0$ with $U\cap X_1=\emptyset$. 
Let $u\in\cC^\infty_0(U\cap\ol M,L^k)\cap{\rm Dom\,}\Box^{(0)}_k$ 
and let $f:=\ddbar u\in\Omega^{0,1}_0(U\cap M,L^k)\subset 
L^2_{(0,1)}(M_0,L^k)$. From Theorem~\ref{t-gue190905ycdq}, we see 
that there is a $g\in L^2(M_0,L^k)\subset L^2(M,L^k)$ such that 
$\ddbar g=\ddbar u$ on $M_0$ (hence on $M$) and 
\begin{equation}\label{e-gue190905ycdz}
\int_{M_0}\abs{g}^2_{h^{L^k}}dv_{M'}\leq 
C_k\int_{M_0}\abs{\ddbar u}^2_{h^{L^k}}dv_{M'},
\end{equation}
where $C_k>0$ is a constant independent of $u$ and $g$. 
Since $(I-B^{(0)}_k)u$ is the solution of $\ddbar g=\ddbar u$ on $M$
of minimal $L^2$ norm, we have 
\begin{equation}\label{e-gue190905ycdm}
\int_{M}\abs{(I-B^{(0)}_k)u}^2_{h^{L^k}}dv_{M'}\leq
\int_{M}\abs{g}^2_{h^{L^k}}dv_{M'}.
\end{equation}
From \eqref{e-gue190905ycdz} and \eqref{e-gue190905ycdm}, we get 
\begin{equation}\label{e-gue190905ycdn}
\int_{M}\abs{(I-B^{(0)}_k)u}^2_{h^{L^k}}dv_{M'}\leq 
C_k\int_{M_0}\abs{\ddbar u}^2_{h^{L^k}}dv_{M'}.
\end{equation}
Since $\ddbar u$ has compact support in $U\cap\ol M$, we have 
\begin{equation}\label{e-gue190905ycdp}
\int_{M_0}\abs{\ddbar u}^2_{h^{L^k}}dv_{M'}=
\int_{M}\abs{\ddbar u}^2_{h^{L^k}}dv_{M'}.
\end{equation}
From \eqref{e-gue190905ycdn} and \eqref{e-gue190905ycdp}, we get 
\begin{equation}\label{e-gue190911yyd}
\int_{M}\abs{(I-B^{(0)}_k)u}^2_{h^{L^k}}dv_{M'}\leq 
C_k\int_{M}\abs{\ddbar u}^2_{h^{L^k}}dv_{M'}.
\end{equation}
Since $u\in{\rm Dom\,}\Box^{(0)}_k$, we can check that 
\begin{equation}\label{e-gue190909yyd}
\begin{split}
&\int_M\abs{\ddbar u}^2_{h^{L^k}}dv_{M'}=
(\,\ddbar u\,|\,\ddbar u\,)_k=(\,\ddbar u\,|\,\ddbar(I-B^{(0)}_k)u\,)_k\\
&=(\,\Box^{(0)}_ku\,|\,(I-B^{(0)}_k)u\,)_k\leq
\left\|\Box^{(0)}_ku\right\|_k\norm{(I-B^{(0)}_k)u}_k.
\end{split}
\end{equation}
The theorem follows from \eqref{e-gue190911yyd} and \eqref{e-gue190909yyd}. 
\end{proof}

From Theorem~\ref{t-gue190905ycd}, Theorem~\ref{t-gue190709yyd} 
and Remark~\ref{r-gue190810yyd}, we immediately get Theorem \ref{t-gue190909yyd}.

\section{$S^1$-equivariant Bergman kernel 
asymptotics and embedding theorems}\label{s-gue190810yyd}

In this section, we assume that $M'$ admits a holomorphic $S^{1}$-action 
$e^{i\theta}$, $\theta\in[0,2\pi[$, $e^{i\theta}: M'\to M'$, 
$x\in M'\To e^{i\theta}\circ x\in M'$. 
Recall that $X_0$ is an open connected component of $X$ 
such that \eqref{e-gue190812yydI} holds and we work 
with Assumption~\ref{a-gue190812yyd}. 

\begin{thm}\label{t-gue190813yyd}
With the notations and assumptions used in the discussion before 
Theorem~\ref{t-gue190817yydhi}, fix $p\in X_0$ and let $U$ 
be an open set of $p$ in $M'$ with $U\cap X_0\neq\emptyset$. 
Suppose that the Levi form is non-degenerate of constant signature 
$(n_-, n_+)$ on $U\cap X_0$, where $n_-$ denotes the number 
of the negative eigenvalues of the Levi form on $U\cap X_0$. Fix $\lambda>0$. 
If $q\neq n_-$, then 
\begin{equation}\label{e-gue190813ycd}
B^{(q)}_{\leq\lambda,m}\equiv0\mod O(m^{-\infty})\ \ \mbox{on $U\cap\ol M$}.
\end{equation}

Let $q=n_-$. Let $N_p:=\set{g\in S^1;\, g\circ p=p}=
\set{g_0:=e, g_1,\ldots, g_r}$, where $e$ denotes the 
identify element in $S^1$ and $g_j\neq g_\ell$ if $j\neq \ell$, 
for every $j, \ell=0,1,\ldots,r$. We have 
\begin{equation}\label{e-gue190813ycdI}
B^{(q)}_{\leq\lambda,m}(x,y)\equiv
\sum^r_{\alpha=0}g^m_\alpha e^{im\phi(x,g_\alpha y)}b_{\alpha}(x,y,m)
\mod O(m^{-\infty})\ \ \mbox{on $U\cap\ol M$}, 
\end{equation}
where for every $\alpha=0,1,\ldots,r$, 
\begin{equation}\label{e-gue190816ycd}
\begin{split}
&b_\alpha(x, y, m)\in S^{n}_{{\rm loc\,}}((U\times U)\cap
(\ol M\times\ol M),T^{*0,q}M'\boxtimes(T^{*0,q}M')^*),\\
&\mbox{$b_\alpha(x, y, m)\sim
\sum^\infty_{j=0}b_{\alpha,j}(x, y)m^{n-j}$ in 
$S^{n}_{{\rm loc\,}}((U\times U)\cap(\ol M\times\ol M),
T^{*0,q}M'\boxtimes(T^{*0,q}M')^*)$},\\
&b_{\alpha,j}(x, y)\in\cC^\infty((U\times U)\cap(\ol M\times\ol M),
T^{*0,q}M'\boxtimes(T^{*0,q}M')^*),\ \ j=0,1,\ldots,\\
&b_{\alpha,0}(x,x)=b_0(x,x),\ \ \mbox{$b_0(x,x)$ is given by 
\eqref{e-gue190531yyda}}, 
\end{split}
\end{equation}
and $\phi(x, y)\in\cC^\infty((U\times U)\cap(\ol M\times\ol M))$ 
is as in \eqref{e-gue190708ycdII}. 
\end{thm} 

\begin{proof}
From \eqref{e-gue190810yyd} and \eqref{e-gue190812ycdI}, 
we can integrate by parts in $\theta$ and get \eqref{e-gue190813ycd}. 
We now prove \eqref{e-gue190813ycdI}. From Theorem~\ref{t-gue190606ycd} 
and \eqref{e-gue190529ycd}, it is straightforward to see that 
\begin{equation}\label{e-gue190816yyd}
B^{(q)}_{\leq\lambda}\equiv\tilde PS_{-,m}L^{(q)}\mod O(m^{-\infty})\ \ 
\mbox{on $U\cap\ol M$},
\end{equation}
where $S_{-,m}(x,y)=\frac{1}{2\pi}\int^{\pi}_{-\pi}S_-(x,e^{i\theta}y)
e^{im\theta}d\theta$ and $S_-(x,y)$ is as in Theorem~\ref{t-gue190527syd}.  
From Theorem~\ref{t-gue190527syd}, we can repeat the 
proof of~\cite[Theorem 3.12]{HM14} with minor change and deduce that 
\begin{equation}\label{e-gue190816yydI}
S_{-,m}(x,y)\equiv\sum^r_{\alpha=0}g^m_\alpha e^{im\varphi_-(x,g_\alpha y)}
a_{\alpha}(x,y,m)\mod O(m^{-\infty})\ \ \mbox{on $U\cap X$}, 
\end{equation}
where for every $\alpha=0,1,\ldots,r$, 
\begin{equation}\label{e-gue190816yydII}
\begin{split}
&a_\alpha(x, y, m)\in S^{n-1}_{{\rm loc\,}}((U\times U)\cap(X\times X),
T^{*0,q}M'\boxtimes(T^{*0,q}M')^*),\\
&\mbox{$a_\alpha(x, y, m)\sim\sum^\infty_{j=0}a_{\alpha,j}(x, y)m^{n-1-j}$ 
in $S^{n-1}_{{\rm loc\,}}((U\times U)\cap(X\times X),
T^{*0,q}M'\boxtimes(T^{*0,q}M')^*)$},\\
&a_{\alpha,j}(x, y)\in\cC^\infty((U\times U)\cap(\ol M\times\ol M),
T^{*0,q}M'\boxtimes(T^{*0,q}M')^*),\ \ j=0,1,\ldots,\\
&a_{\alpha,0}(x,x)=a_0(x,x),\ \ \mbox{$a_0(x,x)$ is given by \eqref{e-gue190528yyda}},
\end{split}
\end{equation}
and $\varphi_-(x, y)\in\cC^\infty((U\times U)\cap(\ol M\times\ol M))$ 
is as in Theorem~\ref{t-gue190527syd}. 
From \eqref{e-gue190816yydI}, we can repeat the procedure in the 
proof of~\cite[Proposition 7.8 of Part II]{Hsiao08} and deduce that the distribution kernel of 
$\tilde PS_{-,m}L^{(q)}$ is of the form \eqref{e-gue190813ycdI}.
\end{proof}

From Theorem~\ref{t-gue190709yyd}, we can repeat the proof of 
Theorem~\ref{t-gue190813yyd} and deduce 

\begin{thm}\label{t-gue190816yyydh}
With the notations and assumptions used in the discussion before 
Theorem~\ref{t-gue190817yydhi}, fix $p\in X_0$ and let $U$ be 
an open set of $p$ in $M'$ with $U\cap X_0\neq\emptyset$. 
Suppose that the Levi form is non-degenerate of constant signature 
$(n_-, n_+)$ on $U\cap X_0$, where $n_-$ denotes the number 
of the negative eigenvalues of the Levi form on $U\cap X_0$. 
Suppose that $\Box^{(q)}$ has local closed range in $U$. If $q\neq n_-$, then 
\begin{equation}\label{e-gue190816yydj}
B^{(q)}_{m}\equiv0\mod O(m^{-\infty})\ \ \mbox{on $U\cap\ol M$}.
\end{equation}

Let $q=n_-$. Let $N_p:=\set{g\in S^1;\, g\circ p=p}=
\set{g_0:=e, g_1,\ldots, g_r}$, where $e$ denotes the identify 
element in $S^1$ and $g_j\neq g_\ell$ if $j\neq \ell$, for every 
$j, \ell=0,1,\ldots,r$. We have 
\begin{equation}\label{e-gue190816yydk}
B^{(q)}_{m}(x,y)\equiv\sum^r_{\alpha=0}g^m_\alpha e^{im\phi(x,g_\alpha y)}
b_{\alpha}(x,y,m)\mod O(m^{-\infty})\ \ \mbox{on $U\cap\ol M$}, 
\end{equation}
where $b_{\alpha}(x,y,m)\in S^{n}_{{\rm loc\,}}((U\times U)
\cap(\ol M\times\ol M),T^{*0,q}M'\boxtimes(T^{*0,q}M')^*)$, 
$\alpha=0,1,\ldots,r$, and $\phi(x,y)\in\cC^\infty((U\times U)
\cap(\ol M\times\ol M))$ are as in Theorem~\ref{t-gue190813yyd}. 
\end{thm} 

Consider $q=0$. When $Z(1)$ holds on $X$, it is well-known 
(see Folland-Kohn~\cite{FK72}) that $\Box^{(0)}$ has $L^2$ closed range. 
From this observation and Theorem~\ref{t-gue190816yyydh}, we deduce that 

\begin{thm}\label{t-gue190817yydh}
With the notations and assumptions used in the discussion 
before Theorem~\ref{t-gue190817yydhi}, fix $p\in X_0$ and let $U$ 
be an open set of $p$ in $M'$ with $U\cap X_0\neq\emptyset$. 
Suppose that the Levi form is 
positive $U\cap X_0$. Suppose that $Z(1)$ holds on $X$. 
Let $N_p:=\set{g\in S^1;\, g\circ p=p}=\set{g_0:=e, g_1,\ldots, g_r}$, 
where $e$ denotes the identify element in $S^1$ and $g_j\neq g_\ell$ 
if $j\neq \ell$, for every $j, \ell=0,1,\ldots,r$. We have 
\begin{equation}\label{e-gue190816yydl}
B^{(0)}_{m}(x,y)\equiv\sum^r_{\alpha=0}g^m_\alpha e^{im\phi(x,g_\alpha y)}
b_{\alpha}(x,y,m)\mod O(m^{-\infty})\ \ \mbox{on $U\cap\ol M$}, 
\end{equation}
where $b_{\alpha}(x,y,m)\in S^{n}_{{\rm loc\,}}((U\times U)
\cap(\ol M\times\ol M))$, $\alpha=0,1,\ldots,r$, and 
$\phi(x,y)\in\cC^\infty((U\times U)\cap(\ol M\times\ol M))$ 
are as in Theorem~\ref{t-gue190813yyd}. 
\end{thm} 

For every $m\in\mathbb N$, let 
\begin{equation}\label{e-gue190817ycd}
\begin{split}
\Phi_m: \ol M&\To\Complex^{d_m},\\
x&\To (f_1(x),\ldots,f_{d_m}(x)),
\end{split}
\end{equation}
where $\set{f_1(x), \ldots,f_{d_m}(x)}$ is an orthonormal basis for 
$H^0_m(\ol M)$ with respect to $(\,\cdot\,|\,\cdot\,)_M$ and 
$d_m={\rm dim\,}H^0_m(\ol M)$. 
We have the following $S^1$-equivaraint embedding theorem.

\begin{thm}\label{t-gue190817yydI}
With the notations and assumptions used in the discussion 
before Theorem~\ref{t-gue190817yydhi}, assume that the 
Levi form is positive on $X_0$ and $Z(1)$ holds on $X$. 
For every $m_0\in\mathbb N$, there exist 
$m_1\in \mathbb N, \ldots, m_k\in\mathbb N$, with 
$m_j\geq m_0$, $j=1,\ldots,k$, and a $S^1$-invariant 
open set $V$ of $X_0$ such that the map 
\begin{equation}\label{e-gue190817ycdI}
\begin{split}
\Phi_{m_1,\ldots,m_k}: V\cap\ol M&\To\Complex^{\hat d_m},\\
x&\To (\Phi_{m_1}(x),\ldots,\Phi_{m_k}(x)),
\end{split}
\end{equation}
is a holomorphic embedding, where $\Phi_{m_j}$ is given 
by \eqref{e-gue190817ycd} and $\hat d_m=d_{m_1}+\cdots+d_{m_k}$. 
\end{thm}

\begin{proof}
Fix $m_0\in\mathbb N$. 
By using Theorem~\ref{t-gue190817yydh}, we can repeat 
the proof of~\cite[Theorem 1.2]{HHL18} with minor change and conclude that 
we can find $m_1\in \mathbb N, \ldots, m_k\in\mathbb N$, 
with $m_j\geq m_0$, $j=1,\ldots,k$, such that 
\begin{equation}\label{e-gue190904yyd}
\mbox{$\Phi_{m_1,\ldots,m_k}: X_0\To\Complex^{\hat d_m}$ is an embedding}
\end{equation}
and there is a $S^1$-invariant open set $U$ of $X_0$ such that  
\begin{equation}\label{e-gue190904yydI}
\mbox{$\Phi_{m_1,\ldots,m_k}: U\cap\ol M\To\Complex^{\hat d_m}$ is an immersion}.
\end{equation}
Fix $x_0\in X_0$. From \eqref{e-gue190904yydI}, 
it is straightforward to see that there are $S^1$-invariant open 
sets $\Omega_{x_0}\Subset W_{x_0}\Subset U_{x_0}$ of $x_0$ in $M'$ such that 
\begin{equation}\label{e-gue190904yydII}
\mbox{$\Phi_{m_1,\ldots,m_k}: 
U_{x_0}\cap\ol M\To\Complex^{\hat d_m}$ is injective}.
\end{equation}
Let 
\begin{equation}\label{e-gue190904ycda}
\delta_{x_0}:=\inf\set{\abs{\Phi_{m_1,\ldots,m_k}(x)-
\Phi_{m_1,\ldots,m_k}(y)};\, x\in \Omega_{x_0}\cap X_0, 
y\in X_0, y\notin W_{X_0}\cap X_0}.
\end{equation}
From \eqref{e-gue190904yyd}, we see that $\delta_{x_0}>0$. 
Let $V^{x_0}$ be a small $S^1$-invariant open set of $X_0$ in $M'$ such that 
for every $x\in V^{x_0}\cap\ol M$, $x\notin U_{x_0}$, 
there is a $y\in X_0$, $y\notin W_{X_0}\cap X_0$, such that 
\begin{equation}\label{e-gue190904ycd}
\abs{\Phi_{m_1,\ldots,m_k}(x)-\Phi_{m_1,\ldots,m_k}(y)}\leq\frac{\delta_{x_0}}{2}. 
\end{equation}
Assume that $X_0=\bigcup^N_{j=1}\Bigr(\Omega_{x_j}\cap X_0\Bigr)$, 
$N\in\mathbb N$, and let 
\[V:=U\cap\Bigr(\cap^N_{j=1}V^{x_j}\Bigr)\cap
\Bigr(\bigcup^N_{j=1}\Omega_{x_j}\Bigr),\] 
where $\Omega_{x_j}$, $V^{x_j}$, $j=1,\ldots,N$, 
are as above, and $U$ is as in \eqref{e-gue190904yydI}. 
From \eqref{e-gue190904yydI}, we see that 
$\Phi_{m_1,\ldots,m_k}: V\cap\ol M\To\Complex^{\hat d_m}$
is an immersion. We claim that 
$\Phi_{m_1,\ldots,m_k}: V\cap\ol M\To\Complex^{\hat d_m}$ 
is injective. Let $p, q\in V\cap\ol M$, $p\neq q$. 
We are going to prove that 
$\Phi_{m_1,\ldots,m_k}(p)\neq\Phi_{m_1,\ldots,m_k}(q)$. 
We may assume that $p\in\Omega_{x_1}\cap\ol M$. 
If $q\in U_{x_1}$. From \eqref{e-gue190904yydII}, 
we see that $\Phi_{m_1,\ldots,m_k}(p)\neq \Phi_{m_1,\ldots,m_k}(q)$. 
Assume that $q\notin U_{x_1}$. From the discussion before 
\eqref{e-gue190904ycd}, we see that there is 
$y_0\in X_0$, $y_0\notin W_{x_1}\cap X_0$ such that 
\begin{equation}\label{e-gue190904ycdb}
\abs{\Phi_{m_1,\ldots,m_k}(p)-\Phi_{m_1,\ldots,m_k}(y_0)}\leq\frac{\delta_{x_1}}{2}. 
\end{equation}
From \eqref{e-gue190904ycdb} and \eqref{e-gue190904ycda}, we have 
\[\begin{split}
&\abs{\Phi_{m_1,\ldots,m_k}(p)-\Phi_{m_1,\ldots,m_k}(q)}\\
&\geq \abs{\Phi_{m_1,\ldots,m_k}(p)-\Phi_{m_1,\ldots,m_k}(y_0)}-
\abs{\Phi_{m_1,\ldots,m_k}(y_0)-\Phi_{m_1,\ldots,m_k}(q)}\\
&\geq\delta_{x_1}-\frac{\delta_{x_1}}{2}>0.
\end{split}\]
Hence, $\Phi_{m_1,\ldots,m_k}(p)\neq \Phi_{m_1,\ldots,m_k}(q)$. 
We have proved that $\Phi_{m_1,\ldots,m_k}: V\cap\ol M\To\Complex^{\hat d_m}$ 
is injective. 
The theorem follows. 
\end{proof}

Without $Z(1)$ condition, we can still have the following 
$S^1$-equivaraint embedding theorem

\begin{thm}\label{t-gue190817yydII}
With the notations and assumptions used in the discussion 
before Theorem~\ref{t-gue190817yydhi}, assume that the Levi form 
is positive on $X_0$.  For every $m_0\in\mathbb N$, there exist a 
$S^1$-invariant open set $V$ of $X_0$ and 
$f_j(x)\in\cC^\infty(V\cap\ol M)$, $j=1,\ldots,K$, 
with $\ddbar f_j=0$ on $V\cap\ol M$, $f_j(e^{i\theta}x)=
e^{im_j\theta}f(x)$, $m_j\geq m_0$, $j=1,\ldots,K$, 
for every $e^{i\theta}\in S^1$ and every $x\in V$,  such that the map 
\begin{equation}\label{e-gue190817ycdII}
\begin{split}
\Phi: V\cap\ol M&\To\Complex^{K},\\
x&\To (f_1(x),\ldots,f_K(x)),
\end{split}
\end{equation}
is a holomorphic embedding.
\end{thm}

\begin{proof}
We may assume that $X_0=\set{x\in M';\, \rho(x)=0}$. Consider the Corona domain 
\[\hat M:=\set{x\in M';\, -\varepsilon<\rho(x)<0},\]
where $\varepsilon>0$ is a small constant. 
Then $\hat M$ is a complex manifold with smooth boundary 
$\hat X$. Moreover, it is easy to see that $X_0$ is an open connected 
component of $\hat X$ and $Z(1)$ holds on $\hat X$. Hence, 
we can apply Theorem~\ref{t-gue190817yydI} to get Theorem~\ref{t-gue190817yydII}. 
\end{proof}

\bibliographystyle{plain}

\end{document}